\documentclass[10pt]{amsart}
\usepackage{import}
\usepackage{./preamble/Pakete}
\usepackage{./preamble/Befehle}

\begin{document}


\title[Graded and Geometric Parabolic Induction]{Graded and Geometric Parabolic Induction for Category $\mathcal{O}$}

\author{Jens Niklas Eberhardt}
\address{Department of Mathematics, University of California Los Angeles\\520 Portola Plaza, Los Angeles, CA 90095
}
\email{jneberhardt@gmail.com}  
\begin{abstract}
We prove that the parabolic induction functor on BGG-category $\mathcal{O}$ associated to a complex reductive Lie algebra is \emph{gradable}, that is, lifts to graded category $\mathcal{O}$ as constructed by Beilinson--Ginzburg--Soergel.
Graded category $\mathcal{O}$ is equivalent to a category of stratified mixed Tate motives on a corresponding flag variety as recently defined by Soergel--Wendt. 
The graded version of parabolic induction is induced by a \emph{geometric parabolic induction} functor we construct on the level of stratified mixed Tate motives.

We also describe the effect of parabolic induction on the level of Soergel modules.
\end{abstract}
\maketitle
\setcounter{tocdepth}{1} 
\tableofcontents

\section{Introduction}
\subsection{Graded Parabolic Induction}
Let $\g\supset\borel\supset\cartan$ be a complex reductive Lie algebra with a Borel and Cartan subalgebra. Fix a parabolic subalgebra $\g\supset\para\supset\borel$ and denote its reductive Levi factor by $\para\twoheadrightarrow\levi$. Denote by $\Weyl\supset\Wi$ the Weyl groups of $\g$ and $\levi$.

The goal of this article is to construct a \emph{graded} and \emph{geometric} version of \emph{parabolic induction} for modules in the BGG-category $\mathcal O$:
$$\ind: \cato{\levi}{}\rightarrow \cato{\g}{},\, M\mapsto \Univ{\g}\otimes_{\Univ{\para}}\res_{\levi}^\para M.$$
 We will, amongst other things, prove:
\begin{theorem*}[Theorem \ref{thm:paraindgradable}]
	Let $\lambda\in \cartan^*$ be a dominant integral weight and $w$ a shortest coset representative in $\Wi\backslash\Weyl$. There is a functor $\grind$ making the following diagram commute (up to natural isomorphism)
	\begin{center}
		\begin{tikzcd}
			\catoz{\levi}{w\cdot\lambda}\arrow[r,"\grind"]\arrow[d,"v"] &\catoz{\g}{\lambda}\arrow[d,"v"]\\
			\cato{\levi}{w\cdot\lambda}\arrow[r,"\ind"] &\cato{\g}{\lambda}
		\end{tikzcd}
	\end{center}
	and fulfilling $\grind\langle n\rangle\cong\langle n\rangle \grind$, where $\langle-\rangle$ is the shift of grading.
\end{theorem*}
Here $v:\mathcal O^\Z\rightarrow \mathcal O$ denotes the \emph{graded category $\mathcal O$} ,
as constructed in \cite{Soe90} and \cite{BGS}.
In the words of \cite{Str}, where similar questions for translation functors are discussed, this means that parabolic induction (at least for integral blocks) is \emph{gradable}.
We construct $\grind$ with geometric methods, which we will explain now.
\subsection{Geometric Parabolic Induction}
As envisioned in \cite{BG}, the grading of category $\mathcal O$ is deeply related to the \emph{mixed geometry} of flag varieties: There should be a derived equivalence between each block of $\mathcal O^\Z$ and a category of \emph{mixed sheaves} on an associated flag variety. 
This vision was realized in \cite{Soe90}, \cite{BGS} and finally \cite{SoeWe}, where an equivalence of categories (up to adding a root of the Tate twist)
$$\MTDer{(B)}{G/Q}\stackrel{\sim}{\rightarrow}\Derb(\mathcal{O}_\lambda^\Z(\g))$$
between \emph{stratified mixed Tate motives} on a (partial) flag variety $G/Q$ for the Langlands dual algebraic group $G/\C$ and a derived (singular) block $\mathcal{O}_\lambda^\Z(\g)$ was constructed. This equivalence is indeed a form of \emph{Koszul duality}: the perverse t-structure on the left hand side corresponds to the Koszul dual t-structure on the right, and vice versa.

Stratified mixed Tate motives are certain constructible motivic sheaves. They behave similarly to mixed $\ell$-adic sheaves and mixed Hodge modules (with the advantage that they have no extensions between Tate motives $\C(n)$). In particular, they are equipped with a full six functor formalism, which we can use to construct a geometric version of parabolic induction as follows.
\begin{theorem*}[Theorem \ref{thm:main} and \ref{thm:paraindgradable}]
	Let $\lambda\in\cartan^*$ be a regular dominant integral weight and $w$ be a shortest coset representative in $\Wi\backslash\Weyl$. Then the following diagram commutes up to natural isomorphism
	\begin{center}
		\begin{tikzcd}[column sep=2cm,row sep=.5cm]
			\MTDer{(B)}{P/B}\arrow[r,"\gindv{w}","\gind{w}"']\arrow[d]\arrow[dd,"v"',bend right=72]&\MTDer{(B)}{G/B}\arrow[d]\arrow[dd,"v",bend left=72]\\
			\Derb(\catoz{\levi}{w\cdot\lambda})\arrow[d,"v"] &\Derb(\catoz{\g}{\lambda})\arrow[d,"v"] \\
			\Derb(\cato{\levi}{w\cdot\lambda})\arrow[r,"\ind"] &\Derb(\cato{\g}{\lambda})
		\end{tikzcd}
	\end{center}
	Here $G\supset P\supset B$ corresponds to $\g\supset\para\supset\borel$ and the functor $\gind{w}=\gindv{w}$, which we call \textbf{geometric parabolic induction}, is defined via maps
	\begin{center}
		\begin{tikzcd} 
			P/B& 
			\arrow[two heads,l,"\pr{w}"']
			PwB/B \arrow[hook,r,"\operatorname{h_w}"] & 
			G/B.
		\end{tikzcd}
	\end{center}
\end{theorem*}
We actually show a stronger statement which also holds for singular weights $\lambda\in\cartan^*$ and allows us to prove that parabolic induction is also gradable in this case.

\subsection{Soergel modules}
In order to prove these theorems, we use the combinatorial description of derived blocks of category $\mathcal{O}$ and stratified mixed Tate motives on flag varieties in terms of the homotopy category of \emph{Soergel modules}. Let $\Coi=H^*(G/B,\C)\rightarrow \Coi'=H^*(P/B,\C)$ be the cohomology rings of the flag varieties $G/B\supset P/B$. Then for a reduced expression $w=s_n\dots s_1\in \Weyl$ define the following complex of Soergel bimodules (it is in fact an instance of a \emph{Rouquier complex}) over $\Coi$
$$\underbar{R}_{\underline{w}}\defi \Rouqs{s_1}\otimes_{\Coi}\cdots\otimes_{\Coi} \Rouqs{s_n},\text{where}$$
$$\Rouqs{s}\defi\dots\rightarrow0\rightarrow \Coi\rightarrow \Coi\otimes_{\Coi^s}\Coi\langle 2\rangle\rightarrow0\rightarrow\dots.$$
With this notation we show:
\begin{theorem*}[Theorem \ref{thm:sindregular}, \ref{thm:sindregulargeom} and \ref{thm:main}]
	Let $\lambda\in\cartan^*$ be a dominant integral regular weight and $w$ a shortest coset representative in $\Wi\backslash\Weyl$. Then the following diagram of functors commutes (up to natural isomorphism)
	\begin{center}
		\begin{tikzcd}
			\MTDer{(B)}{P/B}\arrow[r,"\gind{w}"]\arrow[ddd,"v"', bend right=71]\arrow[d,"\wr"]&\MTDer{(B)}{G/B}\arrow[d,"\wr"']\arrow[ddd,"v", bend left=71]\\
			\Hotb(\Coi'\Smodules^{\Z,ev})\arrow[r,"\sind{w}{}"]\arrow[d,"v"']&\Hotb(C\Smodules^{\Z,ev})\arrow[d,"v"]\\
			\Hotb(\Coi'\Smodules)\arrow[r,"\sind{w}{}"]\arrow[d,"\wr"]&\Hotb(\Coi\Smodules)\arrow[d,"\wr"']\\
			\Derb(\cato{\levi}{w\cdot\lambda})\arrow[r,"\ind"]&\Derb(\cato{\g}{\lambda}).
		\end{tikzcd}
	\end{center}
	Here by $\operatorname{SMod}^{(\Z,ev)}$ we denote the categories of (evenly graded) Soergel modules and
	$$\sind{w}{}: \Coi'\Smodules^{\Z,ev}\rightarrow \Coi\Smodules^{\Z,ev}, M\mapsto \underbar{R}_{\underline{w}}\otimes_{\Coi}\res_{\Coi'}^{\Coi}M.$$
\end{theorem*}
Again, we prove a more general version which also applies to singular weights $\lambda\in\cartan^*$. Our proof strategy is the following: Firstly, we show the statement for $w=e$, which is the easiest case, since then $\ind$ maps projectives to projectives and $\gind{e}$ is weight exact. We  then carefully analyse how (geometric) parabolic induction interacts with (geometric) \emph{wall crossing} functors. Comparing the results, we are able to prove the general case by an induction on the length of $w$.
\subsection{Conventions}
By a $\C$-algebra $A$ we always mean a (not necessarily commutative) $\C$-algebra with unit. By $A\modules$ we denote the category of \emph{finitely generated} $A$-modules. If $A=\bigoplus_{n\in\Z}A_n$ is additionally $\Z$-graded, we denote by $A\modules^\Z$ the category of graded $A$-modules and by $A\modules^{\Z,ev}$ the category of evenly graded modules, i.e. those modules which are concentrated in even degrees. For a graded module $M$ its $n$-th shift by $M\langle n\rangle$, where
$(M\langle n\rangle)^i=M^{i+n}.$

For an abelian category $\mathcal A$, we denote by $\Der(\mathcal A)$ and $\Derb(\mathcal A)$ its (bounded) derived category and by $\Proje\mathcal A$ the full additive subcategory of projective objects in $\mathcal A$.
For an additive category $\mathcal A$, we denote by $\Hot(\mathcal A)$ and $\Hotb(\mathcal A)$ its (bounded) homotopy category of chain complexes.
\subsection{Acknowledgements}
This paper is based on the authors dissertation thesis \cite{phdthesis}. 
I would like to thank my doctoral advisor Wolfgang Soergel for many instructive
discussions and posing the questions treated in this article. I also thank Simon Riche for many helpful comments on my thesis.

The author was financially supported by the DFG Graduiertenkolleg 1821 “Cohomological Methods in Geometry”.
\section{Category $\mathcal{O}$ and Parabolic Induction}
\subsection{Setup}
Let $\g\supset\borel\supset\cartan$ be a reductive complex Lie algebra together with a Borel and Cartan subalgebra. Denote by
	$$\cartan^*\supset\Phi_\g\supset\Phi^+_\g\supset\Delta_\g$$
the space of weights, set of roots, positive and simple roots corresponding to $\g\supset\borel$. By a superscript minus as in $\Phi^-_\g=\Phi_\g\backslash\Phi^+_\g$ or $\borel^-$ we always denote the corresponding negative or opposite. For a root $\alpha\in\Phi$ denote by $\alpha^\vee\in\cartan$ its coroot and by $s_\alpha$ the corresponding reflection. Let
\begin{align*}
	\Weyl&=\langle \, s_\alpha \, |\, \alpha \in \Delta_\g \, \rangle \\
	\Sim&=\left\{ \, s_\alpha \, |\, \alpha \in \Delta_\g \, \right\}
\end{align*}
be the Weyl group and set of simple reflections. Denote by $\langle-,-\rangle$ the natural evaluation pairing on $\cartan^*\otimes\cartan$ and by
\begin{align*}
	\Lambda_\g&=\left\{ \, \lambda\in\cartan^* \, |\, \langle\lambda,\alpha^\vee\rangle\in \Z\text{ for all } \alpha \in \Delta_\g \, \right\}\\
	\Lambda_\g^+&=\left\{ \, \lambda\in\cartan^* \, |\, \langle\lambda,\alpha^\vee\rangle\in \Zp \text{ for all } \alpha \in \Delta_\g \, \right\}
\end{align*}
the integral weight lattice and the set positive integral weights. 
For an integral weight $\lambda\in\Lambda_\g$ denote the unique weight in $\Weyl\lambda\cap\Lambda^+_\g$ by $\Sdom{\lambda}$.

Let $\rho\in\cartan^*$ be the half-sum of positive roots and denote by 
	$$w\cdot\lambda=w(\lambda+\rho)-\rho$$
the dot-action of $\Weyl$ on $\cartan^*$.  We denote the stabilizer of a weight $\lambda\in\cartan^*$ with respect to the dot-action by $\W_{\g,\lambda}$.

There is a partial ordering on the set of weights given by
	$$\lambda\geq\mu \stackrel{def}\Leftrightarrow{} \lambda-\mu \in \Lambda_\g^+.$$
A weight $\lambda\in\cartan^*$ is called \emph{dominant} (for $\Phi^+_\g$) if $\langle\lambda+\rho,\alpha^\vee\rangle\notin\Z_{<0}$ for all $\alpha\in\Phi^+_\g$. The set of integral dominant weights is hence $\Lambda^+_\g-\rho$.

Now let $\g\supset\para\twoheadrightarrow\levi$ be a parabolic and Levi factor of $\g$ such that $\para\supset\borel$, for simplicity we choose a splitting $\levi\subset\para$. We denote by $\nili\subset\para$ the nilpotent radical of $\para$ and by $\zi\subset\levi$ the center of $\levi$. Then we have decompositions of $\para$ and $\g$ into
\begin{align*}
	\para&=\levi\oplus\nili=\cartan\oplus\bigoplus_{\alpha\in\Phi_\levi}\g_\alpha\oplus\bigoplus_{\mathclap{\alpha\in\Phi_\g^+\backslash\Phi_\levi^+}}\g_\alpha\text{ and}\\
	\g&=\nili^-\oplus\para=\nili^-\oplus\levi\oplus\nili.
\end{align*}
For $\alpha \in \Delta$, let $\varpi_\alpha\in\cartan^*$, respectively $\varpi_\alpha^\vee\in\cartan$, be the fundamental weights; they form a dual basis to $\Delta^\vee$, respectively $\Delta$, and are well-defined if we additionally require $\varpi_\alpha(\mathfrak{z}_\g)=\{0\}$ and $\varpi_\alpha^\vee\in\left[\g,\g\right]$.
Then 
	$$\zi=\{H\in\cartan\,|\,\alpha(H)=0\text{ for all }\alpha \in \Delta_\g\backslash\Delta_\levi \}=\langle\varpi_\alpha^\vee \,|\, \alpha \in \Delta_\g\backslash\Delta_\levi\rangle_\C\oplus\mathfrak{z}_\g$$
and there is also a partial ordering on the set of $\zi$-weights, namely
	$$\nu\geq\nu' \stackrel{def}\Leftrightarrow{} \nu-\nu' \in \Zp\{\alpha|_{\zi} \,|\,\alpha\in\Delta_\g\backslash\Delta_\levi\}\text{ for }\nu,\nu'\in\zi.$$
A priori there are two different dot-actions of $\Wi$ on $\cartan^*$. They coincide since
$$w(\lambda+\rho)-\rho=w(\lambda+\rho_\levi)-\rho_\levi$$
for all $w\in\Wi$, where we use that
$$w(\rho-\rho_\levi)=\sum_{\mathclap{\alpha\in\Delta_\g\backslash\Delta_\levi}}w(\varpi_\alpha)=\sum_{\mathclap{\alpha\in\Delta_\g\backslash\Delta_\levi}}\varpi_\alpha=\rho-\rho_\levi.$$
We want to emphasize that we have two different notions of dominant weights now. A weight $\lambda\in\cartan^*$ is dominant for $\Phi_\g^+$ if $\langle\lambda+\rho,\alpha^\vee\rangle\notin\Z_{<0}$ for all $\alpha\in\Phi_\g^+$ and dominant for $\Phi_\levi^+$ if $\langle\lambda+\rho,\alpha^\vee\rangle\notin\Z_{<0}$
for all $\alpha\in\Phi_\g^+.$ The following Lemma explains how those two notions relate for integral weights. 
\begin{lemma}\label{lem:dominantsweightsgvsl}
	The weights in $\Lambda_\g$ which are dominant for $\Phi_\levi^+$ are precisely the weights of the form $w\cdot\lambda$ where $\lambda\in\Lambda_\g$ is dominant for $\Phi_\g^+$ and $w\in\W$ a shortest coset representative for $\Wi\backslash\Weyl$.
\end{lemma}
\subsection{Category $\mathcal O$}
The BGG-category $\mathcal O$ (see \cite{BGG}) associated to a complex reductive Lie algebra with fixed Borel and Cartan subalgebra $\g\supset\borel\supset\cartan$ is the full subcategory of the category of $\g$-modules, $\g\modules$, given by 
\begin{align*}
	\cato{\g}{}\defi\left\{M\in\g\modules\,\,\middle| \,\,\parbox{170pt}{$\cartan$ acts semisimply on $M$,\\$\borel$ acts locally finitely on $M$,\\$M$ is finitely generated under $\g$}\right\}
\end{align*}
For a complex Lie algebra $\mathfrak n$, denote its universal enveloping algebra by $\Univ{\mathfrak n}$. For $\lambda \in \cartan^*$ let 
	$$\Ver{\g}{\lambda}\defi\Univ{\g}\otimes_{\Univ{\borel}}\C_\lambda$$
be the Verma module with highest weight $\lambda$ and 
\begin{center}
	\begin{tikzcd}
		\ProC{\g}{\lambda}\arrow[two heads,r]&\Ver{\g}{\lambda}\arrow[two heads,r]&\Simpl{\g}{\lambda}
	\end{tikzcd}
\end{center}
its projective cover and unique simple quotient in $\cato{\g}{}$. For $\lambda$ dominant, denote by
	$$\cato{\g}{\lambda}=\langle \Ver{\g}{w\cdot\lambda} \,|\, w\in\W_{\g,\left[\lambda\right]}\rangle_{\text{Serre}}\subset\cato{\g}{}$$
the full Serre subcategory of $\cato{\g}{}$ generated by the Verma modules $\Ver{\g}{w\cdot\lambda}$, where by $\W_{\g,\left[\lambda\right]}\subset\Weyl$ we denote the integral Weyl group of $\lambda$.
Then $\cato{\g}{}$ decomposes into blocks
	$$\cato{\g}{}=\bigoplus_{\mathclap{\substack{\lambda \in \cartan^*\\\text{dominant}}}}\cato{\g}{\lambda}$$
and we denote the functor projecting on a block $\cato{\g}{\lambda}$ by $\pr{\lambda}$.
\subsection{Generalities on Parabolic Induction}
\begin{figure}[h]
	\begin{center}
		\begin{tikzpicture}
		\fill[lightlight-gray] (0,3.3) rectangle (6.5,-3.3);
		\foreach \i in {0, 1, ..., 5} {
			\fill (\i*60:3) circle [radius=2pt];
		}
		\node[right](T) at (3, 0) {$\mathbf{t}\cdot \lambda$};
		\node[left] (S) at (2*60:3) {$s\cdot \lambda$};
		\node[left] (ST) at (3*60:3) {$st\cdot \lambda$};
		\node[left] (STS) at (4*60:3) {$sts\cdot \lambda$};
		\node[right] (TS) at (5*60:3) {$\mathbf{ts}\cdot \lambda$};
		\node[right] (E) at (60:3) {$\mathbf{e}\cdot \lambda$};
		\draw[thick] (0,-3.3) -- (0,3.3);
		\draw[dashed,rounded corners=1mm] (4.3,-0.4) rectangle (-4.3,0.4);
		\node[right] (IT) at (4.5,-0.1) {$\cato{\levi}{t\cdot\lambda}$};
		\draw[dashed,rounded corners=1mm] (2.7,2.6-0.4) rectangle (-2.7,2.6+0.4);
		\node[right] (IE) at (2.9,2.6-0.1) {$\cato{\levi}{\lambda}$};
		\draw[dashed,rounded corners=1mm] (2.7,-2.6-0.4) rectangle (-2.7,-2.6+0.4);
		
		\node[right] (IE) at (2.9,-2.6-0.1) {$\cato{\levi}{ts\cdot\lambda}$};
		\end{tikzpicture}
	\end{center}\label{fig:sl2insl3}
	\caption{\textbf{The case $\sln_2\subset\sln_3$:} Here  $\Sim=\{s,t\}$, $\Simi=\{s\}$ and $\lambda\in \cartan^*$ denotes some regular integral weight. The shortest coset representatives $\{e,s,ts\}$ of $\Wi\backslash\Weyl$ parameterize the blocks of $\cato{\levi}{}$ which map into $\cato{\g}{\lambda}$.}
\end{figure}
Let $\g\supset\para\twoheadrightarrow\levi$ be a reductive complex Lie algebra with parabolic subalgebra and Levi factor. Then the \emph{parabolic induction} functor is given by
	$$\ind \defi \Univ{\g}\otimes_{\Univ{\para}}\res_\levi^\para(-):\cato{\levi}{}\rightarrow \cato{\g}{}.$$
We often drop the $\res_\levi^\para$ from the notation.
Since $\ind$ is exact and $\ind \Ver{\levi}{\mu}=\Ver{\g}{\mu}$ for all $\mu\in\cartan^*$ (see below) it respects the block decomposition of category $\cato{\g}{}$, namely restricts to
	$$\ind:\cato{\levi}{w\cdot\lambda}\rightarrow\cato{\g}{\lambda},$$
for integral $\lambda\in\cartan^*$  which are dominant for $\Phi^+_\g$ and $w\in\W$ a shortest coset representative for $\Wi\backslash\Weyl/\W_{\g,\lambda}$ or in other words
	$$(\ind)^{-1}(\cato{\lambda}{\g})=\bigoplus_{w\in\Wi\backslash\Weyl/\W_{\g,\lambda}} \cato{\levi}{w\cdot\lambda},$$
by Lemma \ref{lem:dominantsweightsgvsl}.
This is visualized in the example $\sln_{3}$ in Figure \ref{fig:sl2insl3}.
We now state some general functorial properties of parabolic induction.
\begin{lemma}\label{lem:actionofcenteroflevionnilradofpara}
The adjoint action of $\zi$ (the center of $\levi$)  on $\Univ{\nili}$, respectively $\Univ{\nili^-}$, is semisimple with finite dimensional weight spaces of positive, respectively negative, weight. Furthermore
$$\Univ{\nili}^{\ad{\zi}}=\Univ{\nili^-}^{\ad{\zi}}=\langle \, 1\,\rangle_\C.$$
\end{lemma}
\begin{proof}
	By the PBW theorem $\Univ{\nili}$ is generated by monomials in $X_\alpha$ for $\alpha\in\Phi^+_\g\backslash\Phi^+_\levi$, and $X_\alpha$ a generator of $\g_\alpha$. Furthermore $\zi$ contains $\varpi_\alpha^\vee$ for $\alpha\in\Phi^+_\g\backslash\Phi^+_\levi$ and
		$$\left[\varpi_\alpha^\vee,X_\beta\right]=\delta_{\alpha,\beta}X_\beta$$
	for $\alpha,\beta \in \Phi^+_\g\backslash\Phi^+_\levi$. The statement follows.
\end{proof}
Let $\lambda\in\cartan^*$ be dominant for $\Phi_\levi^+$. We define the \emph{parabolic restriction functor for category $\mathcal O$} by
$$\ppres{-}: \cato{\g}{\lambda}\rightarrow\cato{\levi}{\lambda}, M\mapsto \pr{\lambda}(M^{\nili}_{\lambda|_{\zi}}),$$
where by definition
$$M^{\nili}_{\lambda|_{\zi}}=\{m\in M\,|\,\nili m=0\text{ and }Zm=\lambda(Z)m\text{ for all }Z\in\zi\},$$
and $\pr{\lambda}:\cato{\levi}{}\rightarrow\cato{\levi}{\lambda}$ is the projection. This is indeed well-defined by the next theorem, where we list important properties of parabolic induction.
 \begin{theorem}\label{thm:indcollection}
	Let $\lambda\in \cartan^*$ be an integral weight which is dominant for $\Phi_\levi^+$. Then the following statements hold.
	\begin{enumerate} 
		\item The functor $\ppres{-}$ is well-defined and
			$$\ind: \cato{\levi}{\lambda}\rightleftarrows\cato{\g}{\lambda}:\ppres{-}$$
		are adjoint.
		\item $\ind$ is exact and $\ppres{-}$ is left exact.
		\item Moreover $\ppres{\ind M}\cong M$ for all $M$ in $\cato{\levi}{\lambda}$.
		\item For all $\mu \in \Wi\cdot\lambda$ we have $$\ind \Ver{\levi}{\mu}=\Ver{\g}{\mu}\text{ and }\ppres{ \Ver{\g}{\mu}}=\Ver{\levi}{\mu}.$$
		\item \label{thm:indcollection:projtoproj} The functor $\ppres{-}$ is exact and $\ind$ sends (indecomposable) projective modules in $\cato{\levi}{\lambda}$ to (indecomposable) projectives in $\cato{\g}{\lambda}$, if and only if $\lambda$ is dominant for $\Phi_\g^+$.
	\end{enumerate}
\end{theorem}
\begin{proof}
	
	(1) 
	Let us first show that $\ppres{-}$ is well-defined. For this we need to show that for $M\in\cato{\g}{\lambda}$, $M^{\nili}_{\lambda|_{\zi}}$ is really in $\cato{\levi}{}$ . Firstly, $M^{\nili}_{\lambda|_{\zi}}\subset M$ is an $\levi$-submodule of $M$, hence clearly $\cartan$ acts semisimply and $\boreli$ acts locally finitely on it. We need to show that $M^{\nili}_{\lambda|_{\zi}}$ is finitely generated as an $\levi$-module. For this we show that already  $N=M_{\lambda|_{\zi}}\supset M^{\nili}_{\lambda|_{\zi}}$ is finitely generated. Choose a finite set $\{x_i\}$ of $\g$-generators of $M$. Without loss of generality, we can assume that each $x_i$ is an highest weight vector and hence $\{x_i\}$ is even a set of $\borel^-$-generators, where $\borel^-$ denotes the opposite Borel. We decompose 
	$$\borel^-=\boreli^-\oplus\nili^-$$
	where and $\nili^-$ and $\boreli^-$ are the opposites of $\nili$ and $\boreli$. Now
	$$\Univ{\nili^-}\{x_i\}\cap N$$
	is finite dimensional, since the $\zi$ weight spaces of $\Univ{\nili^-}$ are finite dimensional (Lemma \ref{lem:actionofcenteroflevionnilradofpara}). By the PBW theorem a basis of this space provides a finite set of $\levi$-generators of $N$. Since $\Univ{\levi}$ is Noetherian, also $M^{\nili}_{\lambda|_{\zi}}\subset N$ will be finitely generated and hence in category $\cato{\levi}{}$. That $\ppres{M}$ is contained in the the block $\cato{\levi}{\lambda}$ follows from (4). Hence the first statement follows.
	Generally, there are natural isomorphisms
	$$\Hom{\g}{\ind(-)}{-}\cong\Hom{\para}{-}{-}\cong\Hom{\levi}{-}{\pres{-}}$$
	of functors on
	$\levi\modules^{\operatorname{opp}}\times\g\modules$
	and hence an adjunction
	$$\ind:\levi\modules\rightleftarrows\g\modules:\pres{-}.$$
	One easily sees that this induces the stated adjunction $(\ind,\ppres{-})$.
	
	(2) $\ind$ is exact by the PBW theorem and $\ppres{-}$ is left exact since it is right adjoint.
	
	(3) See also \cite[Lemma 5.10]{Sartori2015256}. Let $M\in\cato{\levi}{\lambda}$. We want to show
	$$\ppres{\ind(M)}\cong M.$$
	By the PBW theorem as a vector space (and even as a $\zi$-module)
	$$\ind(M)=\Univ{\g}\otimes_{\Univ{\para}} M\cong\Univ{\nili^-}\otimes_\C M.$$
	Lemma \ref{lem:actionofcenteroflevionnilradofpara} shows that $\Univ{\nili^-}^{\ad{\zi}}=\langle\, 1\,\rangle_\C$ and hence
	$$\ppres{\ind(M)}\subseteq 1\otimes M\cong M.$$
	But every vector in $1\otimes M$ is $\nili$-invariant and the inclusion is actually an equality. The second statement follows.
	
	(4) The first statement is trivial and the second follows from the first and (3).
	
	(5) See also \cite[Lemma 5.11]{Sartori2015256}. As right adjoint functor $\ppres{-}$ is certainly left exact as explained in (2). Let $$M\rightarrow N$$ be a surjection in $\cato{\g}{\lambda}$. We have to show that 
	$$\ppres{M}\rightarrow \ppres{N}$$
	is also surjective. So let $n\in \ppres{N}\subset N$ and $m$ be a preimage in $M$. It suffices to show that $m$ is $\nili$-invariant. But this is easy to see: Applying an element of $\nili$ to $m$ increases its weight $\pz{\lambda}$ (Lemma \ref{lem:actionofcenteroflevionnilradofpara}), but since $\lambda$ is also dominant for $\W$, this is already maximal, and hence $m$ is $\nili$-invariant. Hence $\ppres{-}$ is also right exact. Now functors which are left adjoint to exact functors send projectives to projectives and the first implication follows. If on the other hand $\lambda$ is not dominant for $\Phi_\g^+$, then $\Ver{\levi}{\lambda}$ is projective but $\ind \Ver{\levi}{\lambda}=\Ver{\g}{\lambda}$ is not.
	The statement about the indecomposablity follows from (3).
\end{proof}
We now explain how the relative Harish-Chandra morphism 
$$\HC{\g}{\para}:\Zent{\g}\rightarrow\Zent{\levi}$$
between the center of $\Univ{\g}$ and $\Univ{\levi}$ and the parabolic restriction functor
interact.
The results are a straightforward generalization of the fact that the action of $z\in\Zent{\g}$ on a Verma module is completely determined by the action on the highest weight vector, which is completely described by its image $\HC{\g}{\borel}(z)\in \Sym{\cartan}$ under the Harish-Chandra morphism.
\begin{definition} 
Denote by $\nili$ and $\nili^-$ the nilradical of $\para$ and its opposite $\para^-$.
The \emph{relative Harish-Chandra homomorphism}
	$$\HC{\g}{\para}: \Univ{\g} \rightarrow\Univ{\levi}$$
is obtained by the projection on the first factor in the PBW-decomposition 
$$\Univ{\g}=\Univ{\levi}\oplus \nili^-\Univ{\para^-}\oplus\Univ{\g}\nili.$$ It restricts to a homomorphism of algebras 
$$\HC{\g}{\para}:\Zent{\g}\rightarrow\Zent{\levi}.$$
\end{definition}
\begin{remark}
	This version of the Harish-Chandra morphism depends not only on the Levi but also on the parabolic subalgebra, hence the decoration $\HC{\g}{\para}$.
\end{remark}
\begin{lemma}\label{lem:hcmorphismandinvariants}
In the above notation let $M$ be a $\g$-module, $u\in\Zent{\g}$ and $m\in M^{\nili}$. Then we have $$um=\HC{\g}{\para}(u)m.$$
\end{lemma}
\begin{proof}
	Let $u=u_1+u_2+u_3$ in the above PBW decomposition.
	Since $$\nili^-\Univ{\para^-}\cap\Univ{\g}^{\ad{\mathfrak{z}_\levi}}=\{0\}$$ it follows that $u_2=0$. Hence $$u-\HC{\g}{\para}(u)\in\Univ{\g}\nili$$ and $(u-\HC{\g}{\para}(u))m=0$ for all $m\in M^{\nili}$. See also \cite{Ho98}.
\end{proof}
\begin{corollary}
	Let $M\in \g\modules$. Then the following diagram commutes:
	\begin{center}
		\begin{tikzcd}
			\Zent{\g}\arrow[d]\arrow[r,"\HC{\g}{\para}"]&\Zent{\levi}\arrow[d]\\
			\End{\g}(M)\arrow[r,"(-)^{\nili}"]&\End{\levi}(M^{\nili})
		\end{tikzcd}
	\end{center}
	where the vertical arrows are the action morphisms.
\end{corollary}
\subsection{Parabolic Induction and Translation Functors}
Let $\lambda$ and $\mu$ be some dominant integral weights for $\Phi^+_\g$ or $\Phi^+_\levi$ and $\nu=\mu-\lambda$. Then the corresponding \emph{translation functors} are given by
\begin{align*}
	\trans{\lambda}{\mu}&\colon\cato{\g}{\lambda}\rightarrow\cato{\g}{\mu},~ M\mapsto \pr{\mu}(M\otimes_C \Simpl{\g}{\Sdom{\nu}})\text{ and}\\
	\trans{\lambda}{\mu}&\colon\cato{\levi}{\lambda}\rightarrow\cato{\levi}{\mu},~ M\mapsto \pr{\mu}(M\otimes_\C \Simpl{\levi}{\Sidom{\nu}}),
\end{align*}
where $\pr{\mu}$ denotes the projection to the corresponding block in $\mathcal O$.

In this section we explain how translation functors and parabolic induction interact. There are two different cases. Either a translation functor maps \emph{into} a more singular block, i.e. $\W_{\g,\lambda}\subset\W_{\g,\mu}$. In this case parabolic induction and the translation functor \emph{commute}. Or the translation functor maps \emph{out of} a more singular block and the situation is more complicated. Most results are a direct generalization of character formulas for translation functors as in \cite[Kapitel 2]{jantzen1979moduln}.

The most important tool for this section is the \emph{tensor identity}, which describes in its most general formulation how tensor products and induction for modules over a Hopf algebra and a subalgebra interact. For us, the following formulation suffices.
\begin{lemma}[Tensor identity]
	Let $\mathfrak n\subset \mathfrak{m}$ be finite dimensional complex Lie algebras and $M$ be an $\mathfrak m$-module. Then there is a natural equivalence of functors
		$$\indop_{\mathfrak{n}}^{\mathfrak{m}}\left(-\otimes_\C \res_{\mathfrak{m}}^{\mathfrak{n}} M\right)
		\cong
		\left(\indop_{\mathfrak{n}}^{\mathfrak{m}}-\right)\otimes_\C M:
		\mathfrak{n}\modules
		\rightarrow
		\mathfrak{m}\modules,$$
	such that for $X\in\mathfrak{m}$ and $m\in M$
		$$X\otimes(-\otimes m)\mapsto (X\otimes-)\otimes m+(1\otimes-)\otimes Xm.$$
\end{lemma}
In our specific case of \emph{parabolic induction} this implies the existence of a filtration on tensor products with induced modules.
\begin{lemma}\label{LemmaTensorIdentity}
	Let $M$ be a $\levi$-module and $E$ a finite dimensional $\g$-module. Denote by $\nu_1,\dots, \nu_n$ the weights of $\zi$ on $E$, ordered in a way that $\nu_i\leq\nu_j$ implies $i\leq j$. Then $\left(\ind M\right)\otimes E$ has a filtration, natural in $M$,
		$$\{0\}=N_{n+1}\subset N_n\subset\dots\subset N_1=\left(\ind M\right)\otimes E$$
	with subquotients $N_i/N_{i+1}\cong \ind\left( M\otimes E_{\nu_i}\right)$, where $\nili$ acts trivially on $E_{\nu_i}$.
\end{lemma}

\begin{proof}
	The tensor identity yields
		$$\left(\ind M\right)\otimes E\cong\ind\left(M\otimes_\C \res_\g^\para E\right).$$
	Now set $M_i:=\sum_{j=i}^{n} M\otimes E_{\nu_j}$; this is clearly a $\levi$-submodule of $M\otimes E$. 
	Since furthermore the weights $\nu_i$ are ordered in an ascending way, $M_i$ is also stable under $\nili$ and hence a $\para$-submodule. 
	The modules $M_i$ give a filtration of $M\otimes E$ as a $\para$-module with subquotients $$M_i/M_{i+1}\cong M\otimes E_{\nu_i}.$$
	Since non-zero elements of $\nili$ have non-zero weights with respect to $\zi$, they indeed act trivially on $E_{\nu_i}$.
	Let $N_i:=\U(\g)\otimes_{\U(\para)} M_i$. Using the exactness of parabolic induction and the tensor identity we see that the $N_i$ define a filtration with the desired property. That this is indeed natural follows directly from the explicit description of the $M_i$.
\end{proof}

Since translation functors are built from tensor products with finite dimensional modules, we need to understand how they split when restricted to a Levi subalgebra. Although this is generally a hard question, certain \emph{extremal} direct summands are easily identified.
\begin{figure}[h]
	\begin{center}
		\begin{tikzpicture}
		\foreach \i in {0, 1, ..., 5} {
			\fill (\i*60:3) circle [radius=2pt];
		}
		\fill (.1,.1) circle [radius=2pt];
		\fill (-.1,-.1) circle [radius=2pt];
		\node[left](Z) at (-.1, 0) {$0$};
		\node[right](T) at (3, 0) {$\alpha$};
		\node[left] (S) at (2*60:3) {$\beta$};
		\node[left] (ST) at (3*60:3) {$-\alpha$};
		\node[left] (STS) at (4*60:3) {$-\rho$};
		\node[right] (TS) at (5*60:3) {$-\beta$};
		\node[right] (E) at (60:3) {$\rho$};
		\draw[thick,<-] (0,-3.5) -- (0,3.5);
		\node[right] (zen) at (0,3.3) {$\zi^*$};
		\draw[dashed,rounded corners=1mm] (4.3,-0.4) rectangle (-4.3,0.4);
		\node[right] (IT) at (4.5,-0.1) {$\Simpl{\levi}{\alpha}$};
		\draw[dashed,rounded corners=1mm] (2.7,2.6-0.4) rectangle (-2.7,2.6+0.4);
		\node[right] (IE) at (2.9,2.6-0.1) {$\Simpl{\levi}{\rho}$};
		\draw[dashed,rounded corners=1mm] (2.7,-2.6-0.4) rectangle (-2.7,-2.6+0.4);
		\node[right] (IE) at (2.9,-2.6-0.1) {$\Simpl{\levi}{-\beta}$};
		\draw[dashed,rounded corners=1mm] (-0.5,-0.7) rectangle (0.5,0.7);
		\node[below] (tr) at (0,-0.7) {$\Simpl{\levi}{0}=\operatorname{triv}$};
		\end{tikzpicture}
	\end{center}\label{fig:sl2insl3adrep}
	\caption{\textbf{Splitting of the adjoint representation in the case $\sln_2\subset\sln_3$:} Here  $\Delta_\g=\{\alpha,\beta\}$, $\Delta_\levi=\{\alpha\}$. 
		Dots indicate the weight spaces and the boxes surround the direct summands of the restriction to $\levi$.}
\end{figure}
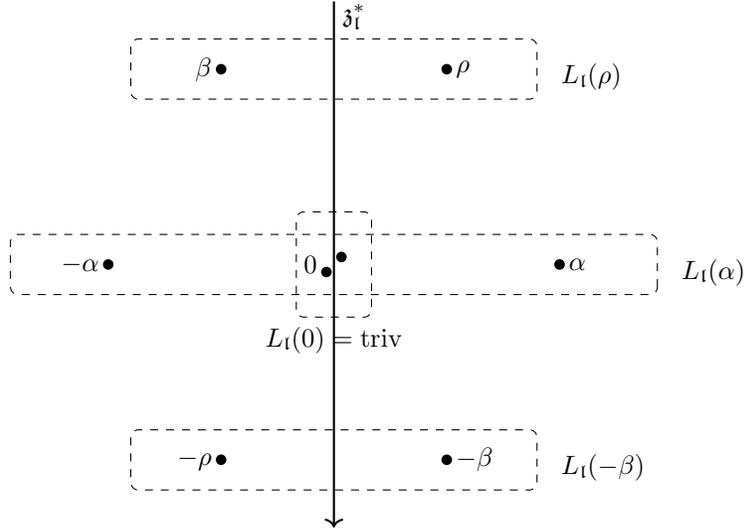
\begin{lemma}\label{LemmaBranchingOfSimples}
	Let $\nu\in\cartan^*$ be some integral weight and $\nu'\defi\nu|_{\zi}$.
	Then, as $\levi$-module, $\Simpl{\levi}{\Sidom{\nu}}$ appears with multiplicity one as direct summand of $\Simpl{\g}{\Sdom{\nu}}_{\nu'}$. Recall that $\Sdom{\nu}$ and $\Sidom{\nu}$ denote the unique elements in $\Weyl\nu\cap\Lambda^+_\g$ and $\Wi\nu\cap\Lambda^+_\levi$.
\end{lemma}
\begin{proof}
	Let $w\in \W$ such that $\Sdom{\nu}=w(\nu)$. Write $w=xy$ with $x\in w\Wi$ a shortest coset representative and $y\in \Wi$. 
	Since $x$ is a shortest coset representative, it maps positive roots for $\levi$ to positive roots for $\g$, and since $\Sdom{\nu}$ is dominant we get
		$$\langle y(\nu),\alpha^\vee\rangle=\langle xy(\nu),x(\alpha)^\vee\rangle=\langle \Sdom{\nu},x(\alpha)^\vee\rangle\geq 0\text{ for all } \alpha \in \Phi_\levi^+.$$
	Therefore
		$$y(\nu)=\Sidom{\nu}\text{ and }xy(\nu)=\Sdom{\nu}.$$
	Now choose some non-zero $v^+\in L(\Sdom{\nu})_{y(\nu)}$. Then $v^+$ is a highest weight vector for $\levi$ since for $\alpha \in \Phi_\levi^+$
		$$\dim_\C \Simpl{\g}{\Sdom{\nu}}_{y(\nu)+\alpha}=\dim_\C \Simpl{\g}{\Sdom{\nu}}_{\Sdom{\nu}+x(\alpha)}=0$$
	because $x(\alpha)\in\Phi^+_\g$ and all weights of $L(\Sdom{\nu})$ are in $\Sdom{\nu}-\Zp\Phi^+_\g$. 
	So indeed $v^+$ generates $\Simpl{\levi}{\Sidom{\nu}}$ as $\levi$-module and $\zi$ acts on it via $\nu'=\nu|_{\zi}=y(\nu)|_{\zi}$. The multiplicity one statement follows from $\dim_\C \Simpl{\g}{\Sdom{\nu}}_{y(\nu)}=1$.
\end{proof}
\begin{theorem}[Translation into a more singular block]\label{thm:intoclosure} 
	Let $\mu$ and $\lambda$ be integral weights which are dominant for $\Phi_\g^+$, such that $\W_{\g,\lambda}\subset\W_{\g,\mu}$. Let $w$ be a shortest  representative of a coset in $\Wi\backslash\Weyl$.
	Let $M\in\cato{\levi}{w\cdot\lambda}$. Then there exists a natural isomorphism
		$$\trans{\lambda}{\mu}\ind M\cong\ind\trans{w\cdot\lambda}{w\cdot\mu}M.$$
\end{theorem}
\begin{proof}
	Firstly, the statement is correct for Verma modules $\Ver{\levi}{xw\cdot\lambda}$, for $x\in\W_{\levi}$, since
	\begin{align*}
	\trans{\lambda}{\mu}\ind \Ver{\levi}{xw\cdot\lambda}&\cong \trans{\lambda}{\mu}\Ver{\g}{xw\cdot\lambda}\cong \Ver{\g}{xw\cdot\mu}\\
	\ind\trans{w\cdot\lambda}{w\cdot\mu} \Ver{\levi}{xw\cdot\lambda}&\cong \ind\Ver{\levi}{xw\cdot\mu}\cong \Ver{\g}{xw\cdot\mu}
	\end{align*}
	by Theorem \ref{thm:indcollection} and \cite[Theorem 7.6]{HumCatO}. By the exactness of the involved functors the statement is hence true on the level of characters. Let $\nu\defi\mu-\lambda$. Then the tensor identity gives a natural isomorphism
	$$\trans{\lambda}{\mu}\ind=\pr{\mu}(\ind(-)\otimes\Simpl{\g}{\Sdom{\nu}})\cong\pr{\mu}(\ind(-\otimes\res_\g^\para\Simpl{\g}{\Sdom{\nu}})).$$
	By Lemma \ref{LemmaTensorIdentity} we see that, for suitable $\nu_i\in\zi^*$, the right hand side has a natural filtration with subquotients
	$$\pr{\mu}(\ind(-\otimes\Simpl{\g}{\Sdom{\nu}}_{\nu_i})).$$
	We will show that this functor is zero except in the case
		$$\nu_i=\pz{w(\nu)}.$$
	Again by exactness, we can test this on Verma modules $\Ver{\levi}{xw\cdot\lambda}$, for $x\in \Wi$.
	In this case 
		$$\ind (\Ver{\levi}{xw\cdot\lambda}\otimes\Simpl{\g}{\Sdom{\nu}})$$
	has a Verma flag with subquotients of the form
		$$\Ver{\g}{xw\cdot\lambda+\xi}$$ 
	for weights $\xi$ of $\Simpl{\g}{\Sdom{\nu}}$.
	By \cite[Satz 2.10]{jantzen1979moduln} or \cite[Lemma 7.5 and Theorem 7.6]{HumCatO} and using the hypothesis $\W_{\g,\lambda}\subset\W_{\g,\mu}$, the only Verma module of this form which is contained in the block $\cato{\g}{\mu}$ is 
	$$\Ver{\g}{xw\cdot\lambda+\xi}=\Ver{\g}{xw\cdot\mu}.$$
	To not be killed by $\pr{\mu}$ hence $\xi$ has to be
	\begin{align*}
		\xi&=xw\cdot\mu-xw\cdot\lambda=xw(\nu) ~~~\text{    and therefore}\\
		\pz{\xi}&=\pz{(xw\cdot\mu-xw\cdot\lambda)}=\pz{xw(\nu)}=\pz{w(\nu)}.
	\end{align*}
	We hence have a natural isomorphism
	$$\pr{\mu}(\ind(-\otimes\res_\g^\para\Simpl{\g}{\Sdom{\nu}}))\cong\pr{\mu}(\ind(-\otimes\Simpl{\g}{\Sdom{\nu}}_{\pz{w(\nu)}})).$$
	Lemma \ref{LemmaBranchingOfSimples} now ensures that, as $\levi$-module, $\Simpl{\levi}{\Sidom{w(\nu)}}$ appears as a direct summand of $\Simpl{\g}{\Sdom{\nu}}_{\pz{w(\nu)}}$. This induces inclusions
		\begin{align*}
		\trans{\lambda}{\mu}\ind&\cong\pr{\mu}(\ind(-\otimes\Simpl{\g}{\Sdom{\nu}}_{\pz{w(\nu)}}))\\
		&\supset\pr{\mu}\ind\left( - \otimes \Simpl{\levi}{\Sidom{w(\nu)}}\right)
		\\
		&\supset\pr{\mu}\ind\pr{w\cdot\mu}\left( - \otimes \Simpl{\levi}{\Sidom{w(\nu)}}\right)\\
		&=\pr{\mu}\ind\trans{w\cdot\lambda}{w\cdot\mu}(-)\\
		&\subset\ind\trans{w\cdot\lambda}{w\cdot\mu}(-)
		\end{align*}
	The inclusions are equalities for Verma modules, and the statement follows by the exactness of all involved functors.
\end{proof}
\begin{theorem}[Translation out of a more singular block]\label{thm:outofclosure}
	Let $\mu$ and $\lambda$ be integral weights, dominant for $\Phi_\levi^+$, with $\W_{\g,\mu}\subset\W_{\g,\lambda}$ and $z^{-1}\in\Wi\backslash\Weyl$ a shortest coset representative such that both $z\cdot\lambda$ and $z\cdot\mu$ are dominant for $\Phi_\g^+$.
	Let $M\in\cato{\levi}{\lambda}$. Then
	$$\trans{z\cdot\lambda}{z\cdot\mu}\ind M$$ has a filtration, natural in $M$,  whose successive quotients are
	$$\ind\trans{\lambda}{w\cdot\mu}M$$
	parametrized by shortest representatives $w$ (with respect to $z\Sim z^{-1}$, see the following Remark \ref{rem:allgroupsareparabolic}) of the double cosets $$\W_{\levi,\lambda}\backslash\W_{\g,\lambda}/\W_{\g,\mu}$$ and
	ordered by the length of $w$. So in particular, $\ind\trans{\lambda}{\mu}M$ is a submodule and $\ind\trans{\lambda}{\tilde{w}\cdot\mu}M$ a quotient of $\trans{z\cdot\lambda}{z\cdot\mu}\ind M$, for $\tilde{w}$ the shortest representative of the longest word in $\W_{\g,\lambda}$.
\end{theorem}
\begin{proof}
	We proceed as in the proof of Theorem \ref{thm:intoclosure}. Let $x\in\Wi$. Then by \cite[Satz 2.17]{jantzen1979moduln} or \cite[Theorem 7.12]{HumCatO} we have the following equalities of characters:
	
	On the one hand
	\begin{align*}
	\character(\trans{z\cdot\lambda}{z\cdot\mu}\ind\Ver{\levi}{x\cdot\lambda})=&\character(\trans{z\cdot\lambda}{z\cdot\mu}\Ver{\g}{x\cdot\lambda})\\
	=&\,\,\sum_{\mathclap{w\in \W_{\g,\lambda}/\W_{\g,\mu}}}\,\,\character \Ver{\g}{xw\cdot\mu}
	\intertext{and on the other hand}
	\character(\trans{\lambda}{w\cdot\mu}\Ver{\levi}{x\cdot\lambda})
	=&\,\,\sum_{\mathclap{y\in \W_{\levi,\lambda}/\W_{\levi,w\cdot\mu}}}\,\,\character \Ver{\levi}{xyw\cdot\mu}
	\intertext{and hence}
	\sum_{\mathclap{w\in \W_{\levi,\lambda}\backslash\W_{\g,\lambda}/\W_{\g,\mu}}}\,\,\character(\ind\trans{\lambda}{w\cdot\mu}\Ver{\levi}{x\cdot\lambda})
	=&\,\,\sum_{\mathclap{\substack{w\in \W_{\levi,\lambda}\backslash\W_{\g,\lambda}/\W_{\g,\mu} \\ y\in \W_{\levi,\lambda}/\W_{\levi,w\cdot\mu}}}}\,\,\character \Ver{\g}{xyw\cdot\mu}\\
	=& \,\,\sum_{\mathclap{w\in \W_{\g,\lambda}/\W_{\g,\mu}}}\,\,\character \Ver{\g}{xw\cdot\mu}.
	\end{align*}
	In the last equality we used that the stabilizer in $\W_{\levi,\lambda}$ of a coset $w\W_{\g,\mu}\in\W_{\g,\lambda}/\W_{\g,\mu}$ is exactly 
	$w\W_{\g,\mu} w^{-1}\cap \W_{\levi,\lambda}=\W_{\levi,w\cdot\mu}$.
	
	Putting everything together, we obtain
		$$\character(\trans{z\cdot\lambda}{z\cdot\mu}\ind\Ver{\levi}{x.\lambda})=\sum_{\mathclap{w\in \W_{\levi,\lambda}\backslash\W_{\g,\lambda}/\W_{\g,\mu}}}\,\,\character(\ind\trans{\lambda}{w\cdot\mu}\Ver{\levi}{x.\lambda}).$$
	 By the exactness of all involved functors, this shows that our theorem is at least true on the level of characters. 
	
	Now we have to take a more refined look. Let $\nu\defi\mu-\lambda$. Then the tensor identity gives a natural isomorphism
	$$\trans{z\cdot\lambda}{z\cdot\mu}\ind=\pr{z\cdot\mu}(\ind(-)\otimes\Simpl{\g}{\Sdom{\nu}})\cong\pr{z\cdot\mu}(\ind(-\otimes\res_\g^\para\Simpl{\g}{\Sdom{\nu}})).$$
	By Lemma \ref{LemmaTensorIdentity} we see that, for suitable $\nu_i\in\zi^*$, the right hand side has a natural filtration with subquotients
	$$\pr{z\cdot\mu}(\ind(-\otimes\Simpl{\g}{\Sdom{\nu}}_{\nu_i})).$$ 
	Let us analyse which of them are non-zero. By exactness, this can be tested on Verma modules $\Ver{\levi}{x\cdot\lambda}$, for $x\in \Wi$. 
	In this case 
	$$\ind (\Ver{\levi}{x\cdot\lambda}\otimes\Simpl{\g}{\Sdom{\nu}})$$
	has a Verma flag with subquotients of the form
	$$\Ver{\g}{x\cdot\lambda+\xi}$$
	for weights $\xi$ of $\Simpl{\g}{\Sdom{\nu}}$.
	The only Verma modules of this form which are contained in the block $\cato{\g}{\mu}$ are of the form
	$$\Ver{\g}{x\cdot\lambda+\xi}=\Ver{\g}{xw\cdot\mu}$$
	for $w\in \W_{\g,\lambda}/\W_{\g,\mu}$.
	Hence
	\begin{align*}
	\xi&=xw\cdot\mu-x\cdot\lambda=x(w\cdot\mu-\lambda) ~~~\text{    and therefore}\\
	\pz{\xi}&=\pz{(xw\cdot\mu-x\cdot\lambda)}=\pz{x(w\cdot\mu-\lambda)}=\pz{(w\cdot\mu-\lambda)}.
	\end{align*}
	Notice that the last term does not depend on $x$. 
	By the above, 
	\begin{align*}
\left(\pr{z\cdot\mu}(\ind(-\otimes\Simpl{\g}{\Sdom{\nu}}_{\nu_i}))\neq0\right) \Rightarrow &\\ (\nu_i=\pz{(w\cdot\mu-\lambda)}&\text{ for some $w\in \W_{\g,\lambda}/\W_{\g,\mu}$}).
	\end{align*}
	Choose such $i$ and $w$. Without loss of generality we can assume that $w$ is a shortest representative of a double coset in $\W_{\levi,\lambda}\backslash\W_{\g,\lambda}/\W_{\g,\mu}$, since for $\hat{w}\in\W_{\levi,\lambda}$  
	$$\pz{(\hat{w}w\cdot\mu-\lambda)}=\pz{(\hat{w}w\cdot\mu-\hat{w}\lambda)}=
	\pz{\hat{w}(w\cdot\mu-\lambda)}=\pz{(w\cdot\mu-\lambda)}.$$
	Now \cite[Satz 2.9]{jantzen1979moduln} implies that
	$$w\cdot\mu-\lambda\in \W_{\g,\lambda}\nu$$ 
	and by Lemma \ref{LemmaBranchingOfSimples}, as $\levi$-module, $\Simpl{\levi}{\Sidom{w\cdot\mu-\lambda}}$ appears as a direct summand of $\Simpl{\g}{\Sdom{\nu}}_{\nu_i}$. We hence have a natural inclusion
	\begin{align*}
	\pr{z\cdot\mu}(\ind(-\otimes\Simpl{\g}{\Sdom{\nu}}_{\nu_i}))&\supset\pr{z\cdot\mu}(\ind(-\otimes\Simpl{\levi}{\Sidom{w\cdot\mu-\lambda}}))\\
	&\supset\pr{z\cdot\mu}(\ind\pr{w\cdot\mu}(-\otimes\Simpl{\levi}{\Sidom{w\cdot\mu-\lambda}}))\\
	&\supset\pr{z\cdot\mu}(\ind\trans{\lambda}{w\cdot\mu}(-))\\
	&=\ind\trans{\lambda}{w\cdot\mu}(-)
	\end{align*}
	Again, by the character computation in the beginning, these are all actually equalities and the statement follows.
\end{proof}
\begin{remark}\label{rem:allgroupsareparabolic}
	In the notation of the preceding Theorem \ref{thm:outofclosure} it makes sense to speak about shortest coset representatives with respect to $z\mathcal{S}_\g z^{-1}$ in the double quotient
	$\W_{\levi,\lambda}\backslash\W_{\g,\lambda}/\W_{\g,\mu}$, since all involved groups are generated by their respective intersection with $z\mathcal{S}_\g z^{-1}$. By \cite[Theorem 1.12 (c)]{HumCox} this holds for $\W_{\g,\lambda}$ and $\W_{\g,\mu}$ since $\lambda$ and $\mu$ are dominant with respect to $z^{-1}\mathcal{S}_\g z$. Since $z^{-1}$ is a shortest coset representative in $\Wi\backslash\Weyl$ one can easily see that $z^{-1}\cdot\lambda$ is dominant for $\Simi$ and hence again by \cite[Theorem 1.12 (c)]{HumCox} $\W_{\levi,z^{-1}\lambda}=z^{-1}\W_{\levi,\lambda}z$ is generated by its intersection with $\Simi$. But this just means that $\W_{\levi,\lambda}$ is generated by its intersection with  $z\Simi z^{-1}\subset z\Sim z^{-1}$
\end{remark}
\begin{example}
	(1) In the case $\lambda=-\rho$, $\levi=\cartan$, this recovers the well-known fact that the antidominant projective $$\ProC{\g}{w_0\cdot\mu}=\trans{-\rho}{\mu}\Ver{\g}{-\rho}$$
	has a Verma flag with quotients $\Ver{\g}{w\cdot\mu}$ of multiplicity one, where $w\in\Weyl$ is a shortest coset representative in $\Weyl/\W_{\g,\mu}$.
	
	(2) The case $\sln_2\subset\sln_3=\g$: Denote by $\{s,t\}$ the simple reflections in $\Weyl$ and let $\levi$ be the Levi subalgebra with $\Wi=\{1,s\}$. Set furthermore $\lambda=-\varpi_{\alpha_s}$ such that $\W_{\g,\lambda}=\{1,t\}$. We are interested in the interaction of $\ind$ and $\trans{\lambda}{0}$. There are two different cases (A) for modules in $\cato{\levi}{\lambda}$ and (B) for modules in $\cato{\levi}{ts\cdot\lambda}$. 
	In the illustration we indicated the effect of  $\trans{\lambda}{0}$ on Verma modules $\Ver{\g}{w\cdot\lambda}$ by dotted lines and labeled the dominant weights for $\levi$ with bold case letters. 
	
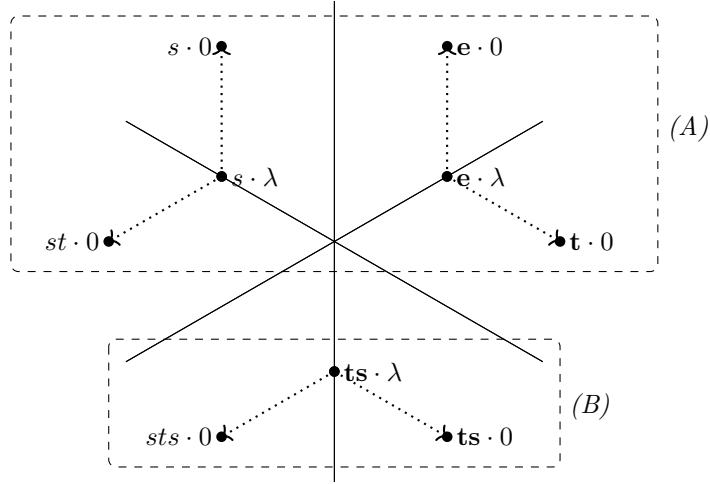
\begin{center}
		\begin{tikzpicture}
		\foreach \i in {0, 1, ..., 5} {
			\draw[thin]  (\i*60+30:3.2)--((\i*60+30+180:3.2);
			\fill  (\i*60:3) circle [radius=2pt];
			\fill  (\i*120+30:1.73) circle [radius=2pt];
		}
		\node[right] (EL) at (30:1.73)  {$\mathbf{e\cdot \lambda}$};
		\node[right] (SL) at (120+30:1.73)  {$s\cdot \lambda$};
		\node[right] (SL) at (240+30:1.73)  {$\mathbf{ts\cdot \lambda}$};
		\node[right](T) at (3, 0) {$\mathbf{t}\cdot 0$};
		\node[left] (S) at (2*60:3) {$s\cdot 0$};
		\node[left] (ST) at (3*60:3) {$st\cdot 0$};
		\node[left] (STS) at (4*60:3) {$sts\cdot 0$};
		\node[right] (TS) at (5*60:3) {$\mathbf{ts}\cdot 0$};
		\node[right] (E) at (60:3) {$\mathbf{e}\cdot 0$};
			\draw[dashed,rounded corners=1mm] (4.3,-0.4) rectangle (-4.3,3);
			\draw[dashed,rounded corners=1mm] (-3,-1.3) rectangle (3,-3);
			\node[right] (A) at (4.3,1.5) {(A)};
			\node[right] (B) at (3,-2.2) {(B)};
		\draw[thick,dotted,->] (30:1.73) -- (60:3);
		\draw[thick,dotted,->] (30:1.73) -- (0:3);
		
		\draw[thick,dotted,->] (120+30:1.73) -- (2*60:3);
		\draw[thick,dotted,->] (120+30:1.73) -- (3*60:3);
		
		\draw[thick,dotted,->] (240+30:1.73) -- (4*60:3);
		\draw[thick,dotted,->] (240+30:1.73) -- (5*60:3);
		\end{tikzpicture}
	\end{center}
	(A) For modules in $M\in\cato{\levi}{\lambda}$ parabolic induction and translation out of the wall do \emph{not commute}, since $\W_{\levi,\lambda}\neq\W_{\g,\lambda}$. We rather get a short exact sequence 	
	\begin{center}
		\ses{\ind\trans{\lambda}{0}M}{\trans{\lambda}{0}\ind M}{\ind\trans{\lambda}{t\cdot 0}M}
	\end{center}
	(B) For modules in $M\in\cato{\levi}{ts\cdot\lambda}$ parabolic induction and translation out of the wall do \emph{commute}, since $\W_{\levi,ts\cdot\lambda}=\W_{\g,ts\cdot\lambda}$:
	$$\trans{\lambda}{0}\ind M=\ind\trans{ts\cdot\lambda}{ts\cdot 0}M.$$
\end{example}
The case (B) from the preceding example can be generalized to the following statement.
\begin{corollary}\label{cor:outofclosurecommutecase}
	In the notation of Theorem \ref{thm:outofclosure} assume additionally that $\W_{\levi,\lambda}=\W_{\g,\lambda}$. Let $M\in\cato{\levi}{\lambda}$. Then there is a natural equivalence
		$$\trans{z\cdot\lambda}{z\cdot\mu}\ind M\cong\ind\trans{\lambda}{\mu}M.$$
\end{corollary}
Composing translation functors into and out of a \emph{wall}, i.e. a block of category $\mathcal O_\lambda$ with $|\W_\lambda|=2$, yield so called \emph{wall crossing functors}, whose interaction with parabolic induction is described in the following. This will be an essential ingredient in the induction step of our proof that parabolic induction and geometric parabolic induction correspond to each other.
\begin{theorem}[Parabolic induction and wall crossing functors]\label{thm:wallcrossing}
	Let $\lambda \in \cartan^*$ be an integral  \emph{regular} weights which is dominant for $\Phi_\g^+$. Let $w\in \Wi\backslash\Weyl$ a shortest coset representative and $s\in \Weyl$ a simple reflection with $ws>w$ such that $ws$ is also a shortest coset representative for $\Wi\backslash\W$.
	Denote by $\theta_s$ a wall-crossing functor through the $s$-wall. 
	Namely, choose some dominant weight $\mu$ with stabilizer $\W_{\g,\mu}=\{1,s\}$ and put $\theta_s=\trans{\mu}{\lambda}\trans{\lambda}{\mu}$.
	Then for all $M\in \cato{\levi}{w\cdot\lambda}$ with a Verma flag there is a short exact sequence, natural in $M$,
	\begin{center}
		\ses{\ind M}{\theta_s\ind M}{\ind\trans{w\cdot\lambda}{ws\cdot\lambda}M}
	\end{center}
	where the first morphism is the unit of the adjunction between $\trans{\mu}{\lambda}$ and $\trans{\lambda}{\mu}$.
\end{theorem}
\begin{proof}
	By Theorem, \ref{thm:intoclosure} we have
		$$\theta_s\ind M\defi\trans{\mu}{\lambda}\trans{\lambda}{\mu}\ind M\cong\trans{\mu}{\lambda}\ind\trans{w\cdot\lambda}{w\cdot\mu} M.$$
		By Theorem \ref{thm:outofclosure}, there is a short exact sequence
		\begin{center}
			\ses{\ind\trans{w\cdot\mu}{w\cdot\lambda}\trans{w\cdot\lambda}{w\cdot\mu} M}{\trans{\mu}{\lambda}\ind\trans{w\cdot\lambda}{w\cdot\mu} M}{\ind\trans{w\cdot\mu}{ws\cdot\lambda}\trans{w\cdot\lambda}{w\cdot\mu} M}
		\end{center}
		Now $w\cdot\mu$ is also regular with respect to for $\W_\levi$: We have
			$$\W_{\levi,w\cdot\mu}=w\W_{\g,\mu}w^{-1} \cap\Wi=\{\id,wsw^{-1}\}\cap\Wi=\{\id\}$$
		since $wsw^{-1}\in \Wi$ would imply that $ws$ and $w$ are in the same coset in $\Wi\backslash\W$ which is a contradiction to the assumption that both are shortest representatives and $ws>w$.

		Hence we have $$\trans{w\cdot\mu}{w\cdot\lambda}\trans{w\cdot\lambda}{w\cdot\mu}\cong\id\text{ and } \trans{w\cdot\mu}{ws\cdot\lambda}\trans{w\cdot\lambda}{w\cdot\mu}\cong\trans{w\cdot\lambda}{ws\cdot\lambda}.$$
		
		That we can indeed choose the first morphism in the short exact sequence as the unit of the adjunction, say $\kappa_s$, follows as in \cite[Theorem 12.2(b)]{Hum} (be aware that his notation is different, since he parametrizes blocks and translation/wall crossing functors by antidominant weights, hence everything is conjugated/multiplied by the longest element $w_0$). By induction on the Verma flag of $M$ we see that the adjunction morphism is indeed injective, and then we use that $\ind M$ is unique as submodule of $\theta_s \ind M$. 
		Let us spell this out in more detail. Let 
		$$0=M_0\subset M_1\subset\cdots \subset M_n=M$$
		be a filtration of $M$ such that the successive quotients are Verma modules. If $n=0$ the statement is trivial.
		Else, we have the following diagram of short exact sequences
		\begin{center}
			\begin{tikzcd}[column sep=0.6cm]
				0\arrow[r]& \theta_s\ind M_{n-1}\arrow[r]&\theta_s \ind  M_{n}\arrow[r]&\theta_s \ind\Ver{\levi}{xw\cdot\lambda}\arrow[r] & 0\\
				0\arrow[r]& \ind M_{n-1}\arrow[r]\arrow[u,"\kappa_s"]&\ind M_{n}\arrow[r]\arrow[u,"\kappa_s"]&\ind\Ver{\levi}{xw\cdot\lambda}\arrow[u,"\kappa_s"]\arrow[r] & 0
			\end{tikzcd}
		\end{center}
		for some $x\in \Wi$. We can assume that the left vertical arrow is injective by induction.
		Now 
		\begin{align*}
		\Hom{\g}{\ind\Ver{\levi}{xw\cdot\lambda}}{\theta_s \ind\Ver{\levi}{xw\cdot\lambda}}&=\\
		\Hom{\g}{\trans{\lambda}{\mu}\Ver{\g}{xw\cdot\lambda}}{\trans{\lambda}{\mu}\Ver{\g}{xw\cdot\lambda}}&=\\
		\Hom{\g}{\Ver{\g}{xw\cdot\mu}}{\Ver{\g}{xw\cdot\mu}}&=\C
		\end{align*}
		Hence $\kappa_s$ is (up to scalar) the unique non-zero morphism $\ind\Ver{\levi}{xw\cdot\lambda}\rightarrow\theta_s \ind\Ver{\levi}{xw\cdot\lambda}$. Since we also know that $\ind\Ver{\levi}{xw\cdot\lambda}$ appears as (even unique) submodule in $\theta_s \ind\Ver{\levi}{xw\cdot\lambda}$, $\kappa_s$ has to be injective.
		Hence also the right vertical arrow of our diagram is injective and 
		we get that 
		$$\kappa_s: \ind  M\rightarrow\theta_s \ind  M$$
		is injective. That $\ind  M$ is indeed unique as a submodule of $\theta_s \ind  M$ can also be seen by an inductive argument. Let $\Ver{\levi}{xw\cdot\lambda}\subseteq M$ such that no weight in $M$ is bigger than $xw\cdot\lambda$. The assumption $ws>w$ guarantees that also in $\theta_s\ind M$ no weight bigger than $xw\cdot\lambda$. Hence we have $$\Univ{\g}(\theta_s\ind M)_{xw\cdot\lambda}=\Ver{\g}{xw\cdot\lambda}^{\oplus(M:\Ver{\levi}{xw\cdot\lambda})}\subseteq\theta_s\ind M.$$
		This is clearly the unique submodule of this form. Now we can pass to the quotient and apply the same argument again. The statement follows by induction.
\end{proof}
\begin{corollary}\label{cor:wallcrossing}
	There is a natural equivalence of functors
	$$\ind\trans{w\cdot\lambda}{ws\cdot\lambda}\cong\coker(\ind\rightarrow\theta_s\ind): \Proje\cato{\levi}{w\cdot\lambda}\rightarrow\Proje\cato{\g}{\lambda}.$$
\end{corollary}
\begin{remark}
	The functor $\coker(\id\rightarrow\theta_s)$ is also known as \emph{shuffling functor}. In general, it maps Verma modules to so called shuffled or twisted Verma modules, i.e. modules which have the same character as a Verma module, but a different (shuffled) composition series. See for example \cite{irving1993shuffled}, \cite{andersen2003twisted} and \cite[Chapter 12.1]{HumCatO}. Since we only apply the functor in the particular situation $ws>w$, no shuffling occurs.
\end{remark}
\begin{corollary}
	Let $\lambda\in\cartan^*$ be an integral \emph{regular} weights which is dominant for $\Phi_\g^+$ and $s\in\Wi$ a simple reflection. Let $M\in\cato{\levi}{\lambda}$. Then there is a natural isomorphism
	$$\theta_s\ind M=\ind \theta_s M.$$
\end{corollary}
\begin{proof}
	Directly follows from Theorem \ref{thm:intoclosure} and Corollary \ref{cor:outofclosurecommutecase}.
\end{proof}
At least up to taking it to some $n$-fold direct sum, parabolic induction for singular blocks of category $\mathcal O$ can be expressed in terms of parabolic induction for regular blocks, by translating out, then inducing, and translating into the singular block again. 
\begin{theorem}\label{thm:fromregtosingind}\sloppy
	Let $\lambda,\mu$ be integral weights which are dominant for $\Phi_\g^+$ where $\mu$ is furthermore regular. Let $w$ be a shortest  representative of a coset in $\Wi\backslash\Weyl/\W_{\g,\lambda}$ and $n=|\W_{\levi,\lambda}|$.
	Then there is a natural equivalence of functors
		$$\trans{\mu}{\lambda}\ind\trans{w\cdot\lambda}{w\cdot\mu}\cong(\ind)^{\oplus n}:\cato{\levi}{w\cdot\lambda}\rightarrow\cato{\g}{\lambda}.$$
\end{theorem}
\begin{proof}
	By Theorem \ref{thm:intoclosure} we have $$\trans{\mu}{\lambda}\ind\trans{w\cdot\lambda}{w\cdot\mu}\cong\ind\trans{w\cdot\mu}{w\cdot\lambda}\trans{w\cdot\lambda}{w\cdot\mu} \cong\ind\id^{\oplus n},$$
	where the last isomorphism follows by using $$\trans{w\cdot\mu}{w\cdot\lambda}\trans{w\cdot\lambda}{w\cdot\mu}\Ver{\levi}{w\cdot\lambda}=\Ver{\levi}{w\cdot\lambda}^{\oplus n}$$
	and the classification of projective functors from \cite{BGproj}, i.e. that projective functors (and natural transformations between them) are completely determined by their effect on a dominant Verma module (in fact any Verma module).
\end{proof}
\subsection{Parabolic Induction and Soergel Modules}
In \cite{Soe90} Soergel gives a completely combinatorial description of the bounded derived category of a block of category $\mathcal O$ in terms of the bounded homotopy category of Soergel modules over the endomorphism ring of its antidominant projective module. In this section we aim to give a description of parabolic induction on the level of Soergel modules, i.e. fill out the question mark in the diagram
\begin{center}
	\begin{tikzcd}
		\Hotb(\CoInv{\levi}{w\cdot\lambda}\Smodules)\arrow[r,"?"]&\Hotb(\CoInv{\g}{\lambda}\Smodules)\\
		\Derb(\cato{\levi}{w\cdot\lambda})\arrow[r,"\ind"]\arrow[u,"\wr"]&\Derb(\cato{\g}{\lambda})\arrow[u,"\wr"].
	\end{tikzcd}
\end{center}
Let $\lambda\in\cartan^*$ be an integral weight which is dominant for $\Phi^+_\g$ and $w$ a shortest representative in $\Wi\backslash\Weyl/\W_{\g,\lambda}$. Denote by
$$\aP{\g}{\lambda}\text{ and }\aP{\levi}{w\cdot\lambda}$$
the antidominant (self-dual) projective in $\cato{\g}{\lambda}$, respectively $\cato{\levi}{w\cdot\lambda}$, and by 
$$\CoInv{\g}{\lambda}=\End{\g}(\aP{\g}{\lambda})\text{ and }\CoInv{\levi}{w\cdot\lambda}=\End{\levi}(\aP{\levi}{w\cdot\lambda})$$
their endomorphism rings. Then \emph{Soergel's functor $\mathbb{V}$} (see \cite{Soe90}) is defined by 
\begin{align*}
	\V{\g}{\lambda}\defi\Hom{\g}{\aP{\g}{\lambda}}{-}&: \cato{\g}{\lambda}\rightarrow \operatorname{mod-}\CoInv{\g}{\lambda}=\CoInv{\g}{\lambda}\modules\text{ and}\\
	\V{\levi}{w\cdot\lambda}\defi\Hom{\levi}{\aP{\levi}{w\cdot\lambda}}{-}&: \cato{\levi}{w\cdot\lambda}\rightarrow \operatorname{mod-}\CoInv{\levi}{w\cdot\lambda}=\CoInv{\levi}{w\cdot\lambda}\modules.
\end{align*}
\begin{theorem}[Struktursatz \cite{Soe90}] Soergel's functor $\mathbb{V}$ is fully faithful on projective modules.
\end{theorem}
\begin{definition}
	The modules in the essential image of the restriction of $\mathbb{V}$ to projective modules are called Soergel modules, so that $\mathbb{V}$ induces an equivalence of categories:
	\begin{align*}
			\V{\g}{\lambda}&:\Proje\cato{\g}{\lambda}\stackrel{\sim}{\rightarrow}\CoInv{\g}{\lambda}\Smodules\\
			\V{\levi}{w\cdot\lambda}&:\Proje\cato{\levi}{w\cdot\lambda}\stackrel{\sim}{\rightarrow}\CoInv{\levi}{w\cdot\lambda}\Smodules
	\end{align*}
	between projectives in $\mathcal O$ and the category of Soergel modules over $\CoInv{\g}{\lambda}$, respectively $\CoInv{\levi}{w\cdot\lambda}$,  denoted by $\CoInv{\g}{\lambda}\Smodules$, respectively by $\CoInv{\levi}{w\cdot\lambda}\Smodules$.
\end{definition}
\begin{remark}
	Abbreviate $C=\CoInv{\g}{\lambda}$. Then the category $C\Smodules$ is generated by modules of the form
	$$C\otimes_{C^{s_n}}\dots C\otimes_{C^{s_1}}\C$$
	for simple reflections $s_i$ 
	, with respect to finite direct sums, taking direct summands and isomorphism. This corresponds to the fact that for regular $\lambda$ all projectives in $\cato{\g}{\lambda}$ appear as direct summands in the projective modules
	$$\theta_{s_n}\cdots\theta_{s_1}\Ver{\g}{\lambda}$$
	and the following Lemma.
\end{remark}
\begin{lemma}[\cite{Soe90} Theorem 10]
	Let $\lambda,\mu\in\cartan^*$ be dominant integral weights such that $\W_{\g,\lambda}\subseteq\W_{\g,\mu}$. Then there are natural isomorphisms of functors
	\begin{align*}
		\V{\g}{\lambda}\trans{\mu}{\lambda}&\cong \CoInv{\g}{\lambda}\otimes_{\CoInv{\g}{\mu}}\V{\g}{\mu}\text{ and}\\
		\V{\g}{\mu}\trans{\lambda}{\mu}&\cong \res_{\CoInv{\g}{\lambda}}^{\CoInv{\g}{\mu}}\V{\g}{\lambda}.
	\end{align*}
\end{lemma}
\begin{corollary}
	There are equivalences of categories
	\begin{align*}
		\Derb(\cato{\g}{\lambda})&\stackrel{\sim}{\leftarrow}\Hotb(\Proje\cato{\g}{\lambda})\stackrel{\sim}{\rightarrow}\Hotb(\CoInv{\g}{\lambda}\Smodules),\\
		\Derb(\cato{\levi}{w\cdot\lambda})&\stackrel{\sim}{\leftarrow}\Hotb(\Proje\cato{\levi}{w\cdot\lambda})\stackrel{\sim}{\rightarrow}\Hotb(\CoInv{\levi}{w\cdot\lambda}\Smodules).
	\end{align*}
\end{corollary}
We will need this statement about antidominant projectives and parabolic restriction.
\begin{lemma}\label{lem:pararesantidomproj} Let $\lambda \in \cartan^*$ be an  integral weight which is dominant for $\Phi^+_\g$. Then
	$$\aP{\levi}{\lambda}\cong\ppres{\aP{\g}{\lambda}}.$$
\end{lemma}
\begin{proof}
	We have the following equalities:
	\begin{align*}
	\aP{\g}{\lambda}\cong\trans{-\rho}{\lambda}\Ver{\g}{-\rho}\cong\trans{-\rho}{\lambda}\ind\Ver{\levi}{-\rho}.
	\end{align*}
	By Theorem \ref{thm:outofclosure} we know that the right hand side has a filtration $\{N_i\}$ with successive quotients
	$$\ind\trans{-\rho}{w\cdot\lambda}\Ver{\levi}{-\rho},$$
	for $w$ shortest representatives of the double cosets $\W_{\levi}\backslash\Weyl/\W_{\g,\lambda}.$
	But for all $w\neq\id$ we certainly have
	$$\ppres{\ind\trans{-\rho}{w\cdot\lambda}\Ver{\levi}{-\rho}}=0$$
	since $\zi$ acts on $\ind\trans{-\rho}{w\cdot\lambda}\Ver{\levi}{-\rho}$ with weights smaller or equal than $\pz{(w\cdot{\lambda})}$, while $\pz{\lambda}>\pz{(w\cdot{\lambda})}$.
	Our filtration provides us with short exact sequences of the form
	\begin{center}
		\ses{N_{i+1}}{N_i}{\ind\trans{-\rho}{w\cdot\lambda}\Ver{\levi}{-\rho}}
	\end{center}
	and since $\ppres{-}$ is left exact we get exact sequences 
	\begin{center}
		\lses{\ppres{N_{i+1}}}{\ppres{N_{i}}}{\ppres{\ind\trans{-\rho}{w\cdot\lambda}\Ver{\levi}{-\rho}}.}
	\end{center}
	As long as $w\neq\id$ the right term vanishes and hence
	$$\ppres{N_{i+1}}=\ppres{N_{i}}.$$
	By induction we get 
	$$\ppres{\aP{\g}{\lambda}}=\ppres{\ind\trans{-\rho}{\lambda}\Ver{\levi}{-\rho}}=\ppres{\ind\aP{\levi}{\lambda}}.$$
	But for modules in $M\in\cato{\levi}{\lambda}$ we know that $\ppres{\ind(M)}=M$  by Theorem \ref{thm:indcollection} and the statement follows.
\end{proof}
Soergel's \emph{Endomorphismensatz} gives a completely explicit description of the endomorphism rings of antidominant projectives in category $\mathcal O$. His description is compatible with parabolic restriction in the following way.
\begin{theorem}[Endomorphismensatz \cite{Soe90}]Let $\lambda \in \cartan^*$ be  an integral weight which is dominant for $\Phi^+_\g$. \label{thm:kompatiblerendomorphsatz}Then the following diagram commutes:
	\begin{center}
		\begin{tikzcd}
			(\Sym{\cartan}/(\Sym{\cartan}^{\Wi}_+))^{\W_{\levi,\lambda}}
			& (\Sym{\cartan}/(\Sym{\cartan}^{\Weyl}_+))^{\W_{\g,\lambda}}\arrow[l]
			\\
			\Sym{\cartan}^{(\Wi\cdot)}  \arrow[two heads, "p\circ(+\lambda)^\sharp",u]
			& \Sym{\cartan}^{(\Weyl\cdot)} \arrow[hook',l] \arrow[two heads,"p\circ(+\lambda)^\sharp"',u]
			\\
			\Zent{\levi}  \arrow[two heads,"\operatorname{act}"',d]\arrow["\HC{\levi}{\boreli}","\wr"',u] &
			\Zent{\g} \arrow[l,"\HC{\g}{\para}"'] \arrow[two heads,"\operatorname{act}",d] \arrow["\HC{\g}{\borel}"',"\wr",u]
			\\
			\CoInv{\levi}{\lambda}\defi\End{\levi}(\aP{\levi}{\lambda})  & 
			\arrow[l,"\ppres{-}"'] \End{\g}(\aP{\g}{\lambda})\defi \CoInv{\g}{\lambda}
		\end{tikzcd}
	\end{center}
	where $(+\lambda)^\sharp$ denotes translation of a function in $\Sym{\cartan}=\mathcal{O}(\cartan^*)$ by $\lambda$ and $p$ the projection.
	Furthermore, the upward arrows $p\circ(+\lambda)^\sharp\circ\HC{\levi}{\boreli}$ and $p\circ(+\lambda)^\sharp\circ\HC{\g}{\borel}$ are surjective and have the same kernel as the downward arrows $\operatorname{act}$, which are also surjective.
\end{theorem}
\begin{proof}
	For the horizontal morphism on the bottom we use Lemma \ref{lem:pararesantidomproj} which provides an isomorphism $\aP{\levi}{\lambda}\cong\ppres{\aP{\g}{\lambda}}$. Lemma \ref{lem:hcmorphismandinvariants} shows that the lower square commutes. The middle square commutes since (relative) Harish-Chandra homomorphisms are compatible, see \cite[Equation 1.12]{Ho98}. The upper one commutes by definition. The other statements are \cite[Endomorphismensatz]{Soe90}.
\end{proof}

\begin{corollary}\label{cor:kompatiblerendomorphsatz}
The following diagram commutes:
	\begin{center}
		\begin{tikzcd}
			(\Sym{\cartan}/(\Sym{\cartan}^{\Wi}_+))^{\W_{\levi,\lambda}}\arrow["\wr"',d] &
			(\Sym{\cartan}/(\Sym{\cartan}^{\Weyl}_+))^{\W_{\g,\lambda}} \arrow[l] \arrow["\wr",d]
			\\
			\CoInv{\levi}{\lambda}=\End{\levi}(\aP{\levi}{\lambda})  & 
			\arrow[l,"\ppres{-}"'] \End{\g}(\aP{\g}{\lambda})=\CoInv{\g}{\lambda}
		\end{tikzcd}
	\end{center}
	Here the vertical arrows are defined as in the preceding Theorem.
\end{corollary}
We are now able to prove how parabolic induction and Soergel's functor $\mathbb{V}$ interact for blocks $\cato{\levi}{\lambda}$ where $\lambda\in\cartan^*$ is an integral weight which is dominant for $\Phi^+_\g$. This is the base case of the inductive proof of the general case.
\begin{theorem}\label{thm:inductionbasecase}
	Let  $\lambda\in\cartan^*$ be an integral weight which is dominant for $\Phi^+_\g$. Then the following diagram commutes up to natural isomorphism.
	\begin{center}
		\begin{tikzcd}
			\CoInv{\levi}{\lambda}\modules \arrow["\res_{\CoInv{\levi}{\lambda}}^{\CoInv{\g}{\lambda}}",r]  & 
			\CoInv{\g}{\lambda}\modules 
			\\
			\cato{\levi}{\lambda} \arrow[r,"\ind"] \arrow[u,"\V{\levi}{\lambda}"]& 
			\cato{\g}{\lambda} \arrow[u,"\V{\g}{\lambda}"']
		\end{tikzcd}
	\end{center}
\end{theorem}
\begin{proof}
	Since $\ind$ is left adjoint, we have to use a different definition of Soergel's functor $\mathbb V$. By \cite[Lemma 9]{Soe90}, there are equivalences of functors:
	\begin{align*}
	\V{\g}{\lambda}&=\Hom{\g}{\aP{\g}{\lambda}}{-}\cong d\Hom{\g}{-}{\aP{\g}{\lambda}}\\
	\V{\levi}{\lambda}&=\Hom{\levi}{\aP{\levi}{\lambda}}{-}\cong d\Hom{\levi}{-}{\aP{\levi}{\lambda}},
	\end{align*}
	where $d$ denotes the duality.
	There are the following equivalences of functors $\cato{\levi}{\lambda}\rightarrow \C\modules$:

	\begin{center}
		\begin{tikzcd}
			\Hom{\g}{\ind-}{\aP{\g}{\lambda}}\arrow[r,"\ppresa{\lambda}{-}","\sim"'] &[.5cm]\Hom{\levi}{-}{\ppresa{\lambda}{\aP{\g}{\lambda}}}=\Hom{\levi}{-}{\aP{\levi}{\lambda}}
		\end{tikzcd}	
	\end{center}
	For the first equivalence we use the adjunction and $\ppresa{\lambda}{\ind-}\cong\id$ (Theorem \ref{thm:indcollection}). For the equality  on the right we identify $\aP{\levi}{\lambda}=\ppresa{\lambda}{\aP{\g}{\lambda}}$ which we are allowed to do by Lemma \ref{lem:pararesantidomproj}. 
	
	By Theorem \ref{thm:kompatiblerendomorphsatz} and Corollary \ref{cor:kompatiblerendomorphsatz} this promotes to an equivalence
	\begin{center}
		\begin{tikzcd}
			\Hom{\g}{\ind-}{\aP{\g}{\lambda}}\arrow[r,"\sim"'] &\res_{\CoInv{\levi}{\lambda}}^{\CoInv{\g}{\lambda}}\Hom{\levi}{-}{\aP{\levi}{\lambda}}
		\end{tikzcd}	
	\end{center}
	of functors $\cato{\levi}{\lambda}\rightarrow \CoInv{\g}{\lambda}\modules$. Now we dualize on both sides and 
	obtain the statement.
\end{proof}
\begin{corollary}\label{cor:inductionbasecase}
	Let  $\lambda\in\cartan^*$ be an integral weight which is dominant for $\Phi^+_\g$. Then the following diagram commutes up to natural isomorphism.
\end{corollary}
\begin{center}
	\begin{tikzcd}[column sep=3cm]
		\Hotb(\CoInv{\levi}{\mu}\Smodules)\arrow[r,"\res_{\CoInv{\levi}{\lambda}}^{\CoInv{\g}{\lambda}}"]&\Hotb(\CoInv{\g}{\mu}\Smodules)\\
		\Hotb(\Proje\cato{\levi}{\mu})\arrow[u,"\wr","\V{\levi}{\mu}"']&\Hotb(\Proje \cato{\g}{\mu})\arrow[u,"\wr"',"\V{\g}{\mu}"]\\
		\Derb(\cato{\levi}{\mu})\arrow[r,"\ind"]\arrow[u,"\wr"]&\Derb(\cato{\g}{\mu})\arrow[u,"\wr"'].
	\end{tikzcd}
\end{center}
\begin{proof}
	Follows from Corollary \ref{cor:kompatiblerendomorphsatz}, using that here $\ind$ maps projectives to projectives by Theorem \ref{thm:indcollection} since $\lambda$ is dominant (\ref{thm:indcollection:projtoproj}) and hence acts on the homotopy categories of projectives by pointwise application.
\end{proof}
Now assume that $\lambda\in\cartan^*$ is some \textbf{regular} integral weight  which is dominant for $\Phi^+_\g$ (for example $\lambda=0$), and abbreviate
$$\Coi\defi\CoInv{\g}{\lambda}\cong \Sym{\cartan}/(\Sym{\cartan}^{\Weyl}_+).$$
For a simple reflection $s$, denote by $\Coi^s$ the $s$-invariants. Then $\Coi^s\subset \Coi$ is a Frobenius extension, and we denote by
$$\Rouqs{s}\defi\cdots\rightarrow0\rightarrow \Coi\rightarrow \Coi\otimes_{\Coi^s}\Coi\rightarrow0\rightarrow\cdots$$
the complex of Soergel bimodules over $\Coi$ known as \emph{Rouquier complex}. Here $\Coi\otimes_{\Coi^s}\Coi$ lives in cohomological degree 0, and the map is the unit of the adjunction between $\res_{\Coi}^{\Coi^s}$ and $\Coi\otimes_{\Coi^s}$. For a reduced expression $w=s_n\cdots s_1$ for $w\in\Weyl$ we define a complex of Soergel bimodules by
$$\Rouq{\underline{w}}\defi \Rouqs{s_1}\otimes_{\Coi}\cdots\otimes_{\Coi} \Rouqs{s_n}.$$
In fact, up to homotopy, this complex does not depend on the choice of shortest expression, but this is not important for us. 
Also abbreviate
$$\res\defi \res_{\CoInv{\levi}{\lambda}}^{\CoInv{\g}{\lambda}}.$$
Furthermore identify
$$\CoInv{\levi}{w\cdot\lambda}=\Sym{\cartan}/(\Sym{\cartan}^{\Wi}_+)=\CoInv{\levi}{\lambda}$$
We will show that that on the level of Soergel modules, parabolic induction for regular blocks
$$\ind:\Derb(\cato{\levi}{w\cdot\lambda})\rightarrow\Derb(\cato{\g}{\lambda})$$
is given by the functor
$$\sind{w}{\lambda}\defi\Rouq{\underline{w}}\otimes_{\Coi}\res(-).$$
\begin{theorem}\label{thm:sindregular} 
	Let $\lambda\in\cartan^*$ be a \emph{regular} integral weight which is dominant for $\Phi^+_\g$. Let $w$ be a shortest coset representative in $\Wi\backslash\Weyl$.
	Then the following diagram of functors commutes up to natural isomorphism 
\begin{center}
	\begin{tikzcd}[column sep=3cm]
	\Hotb(\CoInv{\levi}{w\cdot\lambda}\Smodules)\arrow[r,"\sind{w}{\lambda}"]&\Hotb(\CoInv{\g}{\lambda}\Smodules)\\
	\Hotb(\Proje\cato{\levi}{w\cdot\lambda})\arrow[u,"\wr","\V{\levi}{w\cdot\lambda}"']&\Hotb(\Proje \cato{\g}{\lambda})\arrow[u,"\wr"',"\V{\g}{\lambda}"]\\
	\Derb(\cato{\levi}{w\cdot\lambda})\arrow[r,"\ind"]\arrow[u,"\wr"]&\Derb(\cato{\g}{\lambda})\arrow[u,"\wr"'].
\end{tikzcd}
\end{center}
\end{theorem}
\begin{proof}
	First assume that $l(w)=0$, then $w=e$ and the statement is Corollary \ref{cor:inductionbasecase}.
	
	Now let $ws>w$ with both $ws$ and $w$ shortest representatives in $\Wi\backslash\Weyl$. Assuming that the statement holds for $w$, we show that it holds for $ws$. 
	
	Denote by $\Delta$ the equivalence between a derived category and the homotopy category of projectives. Let $M\in \Derb(\cato{\levi}{ws\cdot\lambda})$. Denote $\overline{M}=\trans{ws\cdot\lambda}{w\cdot\lambda}M$.  
	We have the following diagram of distinguished triangles:
	\begin{center}
		\begin{tikzcd}
		\V{\g}{\lambda}\Delta\ind\overline{M}
		\arrow[r]\arrow[d,equal]&
		\V{\g}{\lambda}\Delta\theta_s\ind\overline{M}
		\arrow[r]\arrow[d,"\wr","(1)"']&
		\V{\g}{\lambda}\Delta\ind M
		\arrow[r,"+1"]\arrow[d,equal]&~
		\\
		\V{\g}{\lambda}\Delta\ind\overline{M}
		\arrow[r,"(*)"]\arrow[d,"\wr","(2)"']&
		\Coi\otimes_{\Coi^s}\V{\g}{\lambda}\Delta\ind\overline{M}
		\arrow[r]\arrow[d,"\wr","(2)"']&
		\V{\g}{\lambda}\Delta\ind M
		\arrow[r,"+1"]\arrow[d,equal]&~
		\\
		\sind{w}{\lambda}\V{\levi}{w\cdot\lambda}\Delta\overline{M}
		\arrow[r]\arrow[d,equal,"(3)"']&
		\Coi\otimes_{\Coi^s}\sind{w}{\lambda}\V{\levi}{w\cdot\lambda}\Delta\overline{M}
		\arrow[r]\arrow[d,equal,"(3)"']&
		\V{\g}{\lambda}\Delta\ind M
		\arrow[r,"+1"]\arrow[d,equal]&~
		\\
		\sind{w}{\lambda}\V{\levi}{ws\cdot\lambda}\Delta M
		\arrow[r]&
		\Coi\otimes_{\Coi^s}\sind{w}{\lambda}\V{\levi}{ws\cdot\lambda}\Delta M
		\arrow[r]&
		\V{\g}{\lambda}\Delta\ind M
		\arrow[r,"+1"]&~
		\end{tikzcd}
	\end{center}
	The first triangle is given by Theorem \ref{thm:wallcrossing}.

	 (1) Since $\theta_s$ is exact and maps projectives to projectives, it commutes with $\Delta$. On Soergel modules $\theta_s$ is given by $\Coi\otimes_{\Coi^s}$, see \cite[Korollar 1]{Soe90}.
	 
	 (2) This is the induction hypothesis.
	 
	 (3) Recall that we identified $\CoInv{\levi}{ws\cdot\lambda}=\Sym{\cartan}/(\Sym{\cartan}^{\Wi}_+)=\CoInv{\levi}{w\cdot\lambda}$. The functor $\trans{ws\cdot\lambda}{w\cdot\lambda}$ is an equivalence of categories and we have $\V{\levi}{ws\cdot\lambda}=\V{\levi}{w\cdot\lambda}\trans{ws\cdot\lambda}{w\cdot\lambda}$.
	 
	 ($*$) This is given by the adjunction homomorphism by Theorem \ref{thm:wallcrossing}.
	 
	 We hence have the following isomorphism
	 \begin{align*}
	 \V{\g}{\lambda}\Delta\ind M\cong
	 & \operatorname{Cone}(\sind{w}{\lambda}\V{\levi}{ws\cdot\lambda}\Delta M\rightarrow
	 \Coi\otimes_{\Coi^s}\sind{w}{\lambda}\V{\levi}{ws\cdot\lambda}\Delta M)\\
	 =&\Rouq{s}\otimes_C\sind{w}{\lambda}\V{\levi}{ws\cdot\lambda}M\\
	 =&\sind{ws}{\lambda}\V{\levi}{ws\cdot\lambda}M
	 \end{align*}
	 where by $\operatorname{Cone}$ we denote the mapping cone.
	 In order to show that this is indeed a \emph{natural} isomorphism, we  apply the Lemma \ref{lem:triangulatedfunctor}. We have to show that for $M,N\in\Derb(\cato{\levi}{ws\cdot\lambda})$
	 $$\Hom{\Derb(\cato{\g}{\lambda})}{\ind \trans{ws\cdot\lambda}{w\cdot\lambda} N}{\ind M}=0.$$
	 But this follows from $$\operatorname{Ext}^i(\Ver{\g}{xw\cdot\lambda},\Ver{\g}{yws\cdot\lambda})=0$$ for all $x,y\in\Wi$, $i\in\Z$, see \cite[Theorem 6.11]{Hum}.
\end{proof}
Now let $\mu\in\cartan^*$ be a possibly \textbf{singular} integral dominant weight and let $w$ be a shortest coset representative in $\Wi\backslash\Weyl/\W_{\g,\lambda}$, let $m=|\W_{\levi,w\cdot\mu}|$. 
Then there are natural maps
$$\CoInv{\g}{\mu}\rightarrow\Coi=\CoInv{\g}{\lambda}\leftarrow\CoInv{\levi}{\lambda}=\CoInv{\levi}{w\cdot\lambda}\leftarrow\CoInv{\levi}{w\cdot\mu}$$
On the level of Soergel modules, parabolic induction for singular blocks (or rather an $m$-fold direct sum of it)
$$(\ind)^{\oplus m}:\Derb(\cato{\levi}{w\cdot\mu})\rightarrow\Derb(\cato{\g}{\lambda})$$
is given by the functor
$$\sindn{w}{\mu}\defi\res_{\Coi}^{\CoInv{\g}{w\cdot\mu}}\Rouq{\underline{w}}\otimes_{\Coi}\CoInv{\levi}{\lambda}\otimes_{\CoInv{\levi}{w\cdot\mu}}(-)$$
\begin{theorem}\label{thm:indandsindsingular}\sloppy
	Let $\mu,\lambda\in\cartan^*$ be integral weights which are dominant for $\Phi^+_\g$, where $\lambda$ is furthermore regular. Let $w$ be a shortest coset representative in $\Wi\backslash\Weyl/\W_{\g,\mu}$. 
	Then the following diagram of functors commutes up to natural isomorphism 
	\begin{center}
	\begin{tikzcd}[column sep=3cm]
		\Hotb(\CoInv{\levi}{w\cdot\mu}\Smodules)\arrow[r,"\sindn{w}{\mu}"]&\Hotb(\CoInv{\g}{\mu}\Smodules)\\
		\Derb(\cato{\levi}{w\cdot\mu})\arrow[r,"(\ind)^{\oplus m}"]\arrow[u,"\wr"]&\Derb(\cato{\g}{\mu})\arrow[u,"\wr"].
	\end{tikzcd}
\end{center}
\end{theorem}
\begin{proof}
	Follows from Theorem \ref{thm:sindregular} and Theorem \ref{thm:fromregtosingind} using that under Soergel's functor $\mathbb V$, $\trans{\mu}{\lambda}$ corresponds to $\res_{\Coi}^{\CoInv{\g}{\mu}}$ and $\trans{\lambda}{\mu}$ to $\Coi\otimes_{\CoInv{\g}{w\cdot\mu}}$, see \cite[Theorem 10]{Soe90}.
\end{proof}
Unfortunately, up until this point, we do not know how to get rid of the $m$-fold direct sum.


\section{Geometric Parabolic Induction}
\subsection{Stratified Mixed Tate Motives}
In this section we recall some of the notations and constructions regarding motivic sheaves used in \cite{SoeWe}.
We denote by $\mathcal T$ the system of triangulated categories of motives constructed from the spectrum representing the semisimplification of de Rahm cohomology, see \cite[Section 2.4]{SoeWe}. This is defined in \cite{brad} and forms a \emph{motivic triangulated category} in the sense of Cisinski--D\'eglise, \cite{CD}.

$\mathcal T$ associates to every complex variety $X\in\Var(\C)$ a triangulated $\C$-linear monoidal category
$\Motives{X}$ 
and to every morphism $f: X {\to} Y$ a symmetric triangulated functor $f^*: \Motives{Y} {\to} \Motives{X}$. Denote the tensor unit in $\Motives{X}$ by $\un_X$, we will also refer to this as \emph{constant motive} or \emph{constant sheaf}.

Since this system of categories is a \emph{motivic triangulated category} it comes with a \emph{full six functor formalism}, $f^*, f_*, f_!, f^!, \otimes, \iHom{}$, Verdier dual $\Dual_X$ functor, and a Tate twist functor $(n)$ for $n\in \Z$, fulfilling all the usual properties one is used to from mixed Hogde modules or mixed $\ell$-adic sheaves, see \cite[A.5.1]{CD}.
Furthermore, due to the particular construction of the categories $\Motives{X}$, we have the following additional properties.
\begin{enumerate}[resume]
	\item \text{(\textbf{Grading condition})} There are no extensions between the Tate motives on $\A^n$, or in formulas
	$$\Hom{\Motives{\A^n}}{\un}{\un(n)[m]}=\begin{cases}
	\C,\text{ for }n=m=0\\
	0,\text{ else.}
	\end{cases}$$
	\item (\textbf{Realization functor}) For every $X\in\Var(\C)$ there is a realization functor
	$$\operatorname{Real}: \Motives{X}\rightarrow \Derb(X(\C);\C)$$
	into the bounded derived category of sheaves on $X(\C)$ equipped with the metric topology. The realization functor is triangulated, monoidal and compatible with the six functors.
\end{enumerate}
The categories $\Motives{X}$, similarly to the derived category of sheaves on $X(\C)$, are gigantic. We will restrict ourselves to the analogue to constructible sheaves in our setting, namely to \emph{stratified mixed Tate motives} as introduced in \cite{SoeWe}. We will recall all important definitions and properties in this section.
Let $(X,S)$ be a an affinely stratified variety over $\C$, i.e. a variety X with a finite partition into locally closed subvarieties (called the strata of $X$)
\begin{equation*}
X=\bigcup_{s\in S} X_s,
\end{equation*}
such that each stratum $X_s$ is isomorphic to $\A^n$ for some $n$, and the closure $\overline{X}_s$ is a union of strata. The embeddings are denoted by $j_s: X_s\hookrightarrow X$. The prime example we always have in mind here is the flag variety of a reductive group with its Bruhat stratification.
Starting from this datum, \cite{SoeWe} defines the category of \emph{stratified mixed Tate motives} on $X$, which we recall in this paragraph.
We start with the basic case of just one stratum.
\begin{definition}[\cite{SoeWe} 3.1]\label{def:mixedtateonstrata}
	For $X\cong\A^n$, denote by $\MTDer{}{X}$ the full triangulated subcategory of $\Motives{X}$ generated by motives isomorphic to $\un_X(p)$ for $p\in \Z$. Recall that by $\un_X$ we denote the tensor unit in $\Motives{X}$.
\end{definition}
By the grading property we get the following.
\begin{proposition}\label{prop:mixedtateonstrata}
	For $X\cong\A^n$, we have the following equivalence of monoidal $\C$-linear categories:
	\begin{equation*}
	\MTDer{}{X}\cong\C\modules^{\Z\times\Z}\cong\Derb({\C\modules^\Z}).
	\end{equation*}
\end{proposition}
We can now proceed to the general case. Since our category should be closed under taking Verdier duals and other reasonable combinations of the six functors, we have to assume that $(X,S)$ fulfills an additional condition:
\begin{definition}[\cite{SoeWe} 4.5]
	$(X,S)$ is called \emph{Whitney-Tate} if for all $s,t\in S$ and $M\in \MTDer{}{X_s}$ we have $j_t^*j_{s*}M\in \MTDer{}{X_t}$.
\end{definition}
From now on we always assume that $(X,S)$ is Whitney-Tate. In \cite{SoeWe} it is shown that (partial) flag varieties and other examples are indeed Whitney-Tate.
\begin{definition}[\cite{SoeWe} 4.7]\label{def:stratmixedtate}
	The category of \emph{stratified mixed Tate motives} on $X$, denoted by $\MTDer{S}{X}$, is the full subcategory of $\Motives{X}$ consisting of objects $M$ such that $j_s^*M\in \MTDer{}{X_s}$ for all $s\in S$.
\end{definition}
The right definition of a map between affinely stratified varieties is different to the usual definition of a stratified map, as defined for example in \cite{GM88}. 
\begin{definition}
	Let $(X,S)$ and $(Y,S^\prime)$ be affinely stratified varieties. We call
	$f:X\rightarrow Y$ an \emph{affinely stratified map} if
	\begin{enumerate}
		\item for all $s \in S^\prime$ the inverse image $f^{-1}(Y_s)$ is a union of strata;
		\item for each $X_s$ mapping into $Y_{s^\prime}$, the induced map $f:X_s\rightarrow Y_{s^\prime}$ is a projection $\A^n\times\A^m\rightarrow \A^m$.
	\end{enumerate}
\end{definition}
Stratified mixed Tate motives are compatible with functors induced from affinely stratified maps.
\begin{lemma}\label{lem:sixfunctorsonstrata} Let $X\in \Var(\C)$. Consider
	\begin{equation*}
	s:X \rightleftarrows  \A^n_X:p
	\end{equation*}
	where $p$ denotes the projection and $s$ the zero section. Then\\
	\begin{minipage}{0.5\linewidth}
		\begin{align*}
		p_*(\un_{\A^n_X})&=\un_{X} 	\\
		p^*(\un_{X})&=\un_{\A^n_X} 	\\
		s^*(\un_{\A^n_X})&=\un_{X} 	
		\end{align*}
	\end{minipage}\hspace{-10pt}
	\begin{minipage}{0.2\linewidth}
		\begin{align*}
		p_!(\un_{\A^n_X})&=\un_{X}(-n)[-2n]\\
		p^!(\un_{X})&=\un_{\A^n_X}(n)[2n]\\
		s^!(\un_{\A^n_X})&=\un_{X}(-n)[-2n]
		\end{align*}\end{minipage}\vspace{10pt}\\
	Furthermore $\Dual_X(\un_{X}(m)[2m])=\un_{X}(\dim X-m)[2\dim X-2m]$ if $X$ is smooth.
\end{lemma}
\begin{proposition}\label{prop:sixfunctorsandstratifiedmaps} 
	Let $(X,S)$ and $(Y,S^\prime)$ be affinely Whitney-Tate stratified varieties and $f:X\rightarrow Y$ an affinely stratified map. Then the induced functors restrict to stratified mixed Tate motives on $X$ and $Y$. In formulas
	\begin{equation*}
	f_*,f_!: \MTDer{S}{X}\rightleftarrows\MTDer{S^\prime}{Y}:f^*,f^!
	\end{equation*}
	Also the internal Hom, duality and tensor product restrict.
\end{proposition}
\begin{proof}\cite[Proposition 3.8]{EKe}
%
%
\end{proof}
Weight structures---as first considered in \cite{Bon}---provide a very concise framework for the powerful \emph{yoga of weights}, as applied, for example, in the proof of the Weil conjectures or the decomposition theorem for perverse sheaves.
\begin{definition}
	Let $\mathcal C$ be a triangulated category. A \emph{weight structure} on $\mathcal C$ is a pair $(\mathcal C_{w\leq 0},\mathcal C_{w\geq 0})$ of full subcategories of $\mathcal C$ such that with $\mathcal C_{w\leq n}:=\mathcal C_{w\leq 0}[n]$ and $\mathcal C_{w\geq n}:=\mathcal C_{w\leq 0}[n]$ the following conditions are satisfied:
	\begin{enumerate}
		\item $\mathcal C_{w\leq 0}$ and $\mathcal C_{w\geq 0}$ are closed under direct summands;
		\item $\mathcal C_{w\leq 0}\subseteq \mathcal C_{w\leq 1}$ and $\mathcal C_{w\geq 1}\subseteq \mathcal C_{w\geq 0}$;
		\item for all $X\in \mathcal C_{w\leq 0}$ and $Y\in\mathcal C_{w\geq 1}$, we have $\Hom{\mathcal C}{X}{Y}=0$
		\item for any $X\in \mathcal C$ there is a distinguished triangle \disttriangle{A}{X}{B} with $A\in \mathcal C_{w\leq 0}$ and $B\in \mathcal C_{w\geq 1}$
	\end{enumerate}
	The full subcategory $\mathcal C_{w=0}=\mathcal C_{w\leq 0}\cap\mathcal C_{w\geq 0}$ is called the \emph{heart of the weight struture}.
\end{definition}
A weight structure on stratified mixed Tate motives on an affinely stratified variety can be obtained by gluing of weight structures on the strata. The motive $\un_{\A^n}(p)[q]$ is defined to have weight $q-2p$. 
\begin{definition}
	Let $\MTDer{}{\A^n}_{w\leq 0}$ (resp. $\MTDer{}{\A^n}_{w\geq 0}$) be the full subcategory of $\MTDer{}{\A^n}$ consisting of objects isomorphic to finite direct sums of $\un_{\A^n}(p)[q]$ for $q\leq 2p$ ($q\geq 2p$). This defines a weight structure on $\MTDer{}{\A^n}$.
\end{definition}
\begin{proof}
	We use Proposition \ref{prop:mixedtateonstrata} to identify $\MTDer{}{\A^n}$ with the derived category of graded vector spaces. Here the axioms of a weight structure are easily checked.
\end{proof}
\begin{definition}[\cite{SoeWe} 5.1] Let $(X, S)$ be an affinely Whitney-Tate stratified variety. Then we obtain a weight structure on $\MTDer{ S}{X}$ by setting
	\begin{align*}
	\MTDer{ S}{X}_{w\leq 0}&:=\left\{M\,|\, j_s^*M \in \MTDer{}{X_s}_{w\leq 0}\text{ for all } s\in  S \right\}\\
	\MTDer{ S}{X}_{w\geq 0}&:=\left\{M\,|\, j_s^!M \in \MTDer{}{X_s}_{w\geq 0}\text{ for all } s\in  S\right\}
	\end{align*}
\end{definition}
With this definition we have the following compatibilities with the six functors.
\begin{proposition}[\cite{EKe} 3.12]\label{prop:weightexactness}
	Let $(X, S)$ and $(Y, S^\prime)$ be affinely Whitney-Tate stratified varieties and $f:X\rightarrow Y$ an affinely stratified map. Then
	\begin{enumerate}
		\item the functors $f^*,f_!$ are weight left exact, i.e. they preserve $w\leq0$;
		\item the functors $f^!,f_*$ are weight right exact, i.e. they preserve $w\geq0$;
		\item the tensor product is weight left exact, i.e. restricts to
		\begin{equation*}
		\MTDer{ S}{X}_{w\leq n}\times\MTDer{ S}{X}_{w\leq m}\rightarrow\MTDer{ S}{X}_{w\leq n+m}
		\end{equation*}
		\item Verdier duality reverses weights, i.e. restricts to
		\begin{equation*}
		\Dual_X: \MTDer{S}{X}_{w\leq n}^{\op}\rightarrow\MTDer{S}{X}_{w\geq -n}
		\end{equation*}
		\item the internal Hom functor $\iHom{X}$ is weight right exact, i.e. restricts to
		\begin{equation*}
		\MTDer{ S}{X}_{w\leq n}^{\op}\times\MTDer{ S}{X}_{w\geq m}\rightarrow\MTDer{ S}{X}_{w\geq m-n}
		\end{equation*}
		\item For f smooth $f^!$ and $f^*$ are weight exact;
		\item For f proper $f_!$ and $f_*$ are weight exact;
		\item If X is smooth $\un_X(n)[2n]$ is of weight zero for all $n\in\Z$.
	\end{enumerate}
\end{proposition}
We apply the tilting formalism (see Appendix \ref{subsec:tilting}) to stratified mixed Tate motives. Under an additional \emph{pointwise purity} condition this allows us to identity the category of stratified mixed Tate motives with the homotopy category of its weight zero objects.
\begin{definition}[\cite{SoeWe} 6.1]
	Let $(X, S)$ be an affinely Whitney-Tate stratified variety and  $?\in\{*,!\}$. A motive $M\in\MTDer{ S}{X}$ is called \emph{pointwise ?-pure} if for all $s\in S$
	\begin{equation*}
	i^?_s M\in \MTDer{}{X_s}_{w=0}.
	\end{equation*}
	If both conditions are satisfied, the motive is called \emph{pointwise pure}.
\end{definition}
\begin{proposition}[\cite{SoeWe} 6.3]\label{prop:pwpuretilting}
	Let $(X, S)$ be an affinely Whitney-Tate stratified variety and $M,N\in\MTDer{ S}{X}$ such that $M$ is pointwise $*$-pure and $N$ is pointwise $!$-pure. Then 
	$\Hom{\Motives{X}}{M}{N[a]}=0$ for all $a> 0$.
\end{proposition}
\begin{theorem}[Tilting for motives, \cite{SoeWe} 9.2]\label{thm:tiltingformotives}
	Let $(X, S)$ be an affinely Whitney-Tate stratified variety, such that all objects of $\MTDer{S}{X}_{w=0}$ are additionally pointwise pure. Then there is an equivalence of categories, called \emph{tilting},
	\begin{equation*}
	\Delta:\MTDer{S}{X}\stackrel{\sim}{\rightarrow}\Hotb(\MTDer{ S}{X}_{w=0}).
	\end{equation*}
\end{theorem}
\begin{proof}
The category $\Motives{X}$ can be embedded in a derived category of a Grothendieck abelian category.
\sloppy We can hence take  a system of homotopy injective resolutions of representatives of the isomorphism classes of indecomposable objects in $\MTDer{S}{X}_{w=0}$. This forms a tilting collection by Proposition \ref{prop:pwpuretilting} and the pointwise purity assumption. We can hence apply Theorem \ref{thm:tilting} and the statement follows.
\end{proof}
\begin{remark}\label{rem:tiltingformotives}
	The last theorem can also be stated differently. Namely, let $$\{L_s\}_{s\in S}\subseteq \operatorname{MTDer}_S(X)_{w=0}$$ be a set of representative of indecomposable weight zero stratified mixed Tate motives on $X$. Then $\{L_s(i)[2i]\}_{s\in S,i\in\Z}$ generates $\MTDer{S}{X}_{w=0}$ as an additive subcategory by \cite[Corollary 11.11]{SoeWe}.
	Assume without loss of generality that the $L_s(i)[2i]$ are homotopy projective. They hence form a tilting collection as considered in the proof. Now let $L=\bigoplus_{s\in S} L_s$ and $$H=\bigoplus_{n\in\Z}\Hom{\Motives{X}}{L}{L(n)[2n]}.$$
	Then $H$ is a graded algebra concentrated in even degrees and we have
	\begin{align*}
	\Derb(\operatorname{mod}^{\Z,ev}H)\cong&\Hotb(\langle \{L_s(i)[2i]\}_{s\in S,i\in\Z}\rangle_\oplus^{\Motives{X}})\\
	=&\Hotb(\MTDer{(B)}{X}_{w=0})\cong\MTDer{(B)}{X}
	\end{align*}
	where by $\operatorname{mod}^{ev,\Z}-H$ we denote the category of finitely generated evenly graded $H$ right modules. See Theorem \ref{thm:tilting}.
\end{remark}
\begin{theorem}[Tilting for motives and functors] \label{thm:tiltingformotivesandfunctors}
	\sloppy Let $(X, S)$, $(Y, S')$ be affinely Whitney-Tate stratified varieties, such that all objects of $\MTDer{S}{X}_{w=0}$ and $\MTDer{S'}{Y}_{w=0}$ are additionally pointwise pure. Let $f:X\rightarrow Y$ be an affinely stratified morphism. Assume that either f is smooth and proper or a closed immersion.
	Then tilting commutes with $f^*=f^!(d)[2d]$ and $f_*=f_!$. So for example the diagram 
	\begin{center}
		\begin{tikzcd}
		\Hotb(\MTDer{S}{X}_{w=0}) \arrow[r,"\Delta","\sim"']\arrow[d,"f_*"]& 
		\MTDer{S}{X}\arrow[d,"f_*"]
		\\
		\Hotb(\MTDer{S'}{Y}_{w=0}) \arrow[r,"\Delta","\sim"']& 
		\MTDer{S'}{Y}
	\end{tikzcd}
	\end{center}
	commutes up to natural isomorphism, where on the left side $f_*$ acts by pointwise application. Furthermore tilting commutes with shifts of the form $(n)[2n]$ for $n\in \Z$.
\end{theorem}
\begin{proof}
First of all, the functors $f^*=f^!(d)[2d]$ and $f_*=f_!$ are weight exact, hence really restrict to weight zero motives by Proposition \ref{prop:weightexactness}. It suffices to show the statement for one of the functors, since the other functors are adjoints (or shifts of it).

Asumme that $f$ is smooth and proper. The tilting equivalence was constructed by embedding $\Motives{X}$ and $\Motives{Y}$ in a derived category of a Grothendieck abelian category, say $\mathcal{A}_X$, $\mathcal{A}_Y$. The functor $f^*:\Motives{Y}\rightarrow \Motives{X}$ is actually just pointwise application of an exact functor $f^*:\mathcal{A}_Y\rightarrow\mathcal{A}_X$, since $f$ is smooth (see \cite[5.1.16]{CD}).
We can furthermore assume that our tilting collections for $\{T_i\}\subset\MTDer{S}{X}_{w=0}$ and $\{U_i\}\subset \MTDer{S'}{Y}_{w=0}$ are given by homotopy-injective complexes and chosen in a way, that for every $T_i$ there is a $U_j$ and a quasi-isomorphism $f^*(T_i)\rightarrow U_j$.
We are hence exactly in the setting of Proposition \ref{prop:tiltingandfunctors} (see also Remark \ref{rem:dualtiltingandfunctors}) which shows that $f^*$ commutes with tilting. By adjunction also $f_*=f_!$ have to commute.

The same argument works for $f_!=f_*$ when $f$ is a closed immersion.

Denote by $\pi:\mathbb{P}^1_X\rightarrow X$ the projection. For every $M\in\Motives{X}$ we have a canonical splitting $\pi_*\pi^*M=M\oplus M(-1)[-2]$, hence also shifts of the form $(n)[2n]$ commute with tilting.
\end{proof}
\subsection{Flag Varieties}
In this section we introduce a lot of notation for reductive algebraic groups and partial flag varieties. We also introduce the maps we will later use to construct a geometric version of parabolic induction.

Let $G\supset B \supset T$ be a reductive algebraic group over the complex numbers together with a Borel subgroup and maximal torus.
Denote by 
$$X(T)\supset\Phi\supset\Phi^+\supset\Delta$$
the character lattice of $T$, the root system, set of positive and simple roots associated to the choice of $B$.
Denote by
$$\W=\operatorname{N}_G(T)/T\supset\Ss$$
the Weyl group and the set of simple reflections. By abuse of notation, we will let elements of $\W$ act on cosets, subgroups, etc. whenever the action does not depend on a choice of representative mod $T$.
For $\alpha \in \Phi$ we denote the root subgroup of $G$ on which the conjugation action by $T$ is described by $\alpha$ by
$$\mathbb{G}_a\cong U_\alpha\subset G.$$ 
Denote by $U$ and $U^-$ the unipotent radical of $B$ and $B^-$, where by $B^-$ we denote the opposite Borel. For $x\in \W$ we define $$U_x\defi U\cap xU^-x^{-1}=\langle U_\alpha\,|\,\alpha \in x(\Phi^-)\cap\Phi^+\rangle\subset B.$$

By a \emph{standard parabolic subgroup} we mean a subgroup $G\supset Q\supset B$.
We denote its Weyl group and simple reflections by
$$\Ss_Q=\Ss\cap\W_Q\subset\W_Q\subset\W.$$
From here on out, we always fix one particular standard parabolic subgroup
$$B\subset P \subset G.$$
We denote its Levi factor by
$$P\twoheadrightarrow P/\operatorname{Rad}_u(P)\defi L$$
and for convenience choose a splitting of this map to interpret $L$ as a subgroup of $P$. We denote by 
$$A=B/\operatorname{Rad}_u(P)\subset L$$
its Borel subgroup. In this chapter we will be interested in (partial) flag varieties associated to $G$ and $L$. But for convenience, we will always prefer to work with quotients of $P$ instead of quotients of $L$ using
$$L/A\cong P/B.$$
Partial flag varieties associated to $G$ or $P$ are always  affinely Whitney-Tate stratified varieties with respect to their stratification by Bruhat cells (these are precisely the $B$-orbits) \cite[4.10]{SoeWe} and we can hence look at the associated categories of stratified mixed Tate motives
$$\MTDer{(B)}{G/B}, \MTDer{(B)}{G/P}, \MTDer{(B)}{P/B}\dots$$

We can now introduce all necessary notations and maps we will later use to define a geometric version of parabolic induction. If $B\subset Q\subset G$ is a standard parabolic subgroup, then there is a well known \emph{generalized Bruhat decomposition} of $G$ into $P\times Q$ orbits, given by
\begin{align*}
G&=\biguplus_{ w\in\W_P\backslash\W/\W_Q}P wQ\\
\intertext{and an associated stratification of the partial flag variety}
G/Q&=\biguplus_{ w\in\W_P\backslash\W/\W_Q}P wQ/Q.
\end{align*}
As it turns out, those strata $P wQ/Q$ are affine bundles over partial flag varieties associated to $P$, let say $P/Q_w$, for $P\supset Q_w\supset B$ a standard parabolic depending on $Q$ and $w$. Geometric parabolic induction will be constructed by passing between sheaves on $P/Q_w$ and $G/Q$, using the maps and notation from the following Theorem.
\begin{theorem}\label{thm:geomparaindsetup}
	Let $w\in \W$ be a shortest coset representative in $\W_P\backslash\W/\W_Q$ and set 
	\begin{align*}
		\WPQ{Q}{w}&\defi\W_P\cap w\W_Qw^{-1}\subset \W_P,\\
		\PQ{Q}{w}&\defi B\WPQ{Q}{w}B\subset P\text{ and }\\
		A_w&\defi L\cap \PQ{Q}{w}.
	\end{align*}
	Let $x$ be a shortest representative in $\W_P/\WPQ{Q}{w}$. Consider the diagram
	\begin{center}
		\begin{tikzcd}[column sep= 0.8cm]
		L/A_w &[-15pt]
		P/\PQ{Q}{w} \arrow[l,"\sim"'] &
		P\times^{\PQ{Q}{w}} BwQ/Q \arrow[r,"\sim","\operatorname{mult}"']\arrow[two heads,l,"\pr{w}"] &[-10pt] 
		PwQ/Q \arrow[hook,r,"\operatorname{h_w}"'] &[-15pt]
		G/Q
		\\
		\Boreli x \Boreli_w/\Boreli_w\arrow[u,hook]& 
		Bx\PQ{Q}{w}/\PQ{Q}{w} \arrow[u,hook]\arrow[l,"\sim"'] &
		Bx\PQ{Q}{w} \times^{\PQ{Q}{w}} BwQ/Q \arrow[r,"\sim","\operatorname{mult}"']\arrow[two heads,l,"\pr{w}"]\arrow[u,hook]&
		BxwQ/Q\arrow[u,hook]&
		\\
		& 
		U_x\dot{x} \arrow[u,"\wr"] &
		U_x\dot{x}\times U_w\dot{w} \arrow[two heads,l,"\operatorname{pr}_1"]\arrow[u,"\wr"] \arrow[r,"\sim","\operatorname{mult}"']&
		U_{xw}\dot{x}\dot{w}\arrow[u,"\wr"]
		&			
		\end{tikzcd}
	\end{center}
	where by $\dot{y}\in \operatorname{N}_G(T)$ we denote a representative in $y=\dot{y}T$.
	Then the following statements hold.
	\begin{enumerate}
		\item $\WPQ{Q}{w}$ is generated by simple reflections,
		\item $\PQ{Q}{w}$ acts on $BwQ$ by left multiplication and
		\item $\PQ{Q}{w}$ contains the stabilizer of the action of $P$ on $wQ/Q$ $$\PQ{Q}{w}\supset P\cap wQw^{-1}.$$
		\item The diagram is well defined and all squares are commutative and Cartesian.
		\item The arrows marked by $\sim$ are isomorphisms.
	\end{enumerate}
\end{theorem}
\begin{proof}
	(1) Since Q is a standard parabolic subgroup, $\W_Q$ is the isotropy group of some dominant weight $\lambda$. Since $w$ is a shortest coset representative $w\cdot\lambda$ is still dominant for $\mathcal{S}_P$. Now $\WPQ{Q}{w}$ is the isotropy group of $w\cdot\lambda$ in $\W_P$ and hence generated by simple reflections, see \cite[Theorem 1.12 (c)]{HumCox}.
	
	(2) Let $s\in \WPQ{Q}{w}$ be a simple reflection. Write $s=wqw^{-1}$ for $q\in\W_Q$. Let $y\in \W_Q$, then
	$$BsBwByB\subset BswyB \cup BwyB=BwqyB \cup BwyB\subset BwQ.$$ 
	The statement follows since $\WPQ{Q}{w}$ is generated by simple reflections.
	
	(3) We show that in fact $B(P\cap wQw^{-1})B\subset\PQ{Q}{w}$.
	We have
	\begin{align*}
	B(P\cap wQw^{-1})B&=B\W_PB\cap BwB\W_QBw^{-1}B\\
	&=B\W_PB\cap Bw\W_QBw^{-1}B\\
	&\subset\bigcup_{I}B(\W_P\cap w\W_Qw_I^{-1})B
	\end{align*}
	where the second equality holds since $w$ is reduced with respect to $\W_Q$ and the
  $w_I$ denote subexpressions of $w$ (see \cite[IV.2.1 Lemma 1]{Bou}). But now assume that there are $p\in\W_P$, $q\in\W_Q$ such that
	$p=wqw_I^{-1}$. Then $pw_I=wq$ and both represent the same coset in $\W_P\backslash \W/\W_Q$. Since $l(w_I)\leq l(w)$ and $w$ is the shortest representative, we have $w=w_I$ and hence 
	\begin{align*}
	B(P\cap wQw^{-1})B\subset B(\W_P\cap w\W_Qw^{-1})B=\PQ{Q}{w}.
	\end{align*}
		
	(4) Follows from (2).
	
	(5) 
	The multiplication maps in the first row is an isomorphism by (3). The maps from the bottom to the middle row are isomorphisms since $x,w$ and $xw$ are shortest coset representatives in $\W_P/\WPQ{Q}{w}, \W/\W_Q$ and $\W/\W_Q$, respectively. Here we use that for $x\in\W_P$, $xw$ is a shortest representative in $\W/\W_Q$ if and only if $x$ is a shortest representative in $\W/(\W_P\cap wW_Qw^{-1})$ (see \cite[1.2]{Xe}). The multiplication map in the bottom row is an isomorphism since
	$l(xw)=l(x)+l(w)$ (see \cite[IV Excercise \S 1.3]{Bou}).
	All other statements follow.
	
	
	See also Chapter 3 in \cite{BorTits} for a good reference on BN-pairs and parabolic subgroups.
\end{proof}

\begin{corollary}
In the regular case $Q=B$ the notation simplifies to
		\begin{center}
			\begin{tikzcd}[column sep= 0.8cm]
				L/\Boreli  & 
				P/B\arrow[l,"\sim"']  & 
				P\times^B BwB/B \arrow[r,"\sim","\operatorname{mult}"']\arrow[two heads,l,"\pr{w}"] & 
				PwB/B \arrow[hook,r,"\operatorname{h_w}"'] & 
				G/B
				\\
				\Boreli x \Boreli/\Boreli\arrow[u,hook]& 
				BxB/B \arrow[u,hook]\arrow[l,"\sim"'] &
				BxB \times^B BwB/B \arrow[r,"\sim","\operatorname{mult}"']\arrow[two heads,l,"\pr{w}"]\arrow[u,hook]&
				BxwB/B\arrow[u,hook]&
				\\
				& 
				U_x\dot{x} \arrow[u,"\wr"] &
				U_x\dot{x}\times U_w\dot{w} \arrow[two heads,l,"\operatorname{pr}_1"]\arrow[u,"\wr"] \arrow[r,"\sim","\operatorname{mult}"']&
				U_{xw}\dot{x}\dot{w}\arrow[u,"\wr"]
				&
			\end{tikzcd}
		\end{center}
		for $x\in \W_P$.
\end{corollary}
\begin{example}
	Let $G=\operatorname{GL}_3$ and $s,t$ be the simple reflections in $\W_G$.
	
	(1) The case of \textbf{disjoint parabolic subgroups}. Let $P=B\cup BsB$, $Q=B\cup BtB$ be minimal parabolic subgroups. Then $G/Q=\mathbb{P}^2$ has two $P$-orbits corresponding to the decomposition $\mathbb{P}^2=\mathbb{P}^1\cup\A^2$.\\
	\begin{center}
		\begin{tikzcd}
			\fill[lightlight-gray] (0,4) rectangle (4,0);
			\draw[line width=6pt,white] (0,0) -- (4,0);
			\draw[line width=1.6pt] (0,0) -- (4,0);
			\draw[white,fill=white] (0,0) circle [radius=6pt];
			\fill (0,0) circle [radius=3pt];
			\node at (2,2) {BtsQ/Q};
			\node at (-1.1,2) {PtsQ/Q=};
			\node at (-1.1,-0.1) {PQ/Q=};
			\node[below] at (2,0) {BsQ/Q};
			\node[below] at (0,0) {Q/Q};
			\draw[->,thick] (4.5,0) -- (5.5,0) node [pos=0.66,above] {\operatorname{pr}_e};
			\draw[->,thick] (4.5,2) -- (5.5,2) node [pos=0.66,above] {\operatorname{pr}_{ts}};
			\draw[line width=6pt,white] (8,0) -- (10.5,0);
			\draw[line width=1.6pt] (8,0) -- (10.5,0);
			\draw[white,fill=white] (8,0) circle [radius=6pt];
			\fill (8,0) circle [radius=3pt];
			\fill (8.5,2) circle [radius=3pt];
			\node at (8.5,2.2) {P/\PQ{Q}{ts}};
			\node[below] at (9.7,0) {Bs\PQ{Q}{e}/\PQ{Q}{e}};
			\node[below] at (8,0) {\PQ{Q}{e}/\PQ{Q}{e}};
			\node at (8-1.1,2) {P/\PQ{Q}{ts}=};
			\node at (8-1.1,-0.1) {P/\PQ{Q}{e}=};
		\end{tikzcd}
	\end{center}
	Here $\operatorname{pr}_e:PQ/Q\cong \mathbb{P}^1\stackrel{\sim}{\rightarrow}\mathbb{P}^1\cong P/\PQ{Q}{e}$ and $\operatorname{pr}_{ts}: PtsQ/Q\cong\mathbb{A}^2\rightarrow pt\cong P/\PQ{Q}{ts}$.
	
	(2) The case of \textbf{meeting parabolic subgroups}. Let $P=Q=B\cup BsB$. Then $G/Q=\mathbb{P}^2$ has two $P$-orbits corresponding to the decomposition  $\mathbb{P}^2=pt\cup \mathcal{O}(1)$.
	\begin{center}
		\begin{tikzcd}
			\fill[lightlight-gray] (0,4) rectangle (4,0);
			\draw[line width=6pt,white] (0,0) -- (0,4);
			\draw[line width=1.6pt] (0,0) -- (0,4);
			\draw[white,fill=white] (0,-.2) circle [radius=9pt];
			\fill (0,-.2) circle [radius=3pt];
			\node at (2,2) {BstQ/Q};
			\node at (-1.1,2) {PtQ/Q=};
			\node at (-1.1,-0.3) {PQ/Q=};
			\node[right] at (0,3.5) {BtQ/Q};
			\node[below] at (0,-0.3) {Q/Q};
			\draw[->,thick] (4.5,-.2) -- (5.5,-.2) node [pos=0.66,above] {\operatorname{pr}_e};
			\draw[->,thick] (4.5,2) -- (5.5,2) node [pos=0.66,above] {\operatorname{pr}_{t}};
			\draw[line width=6pt,white] (8,0) -- (10.5,0);
			\draw[line width=1.6pt] (8,2) -- (10.5,2);
			\fill (8,-.2) circle [radius=3pt];
			\draw[white,fill=white] (8,2) circle [radius=6pt];
			\fill (8,2) circle [radius=3pt];
			\node[below] at (8,2) {\PQ{Q}{t}/\PQ{Q}{t}};
			\node[below] at (9.5,2) {Bs\PQ{Q}{t}/\PQ{Q}{t}};
			\node[below] at (8,-.2) {\PQ{Q}{e}/\PQ{Q}{e}};
			\node at (8-1.1,2) {P/\PQ{Q}{t}=};
			\node at (8-1.1,-0.1-.2) {P/\PQ{Q}{e}=};
		\end{tikzcd}
	\end{center}
	Here $\operatorname{pr}_e: Q/Q\cong pt\stackrel{\sim}{\rightarrow}pt\cong P/\PQ{Q}{e}$ and $\operatorname{pr}_t: PtQ/Q\cong\mathcal{O}(1)\rightarrow \mathbb{P}_1\cong P/\PQ{Q}{t}$, where $\mathcal{O}(1)$ is the \emph{hyperplane bundle} or \emph{Serre's twisting sheaf} on $\mathbb{P}^1.$
\end{example}
All maps constructed in the preceding section are well-behaved with respect to passing between different standard parabolic subgroups, as described in the following.
\begin{lemma}\label{lem:mapsgeomparaindintowall}
	In the notation of Theorem \ref{thm:geomparaindsetup} let $Q\subset Q'$ be another standard parabolic containing $Q$ and denote all objects associated to $Q'$ by $-'$. Let $w\in \W$ be a shortest coset representative in $\W_P\backslash\W/\W_{Q'}$ and $x\in\W_{Q'}$ be a shortest coset representative in $$w^{-1}\WPQ{Q'}{w}w\backslash\W_{Q'}/\W_{Q}=(w^{-1}\W_{P}w\cap \W_{Q'})\backslash\W_{Q'}/\W_{Q}.$$
	Then $\PQ{Q}{wx}\subset \PQ{Q'}{w}$ and the following diagram commutes
	\begin{center}
		\begin{tikzcd}
			PwxQ/Q\arrow[d,"\pi",two heads] &P\times^{\PQ{Q}{wx}}BwxQ/Q\arrow[l,"\operatorname{mult}"',"\sim"] \arrow[r,"\pr{wx}"',two heads]\arrow[two heads, d,"\pi"]&
			P/\PQ{Q}{wx} \arrow[two heads, d,"\pi"] 
				\\
			PwQ'/Q'&P\times^{\PQ{Q'}{w}}BwQ'/Q'\arrow[l,"\operatorname{mult}"',"\sim"]\arrow[r,"\pr{w}'"',two heads]&
			P/\PQ{Q'}{w}
		\end{tikzcd}
	\end{center}
	If $x=e$, the diagram is moreover Cartesian.
\end{lemma}
\begin{lemma}\label{lem:mapsgeomparaindoutofwall}
	In the notation of Lemma \ref{lem:mapsgeomparaindintowall}, denote by $Z$ the pullback of the diagram
	\begin{center}
		\begin{tikzcd}
			Z \arrow[two heads, d,"\pi"]\arrow[hook, r]& 
			G/Q \arrow[two heads, d,"\pi"]\\
			PwQ'/Q'\arrow[hook, r,"h'_w"]&
			G/Q'
		\end{tikzcd}
	\end{center}
	Then  $$Z=\biguplus_{x}PwxQ/Q\subset G/Q$$
	where $x$ runs over the shortest representatives of the double cosets $$w^{-1}\WPQ{Q'}{w}w\backslash\W_{Q'}/\W_{Q}=(w^{-1}\W_{P}w\cap \W_{Q'})\backslash\W_{Q'}/\W_{Q}.$$
\end{lemma}
\begin{corollary}
	In the notation of Lemma \ref{lem:mapsgeomparaindoutofwall} assume additionally that $\W_{Q'}\subset w^{-1}\W_{P}w$, then the following diagram is Cartesian.
	\begin{center}
		\begin{tikzcd}
			PwQ/Q \arrow[two heads, d,"\pi"]\arrow[hook, r,"h_w"]& 
			G/Q \arrow[two heads, d,"\pi"]\\
			PwQ'/Q'\arrow[hook, r,"h'_w"]&
			G/Q'.
		\end{tikzcd}
	\end{center}
\end{corollary}
\subsection[Geometric Parabolic Induction]{Geometric Parabolic Induction and Translation Functors}
In this section we will introduce \emph{geometric parabolic induction} and then study its interaction with the geometric versions of translation functors, i.e. the functors
\begin{center}
	\begin{tikzcd}
		\pi_!:\MTDer{(B)}{G/Q}\arrow[leftrightarrow,r]&\MTDer{(B)}{G/Q'}: \pi^!
	\end{tikzcd}
\end{center}
associated to the projections $G/Q\rightarrow G/Q'$ for $Q\subset Q'$.

There will be two different cases.
Passing \emph{into} a smaller flag variety, i.e. applying the functors $\pi_!$, will \emph{commute} with geometric parabolic induction.
However, the case of passing \emph{out of} a smaller flag variety, i.e. applying $\pi^!$, will be more complicated. 
\begin{definition}
	Let $B\subset Q\subset G$ be a standard parabolic subgroup and let $w\in \W$ be a shortest coset representative in $\W_P\backslash\W/\W_Q$.
	We then call the functor
	\begin{center}
		\begin{tikzcd}[column sep= 2cm]
			\MTDer{(B)}{P/\PQ{Q}{w}}\arrow[r,"\gind{w}"',"\hop_{w,*}\pr{w}^!"]&\MTDer{(B)}{G/Q}
		\end{tikzcd}
	\end{center}
	\textbf{geometric parabolic induction} and denote it by $\gind{w}$. The maps
	\begin{center}
		\begin{tikzcd} 
			P/\PQ{Q}{w}& 
			\arrow[two heads,l,"\pr{w}"']
			PwQ/Q \arrow[hook,r,"\operatorname{h_w}"] & 
			G/Q
		\end{tikzcd}
	\end{center}
	are defined as in Theorem \ref{thm:geomparaindsetup}. 
\end{definition}
\begin{lemma}
	\begin{enumerate}
		\item $\gind{w}$ is well defined, that is, it restricts to stratified mixed Tate motives.
		\item If $w=e$, then  $\gind{e}=\hop_{e,*}$ is weight-exact.
		\item Denote by $i_x:Bx\PQ{Q}{w}/\PQ{Q}{w}\rightarrow P/\PQ{Q}{w}$ and $i'_{xw}:BxwQ/Q\rightarrow G/Q$ the inclusions, then
			$$\gind{w}i_{x,*}\un=i'_{xw,*}\un(l(w))[2l(w)]$$
		where by $\un$ we denote the constant motive on $Bx\PQ{Q}{w}/\PQ{Q}{w}$ and $BxwQ/Q$, respectively.
	\end{enumerate}
\end{lemma}
\begin{proof}
	(1) By Theorem \ref{thm:geomparaindsetup} both $\pr{w}$ and $\hop_w$ are affinely stratified maps, compatible with the Bruhat stratification of $P/\PQ{Q}{w}$ and $G/Q$. The statement follows using Lemma \ref{prop:sixfunctorsandstratifiedmaps}.
	
	(2) In this case $P/\PQ{Q}{w}=P/{P\cap Q}\cong PQ/Q$. Hence $\pr{e}$ is an isomorphism and furthermore $\hop_e$ is a closed embedding and hence weight exact by Proposition \ref{prop:weightexactness}.
	
	(3) By Theorem \ref{thm:geomparaindsetup}, the diagram 
		\begin{center}
			\begin{tikzcd}
				BxwQ/Q \arrow[two heads, d,"\pr{w}"]\arrow[hook,"i''_{xw}", r]& 
				PxwQ/Q \arrow[two heads, d,"\pr{w}"]\arrow[hook,"\hop_{w}", r]&
				G/Q
				\\
				Bx\PQ{Q}{w}/\PQ{Q}{w}\arrow[hook, r,"i_x"]&
				P/\PQ{Q}{w}&
			\end{tikzcd}
		\end{center}
		is Cartesian. Hence by base change and Lemma \ref{lem:sixfunctorsonstrata}
		$$\gind{w}i_{x,*}\un= \gindv{w}i_{x,*}\un=\hop_{w,*}i''_{xw,*}\pr{w}^!\un=i'_{xw,*}\un(l(w))[2l(w)].\qedhere$$
\end{proof}
\begin{theorem}[Geometric parabolic induction and translation into a smaller flag variety]\label{thm:geomparaindintowall}
	Let $B\subset Q\subset Q'$ be standard parabolic subgroups. Denote all objects associated to $Q'$ by $-'$. Let $w$ be a shortest coset representative of $\W_{Q'}\backslash \W/\W_P$. Then the following diagram of functors commutes (up to natural isomorphism).
	\begin{center}
		\begin{tikzcd}[column sep=2 cm]
			\MTDer{(B)}{P/\PQ{Q}{w}} \arrow[r,"\gind{w}"]\arrow[ d,"\pi_*"] &
			\MTDer{(B)}{G/Q} \arrow[ d,"\pi_*"]\\
			\MTDer{(B)}{P/\PQ{Q'}{w}} \arrow[r,"\gindp{w}"]&
			\MTDer{(B)}{G/Q'}
		\end{tikzcd}
	\end{center}
\end{theorem}
\begin{proof}
	Follows immediately from Lemma \ref{lem:mapsgeomparaindintowall}, Lemma \ref{lem:mapsgeomparaindoutofwall} and base change.
\end{proof}
\begin{theorem}[Geometric parabolic induction and translation out of a smaller flag variety]\label{thm:geomparaindoutofwall}
	Let $B\subset Q\subset Q'$ be standard parabolic subgroups. Denote all objects associated to $Q'$ by $-'$. Consider the composition
	\begin{center}
		\begin{tikzcd}[column sep=1.5 cm]
			\MTDer{(B)}{P/\PQ{Q'}{w}} \arrow[r,"\gindp{w}"] &
			\MTDer{(B)}{G/Q'} \arrow[ r,"\pi^!"]&
			\MTDer{(B)}{G/Q}
		\end{tikzcd}
	\end{center}
	Choose an ordering by length $\{x_k\}_{1\leq k\leq n}$ on the set of shortest coset representatives in $$w^{-1}\WPQ{Q'}{w}w\backslash\W_{Q'}/\W_{Q}=(w^{-1}\W_{P}w\cap \W_{Q'})\backslash\W_{Q'}/\W_{Q}.$$
	Then for all $M\in \MTDer{(B)}{P/\PQ{Q}{w}}$, there exists a family of distinguished triangles
	\begin{center}
		\disttriangle{M_{k-1}}{M_k}{\gind{wx_k}\pi_k^!M}
	\end{center}
	in $\MTDer{(B)}{G/Q}$ where 
	$$M_n=\pi^!\gindp{w}M \text{ and } M_0=0$$
	and the right hand side is given by
	\begin{center}
		\begin{tikzcd}[column sep=1.31 cm]
				\MTDer{(B)}{P/\PQ{Q'}{w}} \arrow[ r,"\pi_k^!"]&
				\MTDer{(B)}{P/\PQ{Q}{wx_k}} \arrow[r,"\gind{wx_k}"] &
				\MTDer{(B)}{G/Q}
		\end{tikzcd}
	\end{center}
	and the map $\pi_k$ is induced by the inclusion $Q_{wx_k}\subset \PQ{Q'}{w}$.
\end{theorem}
\begin{proof}
	By Lemma \ref{lem:mapsgeomparaindoutofwall} there is a Cartesian square
	\begin{center}
		\begin{tikzcd}
			Z \arrow[two heads, d,"\pi"]\arrow[hook,"h", r]& 
			G/Q \arrow[two heads, d,"\pi"]\\
			PwQ'/Q'\arrow[hook, r,"\h{w}"]&
			G/Q'
		\end{tikzcd}
	\end{center}
	with  $$Z=\biguplus_{k}Pwx_kQ/Q\subset G/Q.$$
	Denote by 
	\begin{align*}
		i_k :&Pwx_kQ/Q\hookrightarrow Z\\
		i_{\leq k} :&\bigcup_{l\leq k} Pwx_lQ/Q\hookrightarrow Z
	\end{align*} the inclusions. For $N\in \MTDer{(B)}{Z}$ define $N_k\defi i_{\leq k,!}i_{\leq k}^*N$. Then there exists a family of distinguished triangles
	\begin{center}
		\disttriangle{N_{k-1}}{N_{k}}{i_{k,*}i_{k}^!N}
	\end{center}
	with $N_0=0$ and $N_n=N$,
	using the localisation sequence and induction on $n$. Applied to $$N\defi\pi^!\gind{w}M=\pi^!\gindv{w}M=h_*\pi^!\pr{w}^!M$$ we hence obtain distinguished triangles
	\begin{center}
		\disttriangle{h_*M_{k-1}}{h_*M_{k}}{h_*i_{k,*}i_{k}^!\pi^!\pr{w}'^!M}
	\end{center}
	By Lemma \ref{lem:mapsgeomparaindintowall} the right hand side equals 
	$$h_*i_{k,*}i_{k}^!\pi^!\pr{w}'^!M=\gindv{wx_k}\pi_k^!M=\gind{wx_k}\pi_k^!M$$
	and the statement follows.
\end{proof}
\begin{example} Keep the notation from the last Theorem.

(1)	Assume that the set $\{x_k\}=\{1\}$ has just one element. Then there is an isomorphism of functors
$$\pi^!\gindp{w}\cong \gind{w}\pi^!.$$

(2)	Assume that the set $\{x_n\}=\{1,s\}$ has just two elements and let. Then the theorem yields a distinguished triangles
\begin{center}
	\disttriangle{\gind{w}\pi^!M}{\pi^!\gindp{w}M}{\gind{ws}\pi_s^!M}
\end{center}
\end{example}
We discuss the interaction of the geometric version of the wall crossing functors $\theta_s$ for category $\mathcal O$ and geometric parabolic induction. As in the Category $\mathcal O$ case, this will be an essential ingredient in the induction step of our proof that parabolic induction and geometric parabolic induction correspond to each other.
\begin{theorem}[Geometric parabolic induction and wall crossing functors]\label{thm:wallcrossinggeom}
	Let $Q=(B\cup BsB)$. Let $w\in \W_P\backslash\W$ a shortest coset representative and $s\in \W$ a simple reflection with $ws>w$ such that $ws$ is also a shortest coset representative for $\W_P\backslash\W$. Let $\pi:G/B\rightarrow G/Q$.
	Then there is a distinguished triangle of functors
	\begin{center}
		\disttriangle{\gind{w}}{\pi^!\pi_!\gind{w}}{\gind{ws}}
	\end{center}
	from $\MTDer{(B)}{P/B}$ to $\MTDer{(B)}{G/B}$.
	The map on the left hand side is the unit of the adjunction $(\pi_!,\pi^!)$.
\end{theorem}
\begin{proof}
	This is  more or less Theorem \ref{thm:geomparaindintowall} and \ref{thm:geomparaindoutofwall} combined. Consider the Cartesian diagram
	\begin{center}
		\begin{tikzcd}
			PwB/B\uplus PwsB/B \arrow[equal,r,"i\uplus j"] &[-20pt] Z \arrow[two heads, d,"\pi"]\arrow[hook,"h", r]& 
			G/B \arrow[two heads, d,"\pi"]\\
			& PwQ/Q\arrow[hook, r,"\h{w}'"]&
			G/Q
		\end{tikzcd}
	\end{center}
	Then there are the following natural isomorphisms of functors
	\begin{align*}
		\pi^!\pi_!\gind{w}=\pi^!\pi_!\gindv{w}\cong\pi^!\pi_!h_*i_*\pr{w}^!\cong\pi^!\hop_{w,*}'\pi_!i_!\pr{w}^!\cong h_*\pi^!\pi_!i_!\pr{w}^!.
	\end{align*}
	We apply the localization triangle associated to $(i,j)$ to the term on the right hand side and obtain
	\begin{center}
		\disttriangle{h_*i_!i^!\pi^!\pi_!i_!\pr{w}^!}{h_*\pi^!\pi_!i_!\pr{w}^!}{h_*j_*j^!\pi^!\pi_!i_!\pr{w}^!}
	\end{center}
	Since $ws>w$ the map $\pi i: PwB/B\rightarrow PwQ/Q$ is an isomorphism. Hence
	$$h_*i_!i^!\pi^!\pi_!i_!\pr{w}^!=h_*i_!\pr{w}^!=\gind{w}.$$
	Furthermore $ws>w$ implies that $\PQ{Q}{w}=B$, since 
	$$W_P\cap wW_{Q}w^{-1}=W_P\cap w\{1,s\}w^{-1}=\{1\}.$$ 
	By Lemma \ref{lem:mapsgeomparaindintowall} the diagrams
	\begin{center}
		\begin{tikzcd}
			PwB/B \arrow[r,"\pr{w}"',two heads]\arrow[ d,"\pi  i","\wr"']&
			P/B \arrow[equal, d,"\pi'"] 
			\\
			PwQ/Q\arrow[r,"\pr{w}'"',two heads]&
			P/\PQ{Q}{w}=P/B
	\end{tikzcd}\hspace{30pt}
		\begin{tikzcd}
		PwsB/B \arrow[r,"\pr{ws}"',two heads]\arrow[two heads, d,"\pi  j"]&
		P/B \arrow[equal, d,"\pi'"] 
		\\
		PwQ/Q\arrow[r,"\pr{w}'"',two heads]&
		P/\PQ{Q}{w}=P/B
		\end{tikzcd}
	\end{center}
	commute, and the left hand diagram is cartesian. Hence 
	$$h_*j_*j^!\pi^!\pi_!i_!\pr{w}^!\cong h_*j_*j^!\pi^!\pr{w}'^{!}\pi'_! \cong h_*j_*\pr{ws}^{!}\pi'^!\pi'_!\cong\gindv{ws}=\gind{ws}.$$
	Putting everything together, our distinguished triangle reads
	\begin{center}	
		\disttriangle{\gind{w}}{\pi^!\pi_!\gind{w}}{\gind{ws}}
	\end{center}
	But a priori the first map is induced by the counit of the adjunction $(i_!,i^!)$. That this coincides unit of the adjunction $(\pi_!,\pi^!)$ follows from the following general Lemma.
\end{proof}
\begin{lemma}
	Let $i:Z\leftrightarrow X:\pi$ be two morphisms in $\Var(\C)$, such that $i$ is a closed embedding and $\pi i=\id_Z$. Then the following diagram of functors $\Motives{Z}\rightarrow\Motives{X}$ commutes
	\begin{center}
		\begin{tikzcd}
			i_!\arrow[r]\arrow[d,equal]&\pi^!\pi_!i_! \arrow[d,equal]
			\\
			i_!\arrow[r]\arrow[d,equal]& \pi^!\arrow[d,equal]
			\\
			i_!i^!\pi^!\arrow[r]& \pi^!
		\end{tikzcd}
	\end{center}
	where the top row is the unit of the adjunction $(\pi_!,\pi^!)$ and the bottom row the counit of $(i_!,i^!)$.
\end{lemma}
\begin{proof}
	The two base change morphisms (called exchange morphisms in \cite{CD}) $i_!\id^*\rightarrow\pi^!\id_*$ associated to the diagram
	\begin{center}
		\begin{tikzcd}
			Z\arrow[r,"i"]\arrow[d,"\id"]&X\arrow[d,"\pi"]\\
			Z\arrow[r,"\id"]& Z
		\end{tikzcd}
	\end{center}
	coincide.
\end{proof}
The theorem implies that for every $M\in\MTDer{(B)}{P/B}$ there is an isomorphism 
$$\gind{ws}M\cong \operatorname{Cone}(\gind{w}M\rightarrow\pi^!\pi_!\gind{w}M)$$ in  $\MTDer{(B)}{G/B}$, where by $\operatorname{Cone}$ we denote the mapping cone. In general however, mapping cones are \emph{not functorial}. But as in the proof of Theorem \ref{thm:sindregular}, in our particular situation Lemma \ref{lem:triangulatedfunctor} applies and the mapping cone is indeed functorial.
\begin{lemma}\label{lem:nobadmorph}
	For all $M,N\in\MTDer{(B)}{P/B}$ we have
	$$\Hom{\Motives{G/B}}{\gind{w}(M)}{\gind{ws}(N)}=0.$$
\end{lemma}
\begin{proof}
	This is simply a matter of their support. Let 
	\begin{center}
		\begin{tikzcd}
			U\defi PwsB/B\arrow[r,"j",hook] &Z\defi U\cup W& W\defi PwB/B\arrow[l,"i"',hook']
		\end{tikzcd}
	\end{center} and denote by $k$ the inclusion of $Z$ in $X$. Notice that $U$ is open in $Z$.
	\begin{align*}
	&\Hom{\Motives{G/B}}{\gind{w}M}{\gind{ws}N}\\
	&=\Hom{\Motives{G/B}}{\gindvd{ws}N}{\gindvd{w}M}~~ \text{ (duality)}\\
	&=\Hom{\Motives{Z}}{k^*\gindvd{ws}N}{k^*\gindvd{w}M}~~ \text{ (support $\subseteq Z$) }\\
	&=\Hom{\Motives{Z}}{j_!j^*k^*\gindvd{ws}N}{i_!i^*k^*\gindvd{w}M}~~ \text{ (support $\subseteq W$, resp. $U$) }\\
	&=\Hom{\Motives{Z}}{j^*k^*\gindvd{ws}N}{j^*i_!i^*k^*\gindvd{w}M}~~ \text{ (adjunction and $j^*=j^!$)}\\
	&=0 ~~ \text{(since $j^*i_!=0$)}
	\end{align*}
	and the claim follows.
\end{proof}
\begin{corollary}\label{cor:coneisnaturalgeomparind}
	There is a natural equivalence of functors
	$$\gind{ws}\cong \operatorname{Cone}(\gind{w}\rightarrow\pi^!\pi_!\gind{w}).$$
\end{corollary}
\begin{proof}
	Follows from Theorem \ref{thm:wallcrossinggeom}, Lemma \ref{lem:nobadmorph} and Lemma \ref{lem:triangulatedfunctor}.
\end{proof}
Up to direct sums and shifts, geometric parabolic induction for partial flag varieties $G/Q$ can be expressed in terms of geometric parabolic induction for the regular flag variety $G/B$ and geometric translation functors.
\begin{theorem}\label{thm:fromsingtoreggeom}
	Let $B\subset Q\subset G$ and $w$ be a shortest representative of a coset in $\W_P\backslash\W/\W_{Q}$. Let $\pi:G/B\rightarrow G/Q$ and $\pi':P/B\rightarrow P/\PQ{Q}{w}$ the projections.
	Then there is a natural equivalence of functors
	\begin{align*}
\pi_*\gind{w}\pi'^!\cong\bigoplus_{x\in \WPQ{Q}{w}}\gindp{w}(l(x))[2l(x)]:\\\MTDer{(B)}{P/\PQ{Q}{w}}\rightarrow\MTDer{(B)}{G/Q}
	\end{align*}
\end{theorem}
\begin{proof}
	By Theorem \ref{thm:geomparaindintowall} we have
	$$\pi_*\gind{w}\pi'^!=\gind{w}\pi'_*\pi'^!.$$
	Now we argue as in \cite[Lemma 3.5.4]{BGS}. The decomposition theorem yields
	$$\pi'_*\un_{P/B}\cong\bigoplus_{x\in \WPQ{Q}{w}}\un_{P/\PQ{Q}{w}}(-l(x))[-2l(x)]$$
	and using Verdier duality we get
	\begin{align*}
		\pi'_*\pi'^!&\cong\pi'_*\iHom{P/B}(\un,\pi'^!(-))\cong\iHom{P/\PQ{Q}{w}}(\pi'_*\un,-)\\&\cong\bigoplus_{x\in \WPQ{Q}{w}}\iHom{P/\PQ{Q}{w}}(\un_{P/\PQ{Q}{w}}(-l(x))[-2l(x)],-)\\&\cong\bigoplus_{x\in \WPQ{Q}{w}}\id(l(x))[2l(x)].
	\end{align*}
	The statement follows. 
\end{proof}
\begin{remark}
	Let us explain why we were allowed to use the decomposition theorem for perverse sheaves here. After all, the decomposition theorem is a statement about constructible sheaves and not motives.
	
	By \cite{brad}, for all $X\in\Var(\C)$, there is a Hodge realization functor
	$$\operatorname{Real}_H: \Motives{X}\rightarrow\Der(X(\C),\C)$$
	from motives on $X$ into the derived category of sheaves on $X(\C)$ (equipped with its metric topology). This is compatible with the six operations.
	
	For an affinely Whitney-Tate stratified variety $(X,S)$, $\operatorname{Real}_H$ restricts to a functor
	$$\operatorname{Real}_H: \MTDer{S}{X}\rightarrow\Derb_S(X,\C).$$
	By \cite[Theorem 11.3]{SoeWe} it induces isomorphisms
	$$\bigoplus_{i\in\Z}\Hom{\MTDer{S}{X}}{M}{N(i)}\stackrel{\sim}{\rightarrow}\Hom{\Derb_S(X,\C)}{\operatorname{Real}_H(M)}{\operatorname{Real}_H(N)}$$
	which are compatible with composition.
	In the notation of the last proof, the decomposition theorem yields
	$$\pi'_*\underline{\C}_{P/B}\cong\bigoplus_{x\in \WPQ{Q}{w}}\underline{\C}_{P/\PQ{Q}{w}}[-2l(x)]\in\Derb_{(B)}(P/\PQ{Q}{w},\C).$$
	Actually, the decomposition theorem in its full strength is not needed here. We can also use that $\pi$ is a fibration with typical fibre $\PQ{Q}{w}/B$, and apply the Leray-Serre spectral sequence, which degenerates on page two by parity vanishing: The cohomology of both $\PQ{Q}{w}/B$ and $P/B$ is concentrated in even degrees. This implies
	$$\pi'_*\underline{\C}_{P/B}=\bigoplus_{i\in\Z}H^{i}(\PQ{Q}{w}/B,\C)\otimes\underline{\C}_{P/\PQ{Q}{w}}[-i].$$
	Combined with the equality $$H^{i}(\PQ{Q}{w}/B,\C)=\bigoplus_{\mathclap{\substack{x\in \WPQ{Q}{w} \\ 2l(x)=i}}}\C$$
	we obtain our statement.
	By \cite[Lemma 6.6.]{SoeWe}, $\pi'_*\underline{\C}_{P/B}$ and $\un_{P/\PQ{Q}{w}}$ are pointwise pure of weight zero. This implies 
	\begin{align*}
		\Hom{\MTDer{S}{X}}{\pi'_*\un_{P/B}}{\un_{P/\PQ{Q}{w}}(n)[2n](i)}&=0\text{ and}\\
		\Hom{\MTDer{S}{X}}{\un_{P/\PQ{Q}{w}}(n)[2n](i)}{\pi'_*\un_{P/B}}&=0
	\end{align*} 
	for all $i\neq 0$, $n\in\Z$, using \cite[Corollary 6.3]{SoeWe}. Hence there are isomorphisms	
	\begin{align*}
	&\Hom{\MTDer{S}{X}}{\pi'_*\un_{P/B}}{\un_{P/\PQ{Q}{w}}(n)[2n]}\\&\cong\Hom{\Derb_{(B)}(P/\PQ{Q}{w},\C)}{\pi'_*\underline{\C}_{P/B}}{\underline{\C}_{P/\PQ{Q}{w}}(n)[2n]}\text{ and}\\
	&\Hom{\MTDer{S}{X}}{\un_{P/\PQ{Q}{w}}(n)[2n]}{\pi'_*\un_{P/B}}\\&\cong\Hom{\Derb_{(B)}(P/\PQ{Q}{w},\C)}{\underline{\C}_{P/\PQ{Q}{w}}(n)[2n]}{\pi'_*\underline{\C}_{P/B}}.
	\end{align*}
	So we can transport the projections and embeddings from the direct sum decomposition of $\pi'_*\underline{\C}_{P/B}$ to $\pi'_*\un_{P/B}$ and obtain $$\pi'_*\un_{P/B}\cong\bigoplus_{x\in \WPQ{Q}{w}}\un_{P/\PQ{Q}{w}}(-l(x))[-2l(x)].$$
	Admittedly, this argument is awkward, and there should be a much more direct proof using a motivic version of the Leray-Serre spectral sequence. 
\end{remark}
\subsection{Geometric Parabolic Induction and Soergel Modules}
The category of stratified mixed Tate motives on a flag variety has a completely combinatorial description as the bounded homotopy category of Soergel modules. In this section we aim to give a description of geometric parabolic induction on the level of Soergel modules, i.e. fill out the question mark in the diagram
\begin{center}
	\begin{tikzcd}
		\MTDer{(B)}{P/\PQ{Q}{w}}\arrow[r,"\gind{w}"]\arrow[d,"\wr"]&\MTDer{(B)}{G/Q}\arrow[d,"\wr"]\\
		\Hotb(H(P/\PQ{Q}{w})\Smodules^{\Z,ev})\arrow[r,"?"]&\Hotb(H(G/Q)\Smodules^{\Z,ev})
	\end{tikzcd}
\end{center}
Let $(X,S)$ be an affinely Whitney-Tate stratified variety. Then the  \emph{hypercohomology functor} is defined by 
$$\Hyp: \MTDer{S}{X}\rightarrow H(X)\modules^{\Z\times\Z}, M\mapsto\bigoplus_{i,j\in\Z}\Hom{\Motives{X}}{\un_X}{M(i)[j]},$$
where $H(X)\defi \Hyp(\un_X)$.
\begin{theorem}[Erweiterungssatz]
	Let $X\in \Var(\C)$ be a partial flag variety. Then the hypercohomology functor 
	\begin{center}
		\begin{tikzcd}
			\Hyp: \MTDer{(B)}{X}_{w=0}\arrow[r,hook,"\sim"]&H(X)\modules^\Z
		\end{tikzcd}
	\end{center}
 	is fully faithful on weight zero stratified mixed Tate motives.
\end{theorem}
\begin{proof}
	See \cite{ginsburg1991perverse} for a proof using mixed Hodge modules. All the proof really relies on is a six functor formalism and a theory of weights. It hence also holds in our setting as spelled out in \cite[Theorem 8.4]{SoeWe}.
\end{proof}
\begin{definition}
	The modules in the essential image of $\Hyp$ are called (graded) Soergel modules, so that $\Hyp$ induces an equivalence of categories:
	\begin{center}
		\begin{tikzcd}
			\Hyp: \MTDer{(B)}{X}_{w=0}\arrow[r,"\sim"]&H(X)\Smodules^{\Z,ev}
		\end{tikzcd}
	\end{center}
	between weight zero stratified mixed Tate motives and the category of evenly graded Soergel modules over $H(X)$ denoted $H(X)\Smodules^{\Z,ev}$.
\end{definition}
\begin{remark}
	Let $X=G/Q$. Abbreviate $C=H(G/Q)$. Then the category $C\Smodules^\Z$ is generated by modules of the form
	$$C\otimes_{C^{s_n}}\dots C\otimes_{C^{s_1}}\C$$
	for simple reflections $s_i$, with respect to finite direct sums, taking direct summands, shifts and isomorphism. This corresponds to the next Lemma and the fact that for regular $X=G/B$ all simple perverse motives in $\MTDer{(B)}{G/B}$ can appear as (shifts of) direct summands in the motives modules
	$$\pi_n^!\pi_{n,!}\cdots\pi_1^!\pi_{1,!}i_{pt,!}\un$$
	where $\pi_i: G/B\rightarrow G/(B\cup Bs_iB)$ is the projection and $i_{pt}: B/B\rightarrow G/B$ the inclusion of the point.
\end{remark}
\begin{lemma}[Theorem 14 \cite{Soe90}]
	Let $B\subset Q'\subset Q\subset G$ be parabolic subgroups and let $\pi: G/Q'\rightarrow G/Q$ be the projection. Then there are natural isomorphisms of functors
	\begin{align*}
	\Hyp\pi^*&\cong H(G/Q')\otimes_{H(G/Q)}\Hyp\text{ and}\\
	\Hyp\pi_*&\cong \res_{H(G/Q')}^{H(G/Q)}\Hyp.
	\end{align*}
\end{lemma}
\begin{corollary}[Corollary 9.4 \cite{SoeWe}]
	Let $X=G/Q$ be a flag variety. There are equivalences of triangulated categories
	\begin{center}
		\begin{tikzcd}[column sep= 20]
			 \MTDer{(B)}{X}\arrow[r,"\sim"',"\Delta"]&\Hotb(\MTDer{(B)}{X}_{w=0})\arrow[r,"\sim"',"\Hyp"]&\Hotb(H(X)\Smodules^{\Z,ev})
		\end{tikzcd}
	\end{center}
	where $\Delta$ denotes the tilting equivalence, see Theorem \ref{thm:tiltingformotives}.
\end{corollary}
\begin{proof}
	\sloppy The proof uses that for partial flag varieties all objects in $\MTDer{(B)}{X}_{w=0}$ are additionally pointwise pure.
\end{proof}
There is a completely explicit description of the cohomology ring of flag varieties, due to Borel. In this section we describe how this is compatible with respect to the inclusion $P/Q_e\hookrightarrow G/Q$.
\begin{lemma} \label{lem:cohomologyringandcoinvregular}Denote by $X(T)$ the character lattice of the torus $T\subset G$ and by $$S=\Sym{X(T)\otimes_\Z\C}$$
	the symmetric algebra of its complexification.
	Then following diagram of short exact sequences commutes
	\begin{center}
		\begin{tikzcd}[row sep=0.6cm]
			0\arrow[r] & S(S^{\W}_+)\arrow[hook,r]\arrow[hook,d]&S\arrow[r,two heads,"c_1"]&H(G/B)\arrow[r]\arrow[d,"\hop_e^*"]  & 0\\
			0\arrow[r]& S(S^{\W_P}_+)\arrow[hook,r] &S\arrow[r,two heads,"c_1"] & H(P/B)\arrow[r] & 0\\
		\end{tikzcd}
	\end{center}
	where $c_1$ denotes the map induced by the first Chern class of a line bundle induced by a character of $T$.
\end{lemma}
\begin{proof}
	$H(G/B)$ and $H(P/B)$ are the de Rham cohomology groups of $G/B$ and $P/B$ and this is the classical Borel image. The diagram commutes since pullback of line bundles and taking Chern classes commutes.
\end{proof}
\begin{lemma}\label{lem:cohomologyringandcoinv}
	Let $Q\subset G$ be a standard parabolic. Then the following diagram commutes
	\begin{center}
		\begin{tikzcd}
			\CoInvi{G}{Q}\defi(S/S(S^{\W}_+))^{\W_Q}\arrow[rightarrow,r]\arrow[d,"\wr"] & (S/S(S^{\W_P}_+))^{\WPQ{Q}{e}}\arrow[d,"\wr"]\defi \CoInvi{\PQ{Q}{e}}{P}\\
			H(G/Q) \arrow[two heads,r,"\hop_e^*"] & H(P/\PQ{Q}{e})
		\end{tikzcd}
	\end{center}
\end{lemma}
\begin{proof}
	The isomorphism $C^Q_G\cong H(G/Q)$ is established by identifying the image of the injection $$\pi^*: H(G/Q)\hookrightarrow H(G/B)$$ with the $\W_Q$-invariants in $S/S(S^{\W}_+)$, where $\pi:G/B\rightarrow G/Q$ denotes the projection, see \cite{bernstein1973schubert}. The same holds for $H(P/\PQ{Q}{e})$ and the statement follows by $\hop_e\pi=\pi \hop_e$ and Lemma \ref{lem:cohomologyringandcoinvregular}.
\end{proof}
We are now able to prove how geometric parabolic induction $\gind{e}=\h{e,*}$ interact. This is the base case of the inductive proof of the general case $\gind{w}$.
\begin{theorem}\label{thm:inductionbasecasegeom}
	Let $Q\subset G$ be a standard parabolic. The following diagram of functors commutes up to natural isomorphism
	\begin{center}
		\begin{tikzcd}[column sep=2cm]
			\MTDer{(B)}{P/{\PQ{Q}{e}}} \arrow[r,"\gind{e}=\h{e,*}"]\arrow[d,"\Hyp"] & \MTDer{(B)}{G/Q}\arrow[d,"\Hyp"]\\
			\CoInvi{\PQ{Q}{e}}{P}\modules^{\Z}\arrow[rightarrow,r,"\operatorname{Res}^{\CoInvi{G}{Q}}_{\CoInvi{\PQ{Q}{e}}{P}}"] & \CoInvi{G}{Q}\modules^{\Z}
		\end{tikzcd}
	\end{center}
\end{theorem}
\begin{proof}
	Let $M\in \MTDer{(B)}{P/{\PQ{Q}{e}}}$. By definition we have
	\begin{align*}
	\Hyp(h_{e,*}M) &=\bigoplus_{i,j\in\Z}\Hom{\Motives{G/Q}}{\un_{G/Q}}{h_{e,*}M(i)[j]}\\
	&=\bigoplus_{i,j\in\Z}\Hom{\Motives{P/{\PQ{Q}{e}}}}{h_{e}^*\un_{G/Q}}{M(i)[j]}\\
	&=\bigoplus_{i,j\in\Z}\Hom{\Motives{P/{\PQ{Q}{e}}}}{\un_{P/{\PQ{Q}{e}}}}{M(i)[j]}\\
	&=\Hyp(M)
	\end{align*}
	and the statement follows from Lemma \ref{lem:cohomologyringandcoinv}.
\end{proof}
\begin{corollary}\label{cor:inductionbasecasegeom}Let $Q\subset G$ be a standard parabolic.
	Then the following diagram of functors commutes up to natural isomorphism
	\begin{center}
		\begin{tikzcd}
			\MTDer{(B)}{P/{\PQ{Q}{e}}} \arrow[r,"\gind{e}=\h{e,*}"]\arrow[d,"\wr","\Delta"'] & \MTDer{(B)}{G/Q}\arrow[d,"\Delta"',"\wr"]\\
			\Hotb(\MTDer{(B)}{P/Q_w}_{w=0})\arrow[d,"\wr","\Hyp"']&\Hotb(\MTDer{(B)}{G/Q}_{w=0})\arrow[d,"\wr"',"\Hyp"]\\
			\Hotb(C^{\PQ{Q}{e}}_P\Smodules^{\Z,ev})\arrow[rightarrow,r,"\operatorname{Res}^{\CoInvi{G}{Q}}_{\CoInvi{\PQ{Q}{e}}{P}}"] & \Hotb(C^{Q}_G\Smodules^{\Z,ev})
		\end{tikzcd}
	\end{center}
\end{corollary}
\begin{proof}
	Follows from Theorem \ref{thm:inductionbasecasegeom}. We use that $\h{e}$ is a closed embedding, and hence $\h{e,*}=\h{e,!}$ acts on the homotopy categories of weight zero motives by pointwise application, see Theorem \ref{thm:tiltingformotivesandfunctors}.
\end{proof}
Recall that $B\subset P\subset G$ was a parabolic subgroup. Abbreviate
$$\Coi=\CoInv{G}{B}= H(G/B) \cong \Sym{X(T)\otimes_\Z\C}/(\Sym{X(T)\otimes_\Z\C}^{\W}_+).$$
This is a graded ring, living in \emph{even} and positive degrees. Denote for a graded module $M$ its $n$-th shift by $M\langle n\rangle$, such that
$$(M\langle n\rangle)^i=M^{i+n}.$$
For a simple reflection $s$, denote by $\Coi^s$ the $s$-invariants. Then $\Coi^s\subset \Coi$ is a Frobenius extension, and we denote by
$$\Rouqs{s}\defi\dots\rightarrow0\rightarrow \Coi\rightarrow \Coi\otimes_{\Coi^s}\Coi\langle 2\rangle\rightarrow0\rightarrow\dots$$
the complex of graded Soergel bimodules over $\Coi$ known as \emph{Rouquier complex}. Here $\Coi\otimes_{\Coi^s}\Coi\langle 2\rangle$ lives in cohomological degree 0, and the map is the unit of the adjunction between $\res_{\Coi}^{\Coi^s}(-\langle 1\rangle)$ and $\Coi\otimes_{\Coi^s}-\langle 1\rangle$. For a reduced expression $w=s_n\cdots s_1$ of $w\in\W$, we define a complex of graded Soergel bimodules by
$$\underbar{R}_{\underline{w}}\defi \Rouqs{s_1}\otimes_{\Coi}\cdots\otimes_{\Coi} \Rouqs{s_n}.$$
Abbreviate
$$\res\defi \res_{\CoInvi{P}{B}}^{\Coi}.$$
In the rest of this section we will---among other things---prove that on the level of graded Soergel modules, geometric parabolic induction for a regular flag variety
$$\gind{w}:\MTDer{(B)}{P/B}\rightarrow\MTDer{(B)}{G/B}$$
is given by the functor
$$\sind{w}{}\defi\Rouq{\underline{w}}\otimes_{\Coi}\res(-).$$
\begin{theorem}\label{thm:sindregulargeom}
Let $w$ be a shortest coset representative in $\W_{P}\backslash\W_G$.
Then the following diagram of functors commutes up to natural isomorphism 
	\begin{center}
		\begin{tikzcd}[column sep=3cm]
			\MTDer{(B)}{P/B}\arrow[r,"\gind{w}"]\arrow[d,"\wr","\Delta"']&\MTDer{(B)}{G/B}\arrow[d,"\wr"',"\Delta"]\\
			\Hotb(\MTDer{(B)}{P/B}_{w=0})\arrow[d,"\wr","\Hyp"']&\Hotb(\MTDer{(B)}{G/B}_{w=0})\arrow[d,"\wr"',"\Hyp"]\\
			\Hotb(\CoInvi{P}{B}\Smodules^{\Z,ev})\arrow[r,"\sind{w}{}"]&\Hotb(\Coi\Smodules^{\Z,ev}).
		\end{tikzcd}
	\end{center}
\end{theorem}

\begin{proof}The proof mainly relies on Theorem \ref{thm:wallcrossinggeom} and an induction on $l(w)$.
	First assume that $l(w)=0$, then $w=e$ and the statement is Corollary \ref{cor:inductionbasecasegeom}.
	
	Now let $ws>w$ with both $ws$ and $w$ shortest representatives in $\W_P\backslash\W$. Assuming that the statement holds for $w$, we show that it holds for $ws$. 
	
	Denote by $\Delta$ the tilting equivalence. Let $\pi: G/B\rightarrow G/(B\cup BsB)$ be the projection. Let $M\in\MTDer{(B)}{P/B}$. 
	We have the following diagram of distinguished triangles:
	\begin{center}
		\begin{tikzcd}
			\Hyp\Delta\gind{w} M
			\arrow[r]\arrow[d,equal]&
			\Hyp\Delta\pi^!\pi_!\gind{w} M
			\arrow[r]\arrow[d,"\wr","(1)"']&
			\Hyp\Delta\gind{ws} M
			\arrow[r,"+1"]\arrow[d,equal]&~
			\\
			\Hyp\Delta\gind{w} M
			\arrow[r,"(*)"]\arrow[d,"\wr","(2)"']&
			\Coi\otimes_{\Coi^s}\Hyp\Delta\gind{w} M\langle2\rangle
			\arrow[r]\arrow[d,"\wr","(2)"']&
			\Hyp\Delta\gind{ws} M
			\arrow[r,"+1"]\arrow[d,equal]&~
			\\
			\sind{w}{}\Hyp\Delta M
			\arrow[r]&
			\Coi\otimes_{\Coi^s}\sind{w}{}\Hyp\Delta M\langle2\rangle
			\arrow[r]&
			\Hyp\Delta\gind{ws} M
			\arrow[r,"+1"]&~
		\end{tikzcd}
	\end{center}
	The first triangle is given by Theorem \ref{thm:wallcrossinggeom}.
	
	(1) Since $\pi^!\pi_!$ commutes with $\Delta$ by \ref{thm:tiltingformotivesandfunctors}. On Soergel modules $\pi^!\pi_!$ is given by $\Coi\otimes_{\Coi^s}\langle 2\rangle$, see \cite[Korollar 2]{Soe90}.
	
	(2) This is the induction hypothesis.
	
	($*$) This is given by the adjunction homomorphism by Theorem \ref{thm:wallcrossinggeom}.
	
	We hence have the following isomorphism
	\begin{align*}
	\Hyp\Delta\gind{w} M\cong
	& \operatorname{Cone}(\sind{w}{}\Hyp\Delta M\rightarrow
	\Coi\otimes_{\Coi^s}\sind{w}{}\Hyp\Delta M\langle 2\rangle)\\
	=&\Rouq{s}\otimes_C\sind{w}{}\Hyp\Delta M\\
	=&\sind{ws}{}\Hyp\Delta M
	\end{align*}
	where by $\operatorname{Cone}$ we denote the mapping cone.
	This is indeed a \emph{natural} isomorphism by the discussion in Corollary \ref{cor:coneisnaturalgeomparind} and Lemma \ref{lem:triangulatedfunctor}.
\end{proof}
Now let $B\subset Q\subset G$ be another parabolic subgroup. Let $w$ be a shortest coset representative in $\W_P\backslash\W/\W_Q$. Recall that $\W_{Q,w}=\W_P\cap w\W_Qw^{-1}$ and let $m=|\W_{Q,w}|$. 
Then there are natural maps
$$\CoInvi{G}{Q}\rightarrow\Coi=\CoInvi{G}{B}\leftarrow\CoInvi{P}{B}\leftarrow\CoInvi{P}{Q_w}$$
On the level of graded Soergel modules, geometric parabolic induction for partial flag varieties (or better an $m$-fold direct sum of shifted copies of it)
$$\bigoplus_{x\in \W_{P,w}}\gind{w}(-)(-l(x))[-2l(x)]
:\MTDer{(B)}{P/Q_w}\rightarrow\MTDer{(B)}{G/Q}$$
is given by the functor
$$\sindn{w}{Q}\defi\res_{\Coi}^{\CoInvi{G}{Q}}\Rouq{\underline{w}}\otimes_{\Coi}\CoInvi{P}{B}\otimes_{\CoInvi{P}{Q_w}}(-).$$
\begin{theorem}\label{thm:gindandsindsingular}
	let $B\subset Q\subset G$ be another parabolic subgroup. Let $w$ be a shortest coset representative in $\W_P\backslash\W/\W_Q$. Let $w=s_n\cdots s_1$ a reduced expression. 
	Abbreviate
	$$\overline{\operatorname{GInd}}_w=\bigoplus_{x\in \WPQ{Q}{w}}\gind{w}(-)(-l(x))[-2l(x)]$$Then the following diagram of functors commutes up to natural isomorphism 
	\begin{center}
		\begin{tikzcd}[column sep=3cm]
				\MTDer{(B)}{P/Q_w}\arrow[r,"\overline{\operatorname{GInd}}_w"]\arrow[d,"\wr","\Delta"']&\MTDer{(B)}{G/Q}\arrow[d,"\wr"',"\Delta"]\\
				\Hotb(\MTDer{(B)}{P/Q_w}_{w=0})\arrow[d,"\wr","\Hyp"']&\Hotb(\MTDer{(B)}{G/Q}_{w=0})\arrow[d,"\wr"',"\Hyp"]\\
				\Hotb(\CoInvi{P}{Q_w}\Smodules^{\Z,ev})\arrow[r,"\sindn{w}{Q}"]&\Hotb(\CoInvi{G}{Q}\Smodules^{\Z,ev}).
		\end{tikzcd}
	\end{center}
\end{theorem}
\begin{proof}
	Denote by $\pi: G/B\rightarrow G/Q$ and $\pi':P/B\rightarrow P/{Q_w}$ the projection. Then on the one hand we have
	$$\pi_!\gind{w}\pi'^!\cong\gind{w}\pi'_!\pi'^!\cong\overline{\operatorname{GInd}}_w=\bigoplus_{x\in \WPQ{Q}{w}}\gind{w}(-)(l(x))[2l(x)]$$
	by Theorem \ref{thm:fromsingtoreggeom}. On the other hand
	\begin{align*}
	\Hyp\Delta\pi_*\gind{w}\pi'^*&\cong\res_{\Coi}^{\CoInvi{G}{Q}}\Hyp\Delta\gind{w}\pi'^*\\
	&\cong\res_{\Coi}^{\CoInvi{G}{Q}}\Rouq{\underline{w}}\otimes_{\Coi}\Hyp\Delta\pi'^*\\
	&\cong\res_{\Coi}^{\CoInvi{G}{Q}}\Rouq{\underline{w}}\otimes_{\Coi}\CoInvi{P}{B}\otimes_{\CoInvi{P}{Q_w}}\Hyp\Delta\\
	&=\sindn{w}{Q}\Hyp\Delta
	\end{align*}
	where we use that $\pi_*$ and $\pi'_*$ commute with tilting, since $\pi'$ is proper and smooth (Theorem \ref{thm:tiltingformotivesandfunctors}), and that under the hypercohomology functor $\Hyp$, $\pi_*$ corresponds to $\res_{\Coi}^{\CoInvi{G}{Q}}$ and $\pi^*$ to $\Coi\otimes_{\CoInvi{G}{Q}}$, see \cite[Theorem 14]{Soe90}. 
	Now $\pi_!=\pi_*$ since $\pi$ is proper and $\pi'^*=\pi'^!(-d)[-2d]$ since $\pi'$ is smooth, where $d$ denotes the relative dimension of $\pi'$. But $d$ is exactly the length of the longest word in $\WPQ{Q}{w}$, hence
	\begin{align*}
	\pi_*\gind{w}\pi'^*&\cong\bigoplus_{x\in \WPQ{Q}{w}}\gind{w}(-)(l(x)-d)[2l(x)-2d]\\
	&=\bigoplus_{x\in \WPQ{Q}{w}}\gind{w}(-)(-l(x))[-2l(x)]
	\end{align*}
	and the statement follows.
\end{proof}

\section{Main Results}
\subsection{Setup}
We recall and compare some of the notations of Section 2 and 3.
Let $\g\supset\borel\supset\cartan$ be a reductive Lie algebra with Borel and Cartan subalgebra. Denote by $G\supset B\supset T$ a \emph{Langlands dual} algebraic group over $\C$, i.e a group such that the root system with simple roots associated to $\operatorname{Lie}(G)\supset \operatorname{Lie}(B) \supset \operatorname{Lie}(T)$ is dual to the one of $\g\supset\borel\supset\cartan$. The Weyl group $\W$ and simple roots $\mathcal S$ corresponding to $\g\supset \borel$ and $G\supset B$ are hence identified. Now let 
\begin{align*}
	\borel&\subset\para\twoheadrightarrow\levi\\ 
	B&\subset P\twoheadrightarrow L
\end{align*}
be corresponding parabolic subgroups/algebras with their Levi factor, i.e. $\W_P=\Wi$.
Let $\lambda\in\cartan^*$ be a dominant integral weight. Then the stabilizer of $\lambda$ with respect to the dot action $\W_{\g,\lambda}$ is generated by simple roots. Hence $\lambda$ corresponds to a standard parabolic subgroup
$$B\subset Q\subset G$$
with  $\W_{\g,\lambda}=\W_Q$. We also have equalities $\WPQ{Q}{w}=\W_P\cap w\W_Qw^{-1}=\W_{\levi,w\cdot\lambda}$, for $w\in\W$ shortest coset representatives of $\W_P\backslash\W/\W_Q$.

We can naturally identify
\begin{align*}
H(P/Q_w)=\CoInv{P}{Q_w}&=(\Sym{X(T)\otimes_Z\C}/(\Sym{X(T)\otimes_Z\C}^{\W_P}_+))^{\WPQ{Q}{w}}\\&=(\Sym{\cartan}/(\Sym{\cartan}^{\Wi}_+))^{\W_{\levi,w\cdot\lambda}}\\&=\CoInv{\levi}{w\cdot\lambda}=\End{\levi}(\aP{\levi}{w\cdot\lambda})
\end{align*}
 and similarly
\begin{align*}
H(G/Q)=\CoInv{G}{Q}&=(\Sym{X(T)\otimes_Z\C}/(\Sym{X(T)\otimes_Z\C}^{\W}_+))^{\W_Q}\\&=(\Sym{\cartan}/(\Sym{\cartan}^{\Weyl}_+))^{\W_{\g,\lambda}}\\&=\CoInv{\g}{\lambda}=\End{\g}(\aP{\g}{w\cdot\lambda})
\end{align*}
where we use the natural identification $X(T)\otimes_\Z\C=\cartan$.

Furthermore, their categories of (graded) Soergel modules, which is defined as the essential image of projective modules and weight zero stratified mixed Tate motives under Soergel's functor $\mathbb V$ and the hypercohomology functor $\Hyp$, respectively, coincide. By this we mean, that functor $v$, forgetting the grading, restricts to a functor
\begin{align*}
v:&H(G/Q)\Smodules^{\Z,ev}\rightarrow \End{\levi}(\aP{\g}{\lambda})\Smodules\\
v:&H(P/Q_w)\Smodules^{\Z,ev}\rightarrow \End{\levi}(\aP{\levi}{w\cdot\lambda})\Smodules
\end{align*}
and every module on the right hand side can be lifted, i.e. has a preimage under $v$. 
\subsection{Geometric Parabolic Induction and Parabolic Induction}
Combining the results from Section 2 and 3, we obtain our main theorem.
\begin{theorem}\label{thm:main}
	Let $\lambda\in\cartan^*$ be a dominant integral weight and $Q\subset G$ the corresponding standard parabolic subgroup. Let $w\in \W$ be a shortest coset representative in $\W_P\backslash\W/W_Q=\Wi\backslash\Weyl/\W_{\g,\lambda}$. Let $n=|\W_{\levi,w\cdot\lambda}|$ and
	\begin{align*}
	\overline{\operatorname{GInd}}_w&=\bigoplus_{x\in \WPQ{Q}{w}}\gind{w}(-)(-l(x))[-2l(x)]
	\end{align*}
	Then the following diagram commutes up to natural isomorphism
	\begin{center}
		\begin{tikzcd}
		\MTDer{(B)}{P/Q_w}\arrow[r,"\overline{\operatorname{GInd}}_w"]\arrow[ddddd,"v"', bend right=72]\arrow[d,"\wr","\Delta"']&\MTDer{(B)}{G/Q}\arrow[d,"\wr"',"\Delta"]\arrow[ddddd,"v", bend left=72]\\
		\Hotb(\MTDer{(B)}{P/Q_w}_{w=0})\arrow[d,"\wr","\Hyp"']&\Hotb(\MTDer{(B)}{G/Q}_{w=0})\arrow[d,"\wr"',"\Hyp"]\\
		\Hotb(\CoInvi{P}{Q_w}\Smodules^{\Z,ev})\arrow[r,"\sindn{w}{Q}"]\arrow[d,"v"']&\Hotb(\CoInvi{G}{Q}\Smodules^{\Z,ev})\arrow[d,"v"]\\
		\Hotb(\CoInv{\levi}{w\cdot\lambda}\Smodules)\arrow[r,"\sindn{w}{\lambda}"]&\Hotb(\CoInv{\g}{\lambda}\Smodules)\\
		\Hotb(\Proje\cato{\levi}{w\cdot\lambda})\arrow[u,"\wr"',"\V{\levi}{w\cdot\lambda}"]&\Hotb(\Proje \cato{\g}{\lambda})\arrow[u,"\wr","\V{\g}{\lambda}"']\\
		\Derb(\cato{\levi}{w\cdot\lambda})\arrow[r,"(\ind)^{\oplus n}"]\arrow[u,"\wr"']&\Derb(\cato{\g}{\lambda})\arrow[u,"\wr"]
	\end{tikzcd}
	\end{center}
\end{theorem}
\begin{proof}
	The upper and lower rectangles are Theorem \ref{thm:gindandsindsingular} and \ref{thm:indandsindsingular}. By definition $v \sindn{w}{Q}=\sindn{w}{\lambda}v$ and the statement follows.
\end{proof}
Unfortunately, we will not prove that the corresponding
diagram with just  $\ind$ and $\gind{w}$ and without the direct sum commutes.
But let us sketch a possible approach. In the proof of Theorem \ref{thm:paraindgradable} we will show
how a Krull-Remak-Schmidt argument allows to get rid of the direct sum for
the restrictions of the (geometric) parabolic to the heart of a $t$-structure on
the categories and show that the following diagram of functors commutes (up
to natural isomorphism):
\begin{center}
\begin{tikzcd}
\MTDer{(B)}{P/Q_w}^\heartsuit\arrow[r, "\gind{w}"]\arrow[d, "v"]&\MTDer{(B)}{G/Q}^\heartsuit\arrow[d, "v"]\\
\cato{\levi}{w\cdot\lambda}\arrow[r, "\ind"]&\cato{\g}{\lambda}
\end{tikzcd}
\end{center}
Now it would suffice to show that the following diagram of functors commutes
\begin{center}
\begin{tikzcd}
\Derb(\MTDer{(B)}{P/Q_w}^\heartsuit)\arrow[r, "\gind{w}"]\arrow[d, "real","\wr"']&\Derb(\MTDer{(B)}{G/Q}^\heartsuit)\arrow[d, "real","\wr"']\\
\MTDer{(B)}{P/Q_w}\arrow[r, "\gind{w}"]&\MTDer{(B)}{G/Q}
\end{tikzcd}
\end{center}
where the upper horizontal arrow is given by pointwise application of the
($t$-exact) functor $\gind{w}$. This is true by using for example \cite[Lemma A.7.1]{beilinson1987derived} or \cite[Theorem 1.3.3.2]{lur}. Both results require the existence of a lift of
$\gind{w}$ to some upgraded category of motives; an f-category in the former
and a stable $\infty$-category in the latter. Since $\gind{w} = \hop_{w,*}\pr{w}^!$
 is defined using the six functors, this lift exists. Introducing the necessary notation would go beyond the scope of this article. We hence omit a proof.

For regular weights everything just works fine.
\begin{corollary}
	Let $\lambda\in\cartan^*$ be a \emph{regular} integral dominant weight and $w\in \Weyl$ be a shortest coset representative in $\Wi\backslash\Weyl$ then the following diagram commutes up to natural isomorphism.
	\begin{center}
		\begin{tikzcd}[column sep= 1.5cm]
			\MTDer{(B)}{P/B}\arrow[r,"\gind{w}"]\arrow[d,"v"']&\MTDer{(B)}{G/B}\arrow[d,"v"]\\
			\Derb(\cato{\levi}{w\cdot\lambda})\arrow[r,"\ind"]&\Derb(\cato{\g}{\lambda})
		\end{tikzcd}
	\end{center}
\end{corollary}
\subsection{Graded Parabolic Induction}
The main goal of this section is to use our results to show that parabolic induction for integral blocks of category $\mathcal O$ is \emph{gradable} (see \cite[Definition 3.3]{Str}), which means that there is a functor $\grind$ making the following diagram commute up to natural isomorphism
\begin{center}
	\begin{tikzcd}
		\catoz{\levi}{w\cdot\lambda}\arrow[r,"\grind"]\arrow[d,"v"] &\catoz{\g}{\lambda}\arrow[d,"v"]\\
		\cato{\levi}{w\cdot\lambda}\arrow[r,"\ind"] &\cato{\g}{\lambda}
	\end{tikzcd}
\end{center}
and fulfilling $\grind\langle n\rangle=\langle n\rangle \grind$, where $\mathcal O^\Z$ denotes the \emph{graded category $\mathcal{O}$} as defined in \cite{BGS} and $\langle - \rangle$ denotes the shift of grading. 

Graded category $\mathcal{O}$ (for a fixed block) is constructed by establishing a grading on the ring 
	$$A=\End{\mathcal{O}}(P)$$
where $P$ denotes a (minimal) projective generator of the given block and then defining graded category $\mathcal{O}$ as the category of finitely generated graded modules over $A$.
\begin{center}
	\begin{tikzcd}[column sep= 2 cm]
		\mathcal O^\Z_\lambda\arrow[r,equal,"def"]\arrow[d,"v"]&\operatorname{mod}^\Z\operatorname{-}A \arrow[d,"v"]\\
		\mathcal O_\lambda \arrow[r,"\Hom{\mathcal O}{P}{-}","\sim"'] &\operatorname{mod-}A
	\end{tikzcd}
\end{center}
This grading on $A$ is established by realizing it as $\operatorname{Ext}$-ring of a certain complex of sheaves on the Langlands dual flag variety.

Let us explain what this concretely means in our setting. See \cite[Section 11]{SoeWe} for a reference. Fix a dominant integral weight $\lambda$ and the corresponding standard parabolic $Q$. Denote by 
\begin{align*}
	P &=\bigoplus_{x\in \Weyl/\W_{\g,\lambda}} P_\g(x\cdot\lambda)\in \cato{\g}{\lambda}\text{ and}\\
	L &=\bigoplus_{x\in \W/\W_Q} \mathcal{IC}_x\in \operatorname{MTDer}_{(B)}(G/Q)_{w=0}
\end{align*}
the sum of the indecomposable projectives and the sum of simple weight zero perverse stratified  mixed Tate motives $\mathcal{IC}_x$ supported on $\overline{BxQ/Q}$, respectively.
Denote by
\begin{align*}
A &=\End{\g}(P)\text{ and}\\
A' &=\bigoplus_{i\in\Z}\Hom{\Motives{G/Q}}{L}{L(i)[2i]}.
\end{align*}
By showing $\mathbb{V}(P_\g(x\cdot\lambda))\cong \Hyp(\mathcal{IC}_x)$ as $\CoInv{G}{Q}=\CoInv{\g}{\lambda}$-modules and using Soergel's Erweiterungssatz and Struktursatz, one sees that in fact $A\cong A'$. This puts a grading on $A$. To be compatible with \cite{BGS}, we redefine this grading to be even, i.e. we want the shift $(i)[2i]$ to correspond to $\langle 2i\rangle$, or in other words $A'_{2i}=\Hom{\Motives{G/Q}}{L}{L(i)[2i]}$ and $A'_{2i+1}=0$. Then the graded category $\mathcal O$ is defined by 
$$\catoz{\g}{\lambda}\defi\operatorname{mod^\Z-}A'$$
Denote by $B\cong B'$ the algebras analogously defined for $P/Q_w$ and $\cato{\levi}{w\cdot\lambda}$.

The tilting equivalence from Theorem \ref{thm:tiltingformotives} can also be stated as an equivalence
$$\Delta: \MTDer{(B)}{G/Q}\stackrel{\sim}{\rightarrow}\Derb(\operatorname{mod^{\Z,ev}-}A')$$
as discussed in Remark \ref{rem:tiltingformotives}.
This equivalence equips $\MTDer{(B)}{G/Q}$ with a $t$-structure, which is the Koszul dual of the perverse $t$-structure, see \cite[Section 1.4]{SoeWe}. We denote its heart by $\MTDer{(B)}{G/Q}^\heartsuit$.
Hence mixed stratified Tate motives on the flag variety provide a geometric realization of the evenly graded category $\mathcal O$
\begin{align*}
\MTDer{(B)}{G/Q}^\heartsuit &\cong\operatorname{mod^{\Z,ev}-}A'\cong\catozev{\g}{\lambda}\\
\MTDer{(B)}{P/Q_w}^\heartsuit&\cong\operatorname{mod^{\Z,ev}-}B'\cong\catozev{\levi}{w\cdot\lambda}
\end{align*}
and we can use our geometric construction to show that parabolic induction is gradable.
\begin{theorem}\label{thm:paraindgradable}
	Let $\lambda$ be a dominant integral weight. Then there is a functor $\grind$ compatible with the shift of grading $\langle n\rangle$, making the following diagram commute
	\begin{center}
		\begin{tikzcd}
		\catoz{\levi}{w\cdot\lambda}\arrow[r,"\grind"]\arrow[d,"v"] &\catoz{\g}{\lambda}\arrow[d,"v"]\\
		\cato{\levi}{w\cdot\lambda}\arrow[r,"\ind"] &\cato{\g}{\lambda}.
	\end{tikzcd}
	\end{center}
\end{theorem}
\begin{proof}
In the notation of Theorem \ref{thm:main} consider the following diagram.
	\begin{center}
		\begin{tikzcd}[column sep=0.3cm]
		\MTDer{(B)}{P/Q_w}\arrow[rrr,"\overline{\operatorname{GInd}}_w"]\arrow[rd,"\Delta","\sim"'] \arrow[ddd,"v"]&[-10pt]&&[-10pt]\MTDer{(B)}{G/Q}\arrow[ld,"\Delta"',,"\sim"]\arrow[ddd,"v"] \\
		&\Derb(\operatorname{mod^{\Z,ev}-}B')\arrow[d,"v"]\arrow[r,"F"]&\Derb(\operatorname{mod^{\Z,ev}-}A')\arrow[d,"v"]&\\
		&\Derb(\operatorname{mod-}B)\arrow[ld,"\sim"]\arrow[r,"G"]&\Derb(\operatorname{mod-}A)\arrow[rd,"\sim"']&\\
		\Derb(\cato{\levi}{w\cdot\lambda})\arrow[rrr,"(\ind)^{\oplus n}"] &&&\Derb(\cato{\g}{\lambda}).
	\end{tikzcd}
	\end{center}
	The functors $v$ on the very left and right are defined as in Theorem \ref{thm:main} and the functors $v$ in the middle are forgetting the grading and using the isomorphisms $A\cong A'$ and $B\cong B'$.
	In fact, the trapezia on the left and the right commute (up to natural isomorphism). See Remark \ref{rem:expandeddiag} for an expanded version. Both $F$ and $G$ denote the functors induced by the equivalences. Then by definition 
	\begin{align*}
		F&=\overline F\langle i_1\rangle\oplus\dots\oplus \overline F\langle i_n\rangle\\
		G&=\overline G^{\oplus n}.
	\end{align*}
	split into direct summands. $G$ is clearly $t$-exact since $\ind$ is, hence we are precisely in the setting of Proposition \ref{prop:gradedfunctors}, which gives us a natural equivalence of functors 
	$v\overline F_0\cong \overline G_0v$ and thereby a commutative diagram
	\begin{center}
		\begin{tikzcd}[column sep=1cm]
			\MTDer{(B)}{P/Q_w}^\heartsuit\arrow[rd,"\Delta","\sim"']\arrow[rrr,"\gind{w}"] \arrow[ddd,"v"]&[-75pt]&&[-75pt]\MTDer{(B)}{G/Q}^\heartsuit\arrow[ddd,"v"]\arrow[ld,"\Delta"',,"\sim"] \\
			&\catozev{\levi}{w\cdot\lambda}=\operatorname{mod^{\Z,ev}-}B'\arrow[d,"v"]\arrow[r,"\overline F_0"]&\operatorname{mod^{\Z,ev}-}A'=\catozev{\g}{\lambda}\arrow[d,"v"]&\\
			&\operatorname{mod-}B\arrow[ld,"\sim"]\arrow[r,"\overline G_0"]&\operatorname{mod-}A\arrow[rd,"\sim"']&\\
			\cato{\levi}{w\cdot\lambda}\arrow[rrr,"\ind"] &&&\cato{\g}{\lambda}.
		\end{tikzcd}
	\end{center}
	Clearly $\overline F_0$ commutes with the shift of grading, since it is induced by $\gind{w}=\gindv{w}$, which commutes with $(i)[2i]$. So $\overline F_0$ is a grading of parabolic induction for the evenly graded category $\mathcal O$. This can be easily extended to the whole graded category $\mathcal O$ since $\mathcal O^\Z=\mathcal O^{\Z,ev}\oplus \mathcal O^{\Z,ev}\langle1\rangle$.
\end{proof}
\begin{remark}\label{rem:expandeddiag}
	The following diagram commutes (up to natural isomorphism)
	\begin{center}
		\begin{tikzcd}
			\Hotb(\MTDer{(B)}{G/Q}_{w=0})\arrow[d,equal]& \MTDer{(B)}{G/Q}\arrow[d,"\wr"',"\Delta"]\arrow[l,"\sim","\Delta"']\\
			\Hotb(\langle \mathcal{IC}_x \,|\,x\in \W \rangle^{\Motives{G/Q}}_{\oplus,(1)[2]})\arrow[r,"\sim"]\arrow[d,"\Hyp"',"\wr"]& \Derb(\operatorname{mod^{\Z}-}A')\arrow[d,"\wr"']\\
			\Hotb(\langle \Hyp(\mathcal{IC}_x) \,|\,x\in \W \rangle^{\CoInv{G}{Q}\modules^\Z}_{\oplus,\langle 2\rangle})\arrow[r,"\sim"]\arrow[d,"v"']& \Derb(\operatorname{mod^{\Z}-}\Hyp(A'))\arrow[d,"v"]\\
			\Hotb(\langle v\Hyp(\mathcal{IC}_x) \,|\,x\in \W \rangle^{\CoInv{G}{Q}\modules}_{\oplus})\arrow[r,"\sim"]\arrow[d,"\wr"] & \Derb(\operatorname{mod-}v\Hyp(A'))\arrow[d,"\wr"']\\
			\Hotb(\langle \mathbb V P_\g(x\cdot\lambda) \,|\,x\in \W \rangle^{\CoInv{\g}{\lambda}\modules}_{\oplus})\arrow[r,"\sim"] & \Derb(\operatorname{mod-}\mathbb V(A))\\
			\Hotb(\langle  P_\g(x\cdot\lambda) \,|\,x\in \W \rangle^{\cato{\g}{\lambda}}_{\oplus})\arrow[r,"\sim"]\arrow[u,"\mathbb V"',"\wr"] & \Derb(\operatorname{mod-}A)\arrow[u,"\wr"']\\
			\Hotb(\Proje{\cato{\g}{\lambda}})\arrow[r,"\sim"]\arrow[u,equal] & \Derb(\cato{\g}{\lambda})\arrow[u,"\wr"']
		\end{tikzcd}
	\end{center}
	where the horizontal arrows are the obvious equivalences and we denote
	\begin{align*}
		\Hyp(A')&\defi\bigoplus_{i\in\Z}\Hom{\CoInv{G}{Q}\modules^\Z}{\Hyp L}{\Hyp L\langle 2i\rangle},\\
		v\Hyp(A')&\defi\bigoplus_{i\in\Z}\Hom{\CoInv{G}{Q}\modules}{v\Hyp L}{v\Hyp L\langle 2i\rangle}\\&=\End{\CoInv{G}{Q}\modules}(v\mathbb H L)\text{ and}\\
		\mathbb V (A)&\defi\End{\CoInv{\g}{\lambda}\modules}(\mathbb V P).
	\end{align*}
\end{remark}

\appendix
\section{Some Category Theory}
In this appendix we recall some notions from category theory which may not be so
well known for the convenience of the reader.
\subsection{Graded Eilenberg-Watts Theorem}
\begin{proposition}\label{prop:gradedfunctors}
	Let $A$ and $B$ be finite dimensional graded $\C$-algebras and 
	$$v:A\modules^\Z\rightarrow A\modules\text{ and }B\modules^\Z\rightarrow B\modules$$
	be the functors forgetting the grading. Assume that there is a diagram 
	\begin{center}
		\begin{tikzcd}
			\Derb(B\modules^\Z)\arrow[r,"F"]\arrow[d,"v"]&\Derb(A\modules^\Z)\arrow[d,"v"]\\
			\Derb(B\modules)\arrow[r,"G"]&\Derb(A\modules)
		\end{tikzcd}
	\end{center}
	commuting up to natural isomorphism, and that $F$ commutes with the shift of grading. Then the following statements hold.
	\begin{enumerate}
		\item $F$ is exact with respect to the standard $t$-structure if and only if $G$ is.
		\item Assume that $F$ splits into a direct sum
		$$F=\overline F\langle i_1\rangle\oplus\dots\oplus \overline F\langle i_n\rangle.$$
		of shifted versions of a functor $\overline F$. Then $F$ is exact if and only if $\overline F$ is exact.
		\item Assume that $F$ and $G$ are exact and that there are functors $\overline F$, $\overline G$ such that
		\begin{align*}
		F&=\overline F\langle i_1\rangle\oplus\dots\oplus \overline F\langle i_n\rangle\\
		G&=\overline G^{\oplus n}.
		\end{align*}
		Denote the induced functors on the heart of the $t$-structure by $F_0$, $G_0$, $\overline F_0$, $\overline G_0$. Then the following diagram commutes up to natural isomorphism
		\begin{center}
			\begin{tikzcd}
				B\modules^\Z\arrow[r,"\overline F_0"]\arrow[d,"v"]&A\modules^\Z\arrow[d,"v"]\\
				B\modules\arrow[r,"\overline G_0"]&A\modules
			\end{tikzcd}.
		\end{center}
	\end{enumerate}
\end{proposition}
\begin{proof}
	(1) Denote by $\mathcal{H}^i(C)$ the $i$-th cohomology of a complex $C$. Then clearly $\mathcal{H}^i(vC)=v\mathcal{H}^i(C)$ for all $C\in\Derb(A\modules^\Z), \Derb(B\modules^\Z)$. Since the standard $t$-structure is defined by vanishing conditions on cohomology, the statement follows.
	
	(2) As in (1).
	
	(3) Since by assumption and points (1) and (2) $\overline F$ and $\overline G$ are exact, they restrict to the hearts of the $t$-structure, which are naturally isomorphic to $A\modules^\Z$, $B\modules^\Z$, $A\modules$ and $B\modules$, respectively. So the diagram makes sense.
	Now clearly $\overline G_0$, $G_0$, $\overline F_0$ and $ F_0$ are exact functors. By a graded version of the Eilenberg-Watts Theorem \cite{watts1960intrinsic}, this implies that there are natural isomorphisms
	\begin{align*}
	\overline F_0&\cong M\otimes_{B}- \text{ and}\\
	\overline G_0&\cong N\otimes_{B}-
	\intertext{for the (graded) $B$-$A$-bimodules $M=F_0(B),N=G_0(B)$ and hence}
	F_0&\cong \bigoplus_{j}M\langle i_j\rangle \otimes_B- \text{ and}\\
	G_0&\cong N^{\oplus n} \otimes_B-.
	\intertext{By assumption, there is a natural equivalence $G_0v\cong vF_0$ and hence}
	v(M)^{\oplus n}&\cong v(\bigoplus_{j}M\langle i_j\rangle)\cong  N^{\oplus n}.
	\end{align*}
	Now decomposing both $v(M)$ and $N$ into a finite direct sum of indecomposables and applying the Krull--Remak--Schmidt theorem implies that there is an isomorphism 
	$$v(M)\cong N$$
	and hence a natural isomorphism \begin{align*}
	\overline{G_0}v\cong v\overline{F_0}.
	\end{align*}
	The statement follows.
\end{proof}
\subsection{Functoriality Of Cone}
\begin{lemma}\label{lem:triangulatedfunctor}
	Let $F,G:\mathcal{T}\rightarrow\mathcal{T}'$ be triangulated functors between triangulated categories $\mathcal{T}$ and $\mathcal{T}'$, and let $\phi:F\Rightarrow G$ be a morphism of functors. For $X\in \mathcal{T}$ abbreviate $C(X)\defi\operatorname{Cone}(\phi_X : F(X)\rightarrow G(X))$. 
	Assume that for $X,Y\in \mathcal{T}$ we have 
	$$\Hom{\mathcal{T}'}{F(X)[1]}{C(Y)}=0.$$ 
	Then there exists (up to natural isomorphism) a unique functor $H: \mathcal{T}\rightarrow\mathcal{T}'$ and morphisms $G\rightarrow H \rightarrow F[1]$ which induce distinguished triangles
	\begin{center}
		\disttrianglewithmaps{F(X)}{\phi_X}{G(X)}{}{H(X)}
	\end{center}
	for all $X\in \mathcal{T}$. In particular $H(X)\cong C(X).$	
\end{lemma}
\begin{proof}
	Let $f:X\rightarrow Y$ be a morphism in $\mathcal{T}$. We claim that there is a \emph{unique} morphism $C(f)$ making the following diagram commute:
	\begin{center}
		\begin{tikzcd}[column sep= 0.6cm]
			F(X)\arrow[r,"\phi_X"]\arrow[d,"F(f)"] & G(X) \arrow[r,"\psi_X"]\arrow[d,"G(f)"] & C(X)\arrow[r,"+1"]\arrow[d,"\exists!\,C(f)",dashed]& ~ \\
			F(Y)\arrow[r,"\phi_Y"] & G(Y) \arrow[r,,"\psi_Y"] & C(Y)\arrow[r,"+1"]& ~
		\end{tikzcd}
	\end{center}
	To see this, consider the long exact sequence:
	 \begin{center}
	 	\begin{tikzcd} [column sep= 0.4cm]
 		\dots\arrow[r]
 		&
 		0=\Hom{\mathcal{T}'}{F(X)[1]}{C(Y)}
 		\arrow[r]
 		&
		\Hom{\mathcal{T}'}{C(X)}{C(Y)}
		\arrow[r]
		&
		\phantom{\dots}
		\\
		\phantom{\dots}
		\arrow[r,"\psi_X^*"]
		&
		\Hom{\mathcal{T}'}{G(X)}{C(Y)}
		\arrow[r,"\phi_X^*"]
		&
		\Hom{\mathcal{T}'}{F(X)}{C(Y)}
		\arrow[r]
		&
		\dots
	 \end{tikzcd}
	 \end{center}
	 Then $\phi_X^*(\psi_Y G(f))=\psi_Y G(f)\phi_X=\psi_Y\phi_YF(f)=0$. Hence $C(f)$ exists and is uniquely determined by the equation $C(f)\psi_X=\psi_X^*(C(f))=\psi_YG(f)$.
	 
	 Moreover if $g:Y\rightarrow Z$ is a second morphism in $\mathcal{T}$, the uniqueness immediately implies $C(g)C(f)=C(gf)$. Hence $C$ defines a functor with the required properties. The same uniqueness arguments show that $C$ is uniquely determined with these properties (up to natural isomorphism).  
	 \end{proof}
\subsection{Tilting}\label{subsec:tilting}
We recall the formalism of tilting for derived (dg)-categories, as introduced in \cite{rickard}, \cite{Keller1993} and \cite{keller1994deriving} and prove a compatibilty of tilting with other functors.
\begin{definition}
	Let $\mathcal{A}$ be an abelian category. A complex $I\in \Hot(\mathcal{A})$ is called \textbf{homotopy-injective} if the natural map $$\Hom{\Hot(\mathcal{A})}{A}{I}\stackrel{\sim}{\rightarrow} \Hom{\Der(\mathcal{A})}{A}{I}$$ is an isomorphism for all $A\in\Hot(\mathcal{A})$. A complex $P\in \Hot(\mathcal{A})$ is called \textbf{homotopy-projective} if the natural map $$\Hom{\Hot(\mathcal{A})}{P}{A}\stackrel{\sim}{\rightarrow} \Hom{\Der(\mathcal{A})}{P}{A}$$ is an isomorphism for all $A\in\Hot(\mathcal{A})$.
\end{definition}
\begin{definition}
	Let $\mathcal{A}$ be an abelian category. A collection $\{T_i\}$ of complexes $T_i\in \Hot(\mathcal{A})$ is called \textbf{tilting} if for all $i,j$ and $n\in\Z$ the natural map
	$$\Hom{\Hot(\mathcal{A})}{T_i}{T_j[n]}\stackrel{\sim}{\rightarrow}\Hom{\Der(\mathcal{A})}{T_i}{T_j[n]}$$
	is an isomorphism and
	 $$\Hom{\Der(\mathcal{A})}{T_i}{T_j[n]}=0$$ for all $n\neq 0$.
\end{definition}
For complexes $M,N\in\Hot(\mathcal{P})$ in some additive category $\mathcal P$, we  denote by $\HomC{\mathcal{P}}{M}{N}\in\Hot(\mathcal{P})$ their Hom-complex.
\begin{theorem}[Tilting \cite{Keller1993}]\label{thm:tilting}
	Let $\mathcal{A}$ be an abelian category and  $\{T_i\}$ a tilting collection.
	Then there is an equivalence of triangulated categories
	$$\Delta: \Hotb(\langle \{T_i\} \rangle^{\Der(\mathcal A)}_\oplus)\stackrel{\sim}{\rightarrow}\langle \{T_i\} \rangle^{\Der(\mathcal A)}_\Delta\subset\Der(\mathcal A)$$
	called \textbf{tilting}. Here by $\langle-\rangle^{\mathcal{B}}_\oplus$ we denote closure under finite direct sums in an additive category $\mathcal{B}$ and by $\langle-\rangle^{\mathcal{B}}_\Delta$ closure under distinguished triangles in a triangulated category $\mathcal{B}$.
\end{theorem}
\begin{proof}This is copied almost word by word from \cite[Appendix B]{SoeWe}.
	We just sketch a proof for $\{T_i\}=\{T\}$. Since by assumption $$\Hom{\Hot(\mathcal{A})}{T}{T[n]}\stackrel{\sim}{\rightarrow}\Hom{\Der(\mathcal{A})}{T}{T[n]}$$   we have by d\'evissage
	$$\langle T \rangle_\oplus^{\Der}\cong \langle T \rangle_\oplus^{\Hot}\text{ and }\langle T \rangle_\Delta^{\Der}\cong \langle T \rangle_\Delta^{\Hot}.$$
	So it suffices to proof that there is an equivalence 
	$$\Delta: \Hotb(\langle T \rangle^{\Hot(\mathcal A)}_\oplus)\stackrel{\sim}{\rightarrow}\langle T \rangle^{\Hot(\mathcal A)}_\Delta.$$
	 Let $E\defi\Hom{\mathcal{A}}{T}{T}$ be the endomorphism complex of $T$. This is a differential graded algebra (dg-algebra). Let $Z\defi\mathcal Z^0(E)\oplus E^{\leq 0}$ be the truncation of $E$ and $H\defi\mathcal H^0(E)$ be the $0$-th cohomology of $E$. By the tilting property the cohomology of $E$ is concentrated in degree zero and hence the natural morphisms
	$$E\hookleftarrow Z\twoheadrightarrow H$$
	are quasi-isomorphisms of dg-algebras and furthermore $$H=\Hom{\Hot(\mathcal A)}{T}{T}\cong\Hom{\Der(\mathcal A)}{T}{T}.$$ For a dg-algebra $R$ we denote by $$\operatorname{dgHot-}R\supset\operatorname{dgFree-}R$$
	the homotopy category of right $R$-dg-modules and the triangulated subcategory generated by the free module $R$. Then there is the following chain of equivalences of triangulated categories
	\begin{center}
		\begin{tikzcd}[column sep=.5cm]
			\Hotb(\langle T \rangle^{\Hot(\mathcal A)}_\oplus) \arrow[r,"\sim","(1)"']&
			\operatorname{dgFree-}H&
			\operatorname{dgFree-}Z\arrow[l,"\sim"',"(2)"]\arrow[r,"\sim","(3)"'] &
			\operatorname{dgFree-}E & 
			\langle T\rangle^{\Hot(\mathcal A)}_\Delta\arrow[l,"\sim"',"(4)"]
		\end{tikzcd}
	\end{center}
	\begin{enumerate}
		\item The following functor induces an equivalence of categories
		$$\HomC{\Hot(\mathcal{A})}{T}{-}: \langle T \rangle^{\Hot(\mathcal A)}_\oplus\stackrel{\sim}{\rightarrow}\langle H \rangle^{\operatorname{mod-}H}_\oplus.$$
		Since $H$ is a dg-algebra concentrated in degree $0$, dg-modules over $H$ are just complexes of $H$-modules and we have an equivalence
		$$\Hotb(\langle T \rangle^{\Hot(\mathcal A)}_\oplus)\stackrel{\sim}{\rightarrow}\Hotb(\langle H \rangle^{\operatorname{mod-}H}_\oplus)=\operatorname{dgFree-}H.$$
		\item This is $-\otimes_Z H$, which is an equivalence since $Z\twoheadrightarrow H$ is a quasi-isomorphism.
		\item This is $-\otimes_Z E$, which is an equivalence since $Z\hookrightarrow E$ is a quasi-isomorphism.
		\item This is given by the functor
		$$\HomC{\mathcal{A}}{T}{-}: \langle T \rangle^{\Hot(\mathcal A)}_\Delta\stackrel{\sim}{\rightarrow}\operatorname{dgFree-}E.$$
	\end{enumerate}
	Our tilting functor $\Delta$ is defined as the composition of those equivalences (or their inverse functors).
\end{proof}
\begin{remark}\label{rem:dualtilting}
	We could have also used the functors $\Hom{\Hot \mathcal A}{-}{T}$ and $\Hom{\mathcal A}{-}{T}$. Then our tilting equivalence would be of the form 
	\begin{center}
		\begin{tikzcd}[column sep=1cm]
		\Hotb(\langle T \rangle^{\Hot(\mathcal A)}_\oplus) \arrow[r,"\sim"]&
		(H\operatorname{-dgFree})^{\operatorname{op}}& \\
		(Z\operatorname{-dgFree})^{\operatorname{op}} \arrow[r,"\sim"]\arrow[ur,"\sim"'] &
		(E\operatorname{-dgFree})^{\operatorname{op}} & 
		\langle T\rangle^{\Hot(\mathcal A)}_\Delta\arrow[l,"\sim"']
	\end{tikzcd}
	\end{center}
	But we rather avoid opposite categories.
\end{remark}
\begin{proposition}[Tilting and functors]\label{prop:tiltingandfunctors} Let $\mathcal{A},\mathcal{B}$ be abelian categories. Let $\{T_i\}\subset \Hot(\mathcal{A})$ and $\{U_i\}\subset \Hot(\mathcal{B})$ be tilting collections. Assume furthermore that all $U_i$ are homotopy-projective. 
	
Let $F:\mathcal{A}\rightarrow \mathcal{B}$ be an exact functor. Assume that for all $T_i$ there exists a $U_j$ and a quasi-isomorphism $c_i:U_j\rightarrow F(T_i)$. Then the following diagram commutes up to natural isomorphism
	\begin{center}
		\begin{tikzcd}
			\Hotb(\langle \{T_i\} \rangle^{\Der(\mathcal A)}_\oplus) \arrow[r,"\Delta","\sim"']\arrow[d,"F"]& 
			\langle \{T_i\} \rangle_\Delta^{\Der(\mathcal A)}\arrow[r,hook]\arrow[d,"F"]&
			\Der(\mathcal A)\arrow[d,"F"]
			\\
			\Hotb(\langle \{U_i\} \rangle^{\Der(\mathcal B)}_\oplus) \arrow[r,"\Delta","\sim"']& 
			\langle \{U_i\} \rangle_\Delta^{\Der(\mathcal B)}\arrow[r,hook]&
			\Der(\mathcal B)
		\end{tikzcd}
	\end{center}
	where on the left side $F$ acts by pointwise application or in other words, it is induced by an sequence of functors
	$$\langle \{T_i\}\rangle_\oplus^{\Der(\mathcal A)}\stackrel{\sim}{\leftarrow}\langle \{T_i\}\rangle_\oplus^{\Hot(\mathcal A)}\stackrel{F}{\rightarrow}\langle \{F(T_i)\}\rangle_\oplus^{\Hot(\mathcal B)}\rightarrow\langle \{F(T_i)\}\rangle_\oplus^{\Der(\mathcal B)}\hookrightarrow\langle \{U_i\}\rangle_\oplus^{\Der(\mathcal B)}$$
\end{proposition}
\begin{proof}
	This is copied almost word by word from \cite[Appendix]{SoeWeVir}.
	Again, we restrict ourselves to the case $\{T_i\}=\{T\}$ and $\{U_i\}=\{U\}$ and a quasi-isomorphism  $c:U\rightarrow F(T)$. We abbreviate $S=F(T)$ and $E_T=\HomC{\mathcal A}{T}{T}$, $E_S=\HomC{\mathcal B}{S}{S}$, $E_U=\HomC{\mathcal B}{U}{U}$. Let $Z_T,Z_S,Z_U$ and $H_T,H_S,H_U$ their degree 0 truncation and 0-th cohomology, respectively.
We consider a diagram where the top and bottom row are defining the tilting equivalences $\Delta$ as in Theorem \ref{thm:tilting} and fill it up with natural isomorphisms.
\begin{center}
	\adjustbox{max width=1.05\textwidth}{\hspace{-20pt}
		\begin{tikzcd}[column sep=1.35cm,row sep=2cm]
	\Hotb(\langle T \rangle^{\Hot(\mathcal A)}_\oplus) \arrow[r,"\sim"',"\Hom{\Hot(\mathcal A)}{T}{-}"]\arrow[d,"F"']\arrow[rd,phantom,"(1)"]&
	\operatorname{dgFree-}H_T \arrow[d,"-\otimes_{H_T}H_S"']\arrow[rd,phantom,"(2)"] &[-18pt]
	\operatorname{dgFree-}Z_T\arrow[l,"\sim","-\otimes_{Z_T}H_T"']\arrow[r,"-\otimes_{Z_T}E_T","\sim"'] \arrow[rd,phantom,"(3)"]\arrow[d,"-\otimes_{Z_T}Z_S"']&[-18pt]
	\operatorname{dgFree-}E_T \arrow[rd,phantom,"(4)"]\arrow[d,"-\otimes_{E_T}E_S"']& 
	\langle T\rangle^{\Hot(\mathcal A)}_\Delta\arrow[l,"\sim","\Hom{\mathcal A}{T}{-}"']\arrow[d,"F"']
	\\
	\Hotb(\langle S \rangle^{\Hot(\mathcal A)}_\oplus) \arrow[r,"\sim"',"\Hom{\Hot(\mathcal B)}{S}{-}"]\arrow[d]\arrow[rd,phantom,"(5)"]&
	\operatorname{dgFree-}H_S \arrow[d,"-\otimes_{H_S}H_X"']\arrow[rd,phantom,"(6)"] &
	\operatorname{dgFree-}Z_S\arrow[l,"-\otimes_{Z_S}H_S"']\arrow[r,"-\otimes_{Z_S}E_S"] \arrow[rd,phantom,"(7)"]\arrow[d,"-\otimes_{Z_S}Z_X"']&
	\operatorname{dgFree-}E_S \arrow[rd,phantom,"(8)"]\arrow[d,"-\otimes_{E_S}X"']& 
	\langle S\rangle^{\Hot(\mathcal B)}_\Delta\arrow[l,"\sim","\Hom{\mathcal B}{S}{-}"']\arrow[d]
	\\
	\Hotb(\langle U \rangle^{\Der(\mathcal B)}_\oplus) \arrow[r,"\sim"',"\Hom{\Der(\mathcal B)}{U}{-}"]&
	\operatorname{dgFree^{Der}-}H_U\arrow[r,"\sim"',"\res_{H_U}^{Z_U}"]&
	\operatorname{dgFree^{Der}-}Z_U &
	\operatorname{dgFree^{Der}-}E_U\arrow[l,"\sim","\res_{E_U}^{Z_U}"'] & 
	\langle U\rangle^{\Der(\mathcal B)}_\Delta\arrow[l,"\sim","\Hom{\mathcal B}{U}{-}"']
	\\
	\Hotb(\langle U \rangle^{\Hot(\mathcal B)}_\oplus) \arrow[r,"\sim"',"\Hom{\Hot(\mathcal B)}{U}{-}"]\arrow[u,"\wr"']\arrow[ru,phantom,""]&
	\operatorname{dgFree-}H_U \arrow[u,"\wr"']\arrow[ru,phantom,""] &
	\operatorname{dgFree-}Z_U\arrow[l,"\sim","-\otimes_{Z_U}H_U"']\arrow[r,"-\otimes_{Z_U}E_U","\sim"'] \arrow[ru,phantom,""]\arrow[u,"\wr"']&
	\operatorname{dgFree-}E_U \arrow[ru,phantom,""]\arrow[u,"\wr"']& 
	\langle U\rangle^{\Hot(\mathcal B)}_\Delta\arrow[l,"\sim","\Hom{\mathcal B}{U}{-}"']\arrow[u,"\wr"']
	\end{tikzcd}}
\end{center}
	\begin{enumerate}
		\item Here we use the map $$F:H_T=\Hom{Hot(\mathcal{A})}{T}{T}\rightarrow\Hom{Hot(\mathcal{B})}{F(T)}{F(T)}=H_S.$$
		The natural transformation is given by $$\Hom{\Hot(\mathcal A)}{T}{-}\otimes_{H_T}H_S\stackrel{\operatorname{comp}(F\otimes\id)}{\rightarrow} \Hom{\Hot(\mathcal B)}{F(T)}{F(-)}$$
		where by $\operatorname{comp}$ be denote composition. This is clearly an isomorphism when applied to $T$ and hence restricts to a natural isomorphism by devissage.
		\item The morphism $Z_T\rightarrow Z_S$ is induced by  $$F:E_T=\HomC{\mathcal{A}}{T}{T}\rightarrow\HomC{\mathcal{B}}{F(T)}{F(T)}= E_S.$$
		The following diagram commutes.
		\begin{center}
			\begin{tikzcd}
				Z_T\arrow[r,two heads]\arrow[d,"F"]& H_T\arrow[d,"F"]\\
				Z_S\arrow[r,two heads]& H_S
			\end{tikzcd}
		\end{center} 
		We hence get a natural isomorphism
		$$-\otimes_{Z_T}H_T\otimes_{H_T}H_S\cong -\otimes_{Z_T}Z_S\otimes_{Z_S}H_S.$$
		\item As in the last point.
		\item The natural transformation is given by
		$$\Hom{\mathcal A}{T}{-}\otimes_{E_T}E_S\stackrel{\operatorname{comp}(F\otimes\id)}{\rightarrow} \Hom{\mathcal{B}}{F(T)}{F(-)}.$$
		Again this restricts to a natural isomorphism by devissage.
	\end{enumerate}
	For a dg-algebra $R$ we denote by $\operatorname{dgFree^{Der}-}R$ the full triangulated subcategory of $\operatorname{dgDer-}R$ generated by the free module $R$. Localization induces an equivalence of categories $$\operatorname{dgFree-}R\stackrel{\sim}{\rightarrow}\operatorname{dgFree^{Der}-}R.$$
	We denote $X=\HomC{\mathcal B}{U}{S}$. This is a $E_S$-$E_U$-dg-bimodule. We denote by $Z_X$ and $H_X$ its  degree 0 trunctation and 0-th cohomology, respectively. Since $U$ is homotopy-projective by assumption, the cohomology of $X$ is concentrated in degree 0 and $$H_X=\Hom{\Hot\mathcal B}{U}{S}=\Hom{\Der\mathcal B}{U}{S}\mathop{\leftarrow}_{\sim}^{c_*}\Hom{\Der\mathcal B}{U}{U}=H_U$$ is freely generated by $\left[c\right]$ as right (dg-)module over $H_U$. Furthermore, the maps
	$$H_X\twoheadleftarrow Z_X\hookrightarrow X$$ 
	are quasi-isomorphisms of $Z_S$-$Z_U$-dg-bimodules.
	\begin{enumerate}[resume]
		\item The left vertical arrow is defined via 
		$$\langle S \rangle^{\Hot(\mathcal B)}_\oplus\rightarrow \langle S \rangle^{\Der(\mathcal B)}_\oplus=\langle U \rangle^{\Der(\mathcal B)}_\oplus.$$
		The right vertical arrow restricts to $\operatorname{dgFree^{Der}-}H_U$ since $H_X$ is free. The natural isomorphism is defined similarly as in (1) by 
		\begin{align*}
		\Hom{\Hot(\mathcal B)}{S}{-}\otimes_{H_S}H_X\stackrel{\operatorname{comp}}{\rightarrow} &\Hom{\Hot(\mathcal B)}{U}{-}\\
		=&\Hom{\Der(\mathcal B)}{U}{-}.
		\end{align*}
		\item The right vertical and bottom horizontal arrow restrict to $\operatorname{dgFree^{Der}-}Z_U$, since $H_U$ and $Z_X$ are quasi-isomorphic to $Z_U$ as $Z_U$-dg-modules. It is easy to see that the following diagram commutes up to natural isomorphism.
		\begin{center}
			\begin{tikzcd}[column sep=2cm,row sep=1cm]
				\operatorname{dgFree-}H_S
				\arrow[d,"-\otimes_{H_S}H_X"']&
				\operatorname{dgFree-}Z_S
				\arrow[l,"-\otimes_{Z_S}H_S"']\arrow[d,"-\otimes_{Z_S}Z_X"']
				\\
				\langle H_X \rangle^{\operatorname{dgHot-}H_U}_\Delta
				\arrow[d,"\wr"]&
				\langle Z_X \rangle^{\operatorname{dgHot-}Z_U}_\Delta
				\arrow[l,"-\otimes_{Z_S}H_S"']\arrow[d,"\wr"]
				\\
				\operatorname{dgFree^{Der}-}H_U
				\arrow[r,"\res_{Z_U}^{H_U}"]
				&
				\operatorname{dgFree^{Der}-}Z_U
			\end{tikzcd}
		\end{center}
		Here we use $\res_{Z_U}^{H_U}$, which goes in the wrong direction, to avoid using its inverse functor, which is the derived functor $-\otimes^{\mathbb L}_{Z_U}H_U$.
		\item As in the last point.
		\item Similar to (5).
	\end{enumerate}
	The squares in the bottom are easily filled with natural isomorphisms and we are finished with this proof.
\end{proof}
\begin{remark}\label{rem:dualtiltingandfunctors}
	The completely dual statements hold when we instead require homotopy-injective resolutions $c:F(T)\rightarrow U$ and use the tilting equivalence discussed in Remark \ref{rem:dualtilting}.
\end{remark}

\bibliographystyle{amsalpha} 
\bibliography{main}

\def\cprime{$'$}
\providecommand{\bysame}{\leavevmode\hbox to3em{\hrulefill}\thinspace}
\providecommand{\MR}{\relax\ifhmode\unskip\space\fi MR }
\providecommand{\MRhref}[2]{%
  \href{http://www.ams.org/mathscinet-getitem?mr=#1}{#2}
}
\providecommand{\href}[2]{#2}
\begin{thebibliography}{BGG73}

\bibitem[AL03]{andersen2003twisted}
Henning~Haahr Andersen and Niels Lauritzen, \emph{Twisted verma modules},
  Studies in memory of Issai Schur, Springer, 2003, pp.~1--26.

\bibitem[Bei87]{beilinson1987derived}
AA~Beilinson, \emph{On the derived category of perverse sheaves}, K-theory,
  arithmetic and geometry, Springer, 1987, pp.~27--41.

\bibitem[BG80]{BGproj}
J.~N. Bernstein and S.~I. Gelfand, \emph{Tensor products of finite- and
  infinite-dimensional representations of semisimple {L}ie algebras},
  Compositio Math. \textbf{41} (1980), no.~2, 245--285. \MR{581584}

\bibitem[BG86]{BG}
A.~A. Beilison and V.~A. Ginzburg, \emph{Mixed categories, {E}xt-duality and
  representations (results and conjectures)}, unpublished (1986).

\bibitem[BGG71]{BGG}
J.~N. Bernstein, I.~M. Gelfand, and S.~I. Gelfand, \emph{Structure of
  representations that are generated by vectors of highest weight},
  Funckcional. Anal. i Prilo\v zen. \textbf{5} (1971), no.~1, 1--9. \MR{0291204
  (45 \#298)}

\bibitem[BGG73]{bernstein1973schubert}
\bysame, \emph{Schubert cells and cohomology of the spaces g/p}, Russian
  Mathematical Surveys \textbf{28} (1973), no.~3, 1.

\bibitem[BGS96]{BGS}
Alexander Beilinson, Victor Ginzburg, and Wolfgang Soergel, \emph{Koszul
  duality patterns in representation theory}, J. Amer. Math. Soc. \textbf{9}
  (1996), no.~2, 473--527. \MR{1322847 (96k:17010)}

\bibitem[Bon10]{Bon}
Mikhail~V. Bondarko, \emph{Weight structures vs. {$t$}-structures; weight
  filtrations, spectral sequences, and complexes (for motives and in general)},
  J. K-Theory \textbf{6} (2010), no.~3, 387--504. \MR{2746283}

\bibitem[Bou02]{Bou}
Nicolas Bourbaki, \emph{Lie groups and {L}ie algebras. {C}hapters 4--6},
  Elements of Mathematics (Berlin), Springer-Verlag, Berlin, 2002, Translated
  from the 1968 French original by Andrew Pressley. \MR{1890629}

\bibitem[BT72]{BorTits}
Armand Borel and Jacques Tits, \emph{Compl\'ements \`a l'article: ``{G}roupes
  r\'eductifs''}, Inst. Hautes \'Etudes Sci. Publ. Math. (1972), no.~41,
  253--276. \MR{0315007}

\bibitem[CD12]{CD}
Denis-Charles Cisinski and Fr{\'e}d{\'e}ric D{\'e}glise, \emph{{Triangulated
  categories of mixed motives}}, Arxiv preprint arXiv:0912.2110 (2012).

\bibitem[Dre15]{brad}
Brad Drew, \emph{Rectification of deligne's mixed hodge structures}, 2015.

\bibitem[Ebe17]{phdthesis}
Jens~Niklas Eberhardt, \emph{Graded and geometric parabolic induction}, Ph.D.
  thesis, Albert-Ludwigs-Universität Freiburg, 2017.

\bibitem[EK16]{EKe}
Jens~Niklas Eberhardt and Shane Kelly, \emph{Mixed motives and geometric
  representation theory in equal characteristic}, 2016.

\bibitem[Gin91]{ginsburg1991perverse}
Victor Ginsburg, \emph{Perverse sheaves and {$\C^*$}-actions}, Journal of the
  American Mathematical Society (1991), 483--490.

\bibitem[GM88]{GM88}
Mark Goresky and Robert MacPherson, \emph{Stratified {M}orse theory},
  Ergebnisse der Mathematik und ihrer Grenzgebiete (3) [Results in Mathematics
  and Related Areas (3)], vol.~14, Springer-Verlag, Berlin, 1988. \MR{932724}

\bibitem[He07]{Xe}
Xuhua He, \emph{Minimal length elements in some double cosets of coxeter
  groups}, Advances in Mathematics \textbf{215} (2007), no.~2, 469 -- 503.

\bibitem[How00]{Ho98}
Roger~E. Howe, \emph{Harish-{C}handra homomorphisms}, The mathematical legacy
  of {H}arish-{C}handra ({B}altimore, {MD}, 1998), Proc. Sympos. Pure Math.,
  vol.~68, Amer. Math. Soc., Providence, RI, 2000, pp.~321--332. \MR{1767901}

\bibitem[Hum90]{HumCox}
James~E. Humphreys, \emph{Reflection groups and {C}oxeter groups}, Cambridge
  Studies in Advanced Mathematics, vol.~29, Cambridge University Press,
  Cambridge, 1990. \MR{1066460}

\bibitem[Hum98]{Hum}
J.~Humphreys, \emph{Linear algebraic groups}, second ed., Graduate Texts in
  Mathematics, vol.~21, Springer, 1998.

\bibitem[Hum08]{HumCatO}
James~E. Humphreys, \emph{Representations of semisimple {L}ie algebras in the
  {BGG} category {$\mathcal O$}}, Graduate Studies in Mathematics, vol.~94,
  American Mathematical Society, Providence, RI, 2008. \MR{2428237
  (2009f:17013)}

\bibitem[Irv93]{irving1993shuffled}
Ronald~S Irving, \emph{Shuffled verma modules and principal series modules over
  complex semisimple lie algebras}, Journal of the London Mathematical Society
  \textbf{2} (1993), no.~2, 263--277.

\bibitem[Jan79]{jantzen1979moduln}
J.C. Jantzen, \emph{Moduln mit einem h{\"o}chsten {G}ewicht}, Lecture Notes in
  Mathematics, Springer, 1979.

\bibitem[Kel93]{Keller1993}
Bernhard Keller, \emph{A remark on tilting theory and dg algebras}, manuscripta
  mathematica \textbf{79} (1993), no.~1, 247--252.

\bibitem[Kel94]{keller1994deriving}
\bysame, \emph{Deriving dg categories}, Annales scientifiques de l'Ecole
  normale sup{\'e}rieure, vol.~27, 1994, pp.~63--102.

\bibitem[Lur]{lur}
Jacob Lurie, \emph{Higher algebra}.

\bibitem[Ric89]{rickard}
Jeremy Rickard, \emph{Morita theory for derived categories}, J. London Math.
  Soc. (2) \textbf{39} (1989), no.~3, 436--456. \MR{1002456}

\bibitem[Soe90]{Soe90}
Wolfgang Soergel, \emph{Kategorie {$\mathcal O$}, perverse {G}arben und
  {M}oduln \"uber den {K}oinvarianten zur {W}eylgruppe}, J. Amer. Math. Soc.
  \textbf{3} (1990), no.~2, 421--445. \MR{1029692 (91e:17007)}

\bibitem[SS15]{Sartori2015256}
Antonio Sartori and Catharina Stroppel, \emph{Categorification of tensor
  product representations of {$\mathfrak{sl}_k$} and category {$\mathcal O$}},
  Journal of Algebra \textbf{428} (2015), 256 -- 291.

\bibitem[Str03]{Str}
Catharina Stroppel, \emph{Category {$\mathcal O$}: gradings and translation
  functors}, Journal of Algebra \textbf{268} (2003), no.~1, 301--326.

\bibitem[SVW]{SoeWeVir}
Wolfgang Soergel, Rahbar Virk, and Matthias Wendt, \emph{Equivariant motives
  and representation theory}, (To Appear).

\bibitem[SW16]{SoeWe}
Wolfgang Soergel and Matthias Wendt, \emph{Perverse motives and graded derived
  category {$\mathcal O$}}, Journal of the Institute of Mathematics of Jussieu
  (2016), 1--49.

\bibitem[Wat60]{watts1960intrinsic}
Charles~E Watts, \emph{Intrinsic characterizations of some additive functors},
  Proceedings of the American Mathematical Society \textbf{11} (1960), no.~1,
  5--8.

\end{thebibliography}

\end{document}